\documentclass[11pt]{article}
\usepackage{amsmath}[1996/11/01]
\usepackage{amssymb,amsthm,amsxtra}
\usepackage{amscd}
\usepackage[margin=1in]{geometry}
\usepackage{tikz}
\usetikzlibrary{calc}
\usetikzlibrary{positioning}
\usetikzlibrary{patterns}

\def\[#1\]{\begin{equation}#1\end{equation}}
\makeatletter
\def\beq{%
   \relax\ifmmode
      \@badmath
   \else
      \ifvmode
         \nointerlineskip
         \makebox[.6\linewidth]%
      \fi
      $$
   \fi
}
\def\eeq{%
   \relax\ifmmode
      \ifinner
         \@badmath
      \else
         $$
      \fi
   \else
      \@badmath
   \fi
   \ignorespaces
}

\def\enddisplaymath{\eeq\global\@ignoretrue}
\makeatother

\newtheorem{thm}{Theorem}
\newtheorem{cor}[thm]{Corollary}
\newtheorem{lem}[thm]{Lemma}
\newtheorem{prop}[thm]{Proposition}

\theoremstyle{remark}
\newtheorem*{rem}{Remark}
\newtheorem{rems}{Remark}[thm]

\newtheorem{eg}{Example}
\theoremstyle{definition}
\newtheorem{defn}{Definition}

\numberwithin{equation}{section}
\numberwithin{thm}{section}
\numberwithin{eg}{section}

\newcommand{\F}{\mathbb F}
\renewcommand{\P}{\mathbb P}
\newcommand{\Q}{\mathbb Q}
\newcommand{\Z}{\mathbb Z}
\newcommand{\G}{\mathbb G}
\newcommand{\C}{\mathbb C}
\newcommand{\N}{\mathbb N}
\newcommand{\A}{\mathbb A}

\DeclareMathOperator{\GL}{GL}
\DeclareMathOperator{\PGL}{PGL}

\DeclareMathOperator{\Pic}{Pic}
\DeclareMathOperator{\Hilb}{Hilb}

\DeclareMathOperator{\Ext}{Ext}
\DeclareMathOperator{\Sym}{Sym}

\DeclareMathOperator{\Aut}{Aut}
\DeclareMathOperator{\Isom}{Isom}
\DeclareMathOperator{\Hom}{Hom}
\DeclareMathOperator{\Mat}{Mat}
\DeclareMathOperator{\rank}{rnk}
\DeclareMathOperator{\coker}{coker}
\DeclareMathOperator{\supp}{supp}
\DeclareMathOperator{\ord}{ord}

\DeclareMathOperator{\Proj}{Proj}
\DeclareMathOperator{\Spec}{Spec}
\DeclareMathOperator{\im}{im}
\DeclareMathOperator{\ch}{char}
\let\div\relax
\DeclareMathOperator{\div}{div}
\newcommand{\sO}{\mathcal O}
\newcommand{\sExt}{\mathcal Ext}
\newcommand{\sHom}{\mathcal Hom}

\newcommand{\Irr}{\mathcal Irr}
\newcommand{\Tor}{\mathcal Tor}

\DeclareMathOperator{\Ai}{Ai}
\DeclareMathOperator{\Gal}{Gal}
\newcommand{\dR}{{\bf R}}
\newcommand{\dL}{{\bf L}}

\begin{document}

\title{Generalized Hitchin systems on rational surfaces}
\author{
Eric M. Rains\\Department of Mathematics, California
  Institute of Technology}

\date{July 25, 2019}
\maketitle

\begin{abstract}
By analogy with work of Hitchin on integrable systems, we construct natural
relaxations of several kinds of moduli spaces of difference equations, with
special attention to a particular class of difference equations on an
elliptic curve (arising in the theory of elliptic special functions).  The
common feature of the relaxations is that they can be identified with
moduli spaces of sheaves on rational surfaces.  Not only does this make
various natural questions become purely geometric (rigid equations
correspond to $-2$-curves), it also establishes a number of nontrivial
correspondences between different moduli spaces, since a given moduli space
of sheaves is typically the relaxation of infinitely many moduli spaces of
equations.  In the process of understanding this, we also consider a number
of purely geometric questions about rational surfaces with anticanonical
curves; e.g., we give an essentially combinatorial algorithm for testing
whether a given divisor is the class of a $-2$-curve or is effective with
generically integral representative.

\end{abstract}

\tableofcontents

\section{Introduction}

One of the more striking properties of hypergeometric functions is that
they are the solutions of a {\em rigid} differential equation, i.e., one
which is determined by its order and its singularities.  For instance, it
was already observed by Riemann that any second-order Fuchsian differential
equation with exactly three singular points reduces by a simple change of
variables to the equation satisfied by a hypergeometric function of type
${}_2F_1$.  More generally, the equation satisfied by a hypergeometric
function of type ${}_rF_{r-1}$ is
rigid for any $r$; it is $r$-th order, with Fuchsian singularities at $0$,
$1$, and $\infty$, and is uniquely determined by the exponents at $0$ and
$1$, together with a more significant constraint on the singularity at
$\infty$.  This suggests that one should study rigid equations more
generally; for an exploration of this from the monodromy perspective, see
\cite{KatzNM:1996}.  In addition, the modern theory of Painlev\'e
transcendents leads us to consider what happens for non-rigid equations;
for instance, the Painlev\'e VI equation can be interpreted as a flow in a
2-dimensional moduli space of second order Fuchsian equations with four
singular points, with specified exponents at the singular points.

If we extend our focus to include $q$-hypergeometric functions, then we see
that we must consider more than just differential equations: Gauss'
differential equation becomes a $q$-difference equation when extended to
$q$-hypergeometric functions.  Even more generally, we could consider {\em
  elliptic} hypergeometric functions\footnote{Very roughly speaking, these are
series in which the usual hypergeometric constraint ``ratios of consecutive
terms are rational functions of the index'' is replaced by ``ratios of
consecutive terms are {\em elliptic} functions of the index''; we will not
need details, as we are considering these for motivation only.}, and thus
elliptic difference equations.

If we look carefully at the equations satisfied by most of the known
elliptic hypergeometric functions \cite{dets}, or the equations
satisfied by semiclassical elliptic biorthogonal functions
\cite{isomonodromy}, we find that there is an important additional
structure.  A general elliptic difference equation has the form
\[
v(z+q) = A(z) v(z)
\]
where $q$ is a point of an elliptic curve and $A(z)$ is a matrix of
elliptic functions (with $\det A(z)$ not identically 0); in the cases
of interest, though, we have the additional constraint
\[
A(-q-z) = A(z)^{-1}.
\]
This is more natural than it may appear.  We can view a difference equation
as a $1$-cocycle for the group $\Z$ acting on $\GL_n(k(E))$ via translation
by $q$, and cohomologous $1$-cocycles are simply related by
gauge transformations (called isomonodromy transformations in \cite{isomonodromy}, as they preserve a suitable notion of monodromy)
\[
A(z) \mapsto C(z+q) A(z) C(z)^{-1};
\]
i.e., the equation satisfied by $C(z)v(z)$ is cohomologous to the equation
satisfied by $v(z)$.  In the cases with symmetric equations, the solutions
of interest are symmetrical: they satisfy the additional constraint
$v(-z)=v(z)$.  Now, the {\em pair} of equations
\[
v(z+q) = A(z) v(z),\qquad v(-z) = v(z)
\]
can also be viewed as an object in nonabelian cohomology, namely a
$1$-cochain for the infinite dihedral group.  The constraint on $A$ simply
says that this cochain is a cocycle, basically a formal self-consistency
condition.  Indeed, a symmetrical solution of the difference equation
satisfies
\[
v(-q-z) = A(z) v(z)
\]
and thus
\[
v(z) = A(-q-z)v(-q-z) = A(-q-z)A(z) v(z).
\]
If $A(-q-z)\ne A(z)^{-1}$, then we we will have ``too few'' symmetric
solutions.  (Note that two cocycles for the infinite dihedral group are
cohomologous iff they are related by a gauge transformation with
$C(-z)=C(z)$.)

We are thus led to the following natural question: What are the {\em
  symmetric} elliptic difference equations which are rigid?  In
particular, we would expect, and will see, that the elliptic hypergeometric
equations are indeed rigid, and the equations related to the elliptic
Painlev\'e equation fit into a $2$-dimensional moduli space.  This, of
course, degenerates to corresponding questions about symmetric ordinary and
$q$-difference equations, and at those levels can also be degenerated to
questions about ordinary and $q$-difference equations without symmetry, not
to mention differential equations on $\P^1$.

In the case of differential equations, there is a useful relaxation of the
problem due to Hitchin \cite{HitchinN:1987}.  Rather than consider
differential equations themselves, i.e., connections on vector bundles over
$\P^1$, one considers $1$-form valued endomorphisms $V\to V\otimes \Omega$.
When $V\cong \sO_{\P^1}^n$ is a trivial bundle, these notions are
essentially the same: a $1$-form valued endomorphism $A$ corresponds to a
differential equation $dv = A v$.  While the notions
diverge for nontrivial bundles, we can still expect that a large open
subset of the two moduli spaces should coincide.  There is a corresponding
relaxation for difference equations \cite{HurtubiseJC/MarkmanE:2002a};
again, we classify matrices rather than difference equations, but the two
moduli problems are closely related.

A significant advantage of the relaxation over the original problem is that
it reduces to a moduli problem about sheaves on a smooth projective
surface.  The relaxed version of a differential equation on a smooth curve
$C$ corresponds to a sheaf on the ruled surface $\P(\sO_C\oplus \omega_C)$,
while the relaxed version of an elliptic difference equation (without
symmetry) corresponds to a sheaf on $E\times \P^1$ (in either case, this is
generically a line bundle on a certain spectral curve).  The information
about singularities translates to a specification of how said sheaf meets
a certain anticanonical curve on the surface (the zero locus of a Poisson
structure).

In the case of differential equations on $\P^1$, the surface in question is
the Hirzebruch surface $F_2\cong \P(\sO_{\P^1}\oplus \sO_{\P^1}(-2))$, and
the relevant anticanonical curve has the form $2S$ where $S$ is a section
disjoint from the $-2$ curve on $F_2$.  In contrast to the case of ruled
surfaces of higher genus, this anticanonical curve is extremely special,
and one might thus wonder whether there is a natural interpretation for the
moduli spaces associated to more general anticanonical curves on $F_2$.
The most general anticanonical curve on $F_2$ is in fact a smooth curve of
genus 1 (more precisely, a {\em hyperelliptic} curve of genus 1; specifying
an embedding in $F_2$ is equivalent to specifying a degree 2 map to
$\P^1$), and as we will see below, the corresponding moduli problem is a
relaxation of the moduli problem of symmetric elliptic difference
equations.

The fact that symmetric elliptic difference equations correspond to sheaves
on a {\em rational} surface appears to be at the core of why they appear in
special function theory.  Indeed, as we noted in \cite{poisson}, rigid
sheaves in the relaxation can only exist on a rational surface (even the
trivial equation $v(z+q)=v(z)$ fails to be rigid as a nonsymmetric elliptic
difference equation!).  In addition, the birational maps between irrational
ruled surfaces are extremely simple (between minimal surfaces, all
birational maps are compositions of elementary transformations), while
rational surfaces have a rich structure coming from birational maps.  As we
will see, this means that any given sheaf actually corresponds to a large
(in some cases infinite) set of inequivalent equations; in higher genus
cases, all we can do is multiply $v$ by the solution of a first-order
equation.

The purpose of the present note is to explore these relaxations, and in
particular the additional structure afforded by the fact that they live on
a rational surface.  We first show how to translate a symmetric elliptic
difference equation into a sheaf on a rational surface (including mild
generalizations where we twist by a line bundle), and discuss how this
degenerates to ordinary and $q$-difference cases.  Note that there is a
somewhat subtle issue here, in that what it means to be singular changes
slightly if we forget the symmetry of the equation (consider the equation
$v(z+q)=-v(z)$, which has no {\em symmetric} solutions which are
holomorphic and nonzero near $z=-q/2$).  This is one reason why we should
indeed think of symmetric elliptic equations as more than just a special
case of elliptic equations.  We also consider a few other moduli problems
that also reduce to questions about sheaves on rational surfaces.

Given this translation, considerations of \cite{poisson} reduce questions
of rigidity to much simpler questions in algebraic geometry.  Indeed, the
sheaves corresponding to relaxations of rigid difference/differential
equations are just direct images of line bundles on $-2$ curves on suitable
blowups of the original ambient surface (specifically, $-2$ curves which
are disjoint from the anticanonical curve).  Similarly, the $2$-dimensional
moduli spaces (there is an induced symplectic structure, so all moduli
spaces here are even-dimensional) are related to (quasi-)elliptic pencils.

In this way, our questions about moduli problems of difference equations
translate to structural questions about rational surfaces with an
anticanonical curve: what are the $-2$ curves, and which divisor classes
have integral representatives disjoint from the anticanonical curve?  And,
of course, what are the different ways of blowing a given rational surface
down to a Hirzebruch surface?  Earlier work on blowups of $\P^2$ leads us
to a certain family of Coxeter groups, which almost acts on the set of ways
of blowing down; each simple reflection acts unless a corresponding divisor
class is effective.  Using this action, we obtain algorithms for
determining (a) whether a given divisor class is the class of a $-2$ curve
(i.e., whether the corresponding sheaves represent rigid equations), and
(b) whether a given divisor class is effective (or nef, or integral).

The one major drawback of the relaxation is that the various natural
transformations of sheaves (changing the blowdown, twisting by a line
bundle on a blowup) will almost always act in the wrong way from the
difference equation perspective.  (E.g., twisting has the effect of
conjugating the matrix $A$ by a suitable rational matrix, and this needs to
be replaced by a suitable $q$-deformed conjugation.)  Since the relaxation
lives on a {\em Poisson} rational surface, it is natural to conjecture that
the original problem should correspond to sheaves on a {\em noncommutative}
rational surface.  This is bolstered by recent work \cite{P2Painleve}
showing that one can obtain the elliptic Painlev\'e equation as a
Hitchin-type system on a noncommutative $\P^2$.  In a future paper
\cite{noncomm1}, we will show how to use elliptic difference operators to
construct a suitable family of noncommutative rational surfaces, and extend
this to general ruled surfaces in \cite{noncomm2}; this will require some
additional facts about commutative rational surfaces which we establish
here.  In particular, our noncommutative rational surfaces will be
constructed via certain flat families of difference operators, and it is
already a nontrivial fact, established below, that the corresponding spaces
are flat in the commutative setting.

Related to this, we also consider some general questions about the moduli
space (stack) of anticanonical rational surfaces.  In particular, this
moduli stack naturally splits as a union of locally closed substacks based
on the structure of the anticanonical curve.  This leads to the question of
how this different pieces are related, specifically how their closures
intersect.  This appears to be a rather hard problem in general; we give an
easy necessary condition for one such substack to be contained in the
closure of another, as well as a much more subtle necessary condition,
which leads to some pathologies in small characteristic.  In particular, we
give a corrected version of the diagram of degenerations of surfaces with
$K^2=0$ (corresponding to Sakai's hierarchy \cite{SakaiH:2001} of discrete
and continuous Painlev\'e equations).  We also discuss in detail how the
structure of the anticanonical curve relates to the structure of the
corresponding difference/differential equations.

We will then conclude with a couple of sections discussing the implications
of these results for symmetric elliptic difference equations (including
what most of the natural operations do both in the relaxed and in the
nonrelaxed versions), as well as certain degenerate cases.  The latter
include natural birational maps between spaces of symmetric $q$-difference
equations and spaces of nonsymmetric $q$-difference equations, as well as
maps between such equations and solutions of the ``multiplicative
Deligne-Simpson problem''.  This includes settling a conjecture of
\cite{EOR}, as a special case of a theorem identifying the Jacobian of a
rational elliptic surface (Theorem \ref{thm:Picr} below).  We also briefly
consider some deformations generalizing certain Calogero-Moser spaces.

{\bf Acknowledgements}. The author would like to thank D. Arinkin,
A. Borodin, P. Etingof, T. Graber, A. Knutson, and A. Okounkov for helpful
conversations, as well as N. Joshi for some \TeX\ assistance.  This work was
partially supported by grants from the National Science Foundation,
DMS-1001645 and DMS-1500806.

\section{Sheaves from difference equations}

As we discussed in the introduction, the analogue of a differential equation
at the top (elliptic) level in the hierarchy of special functions is a {\em
  symmetric elliptic difference equation}, which we should think of as the
pair of equations
\[
v(z+q) = A(z) v(z),
\qquad
v(-z) = v(z),
\]
where $A$ is a matrix of elliptic functions subject to the consistency
condition $A(-q-z)A(z)=1$.  The natural relaxation of this problem is to
forget the difference equation, and simply classify matrices $A$ of
elliptic functions satisfying $A(-q-z)A(z)=1$.  (We will also want to take
into account singularities, but will table that question for the moment.)

Since we plan to relate this to an algebraic geometric object, it will be
helpful to rephrase this original problem in a somewhat more abstractly
geometric way.  Thus we suppose given a smooth genus 1 curve $C_\alpha$ over
an algebraically closed field $k$ (not necessarily of characteristic 0),
along with a translation $\tau_q:C_\alpha\to C_\alpha$ and a hyperelliptic
involution $\eta:C_\alpha\to C_\alpha$, i.e., such that the quotient of
$C_\alpha$ by the involution is isomorphic to $\P^1$.  In the analytic
setting, $C_\alpha$ is $\C/\Lambda$ for some lattice $\Lambda$, $\tau_q$ is
the map $z\mapsto z+q$, and $\eta$ is the map $z\mapsto -q-z$.  We take the
latter choice for $\eta$ so that the problem of classifying $A$ becomes the
following: Classify matrices $A\in \GL_n(k(C_\alpha))$ such that $\eta^*A =
A^{-1}$.

Just as the original problem can be rephrased in terms of $1$-cocycles of
the infinite dihedral group on $\GL_n(k(C_\alpha))$, this question is
itself related to nonabelian cohomology: a matrix $A$ such that
$\eta^*A=A^{-1}$ specifies a $1$-cocycle for the cyclic group
$\langle\eta\rangle$ (of order 2).  Now, the action of $\eta$ allows us to
think of $k(C_\alpha)$ as a Galois extension of the invariant subfield
$k(\P^1)$, and thus we find
\[
H^1(\langle \eta\rangle;\GL_n(k(C_\alpha))) = H^1(\Gal(k(C_\alpha)/k(\P^1)),\GL_n).
\]
It is a classical fact that the latter Galois cohomology set is trivial,
and this translates to the following fact (often referred to as Hilbert's
Theorem 90, though Hilbert only considered the case of a cyclic Galois
group acting on $\GL_1$).

\begin{prop}
Let $L/K$ be a quadratic field extension, and let $A\in \GL_n(L)$ be a
matrix such that $\bar{A}=A^{-1}$, where $\bar\cdot$ is the
conjugation of $L$ over $K$.  Then there exists a matrix $B\in \GL_n(L)$
such that $A = \bar{B} B^{-1}$, and $B$ is unique up to
right-multiplication by $\GL_n(K)$.
\end{prop}

\begin{proof}
  In this case, the argument is particularly simple.  Given any vector
  $w\in L^n$, the vector $v=\bar{w}+A^{-1}w$ satisfies
  $\bar{v}=Av$.  If we apply this to a basis of $L^n$ over $K$, we
  obtain in this way at least $n$ vectors satisfying $\bar{v}=Av$
  which are linearly independent over $K$.  It follows that there exists a
  matrix $B\in \GL_n(L)$ such that $\bar{B} = A B$, which is what we
  want.  If $B'$ is another such matrix, then
\[
\overline{B^{-1}B'} = B^{-1} A^{-1} A B' = B^{-1}B',
\]
and thus $B^{-1}B'\in \GL_n(K)$ as required.
\end{proof}

In our setting, it will turn out to be appropriate to make the
factorization have the form $A = \eta^* B^{-t} B^t$.  (In the
noncommutative setting, the most natural correspondence between difference
equations and sheaves is contravariant and holomorphic in $B$.)  The
nonuniqueness (we can still multiply $B$ on the right by any element of
$\GL_n$) is of course still an issue, but it turns out there is a slight
modification which can be made unique.  The first step is to make the
nonuniqueness problem worse by allowing $B$ to be a map between vector
bundles.  Let $\pi_\eta:C_\alpha\to \P^1$ be the morphism quotienting by
the action of $\eta$. Then for any vector bundle $V$ on $\P^1$, and any
meromorphic (and generically invertible) map
\[
B:\pi_\eta^*V\to \sO_{C_\alpha}^n,
\]
we obtain a well-defined matrix $\eta^* B^{-t}B^t$, and of course any
matrix with $\eta^*A = A^{-1}$ can be represented in this way (just take
$V$ to be $\sO_{\P^1}^n$\dots).  (Here, by the transpose $B^t$, we mean the
image of $B$ under the functor $\sHom_{C_\alpha}(-,\sO_{C_\alpha})$.)  The advantage of
allowing $V$ to be a more general vector bundle is that we can then insist
that $B$ be {\em holomorphic} (and thus injective), by absorbing any poles
into $V$.  This is still non-unique, since we could freely replace $V$ by
any vector bundle it contains, and still obtain an injective morphism
supporting a factorization of $A$.  However, we can make this unique by
imposing a maximality condition on $V$.

\begin{prop}
  Suppose $B:\pi_\eta^*V_0\to k(C_\alpha)^n$ is an injective map of sheaves, with
  $V_0$ a rank $n$ vector bundle on $\P^1$.  This induces an isomorphism
  $B:\pi_\eta^*(V_0\otimes_{\sO_{\P^1}} k(\P^1))\cong k(C_\alpha)^n$, and the set of
  bundles $V\subset V_0\otimes_{\sO_{\P^1}} k(\P^1)$ such that $BV\subset
  \sO_{C_\alpha}^n$ is nonempty, with a unique maximal element.
\end{prop}

\begin{proof} That the induced map of vector spaces over $k(C_\alpha)$ is an
  isomorphism follows from the fact that it is an injective map of vector
  spaces of the same dimension.  That the set of bundles $V$ is nonempty is
  straightforward, as we have already mentioned (just absorb any poles of
  $B$ into $V$).  Finally, if $V_1$, $V_2$ are vector bundles contained in
  $V_0\otimes_{\sO_{\P^1}}k(\P^1)$ such that $BV_1, BV_2\subset \sO_{C_\alpha}^n$,
  then $V_1+V_2$ is still contained in $V_0\otimes_{\sO_{\P^1}}k(\P^1)$, so
  is torsion-free, thus a vector bundle; and $B(V_1+V_2)=BV_1+BV_2\subset
  \sO_{C_\alpha}^n$.  Since $BV\subset \sO_{C_\alpha}^n$ implies $\deg(\pi_\eta^*V)=\deg(V)\le
  0$, it follows that there is a unique maximal such bundle.
\end{proof}

To summarize the above considerations, given any matrix $A\in
\GL_n(k(C_\alpha)))$ such that $\eta^*A=A^{-1}$, there is a canonical
factorization $A=\eta^*B^{-t}B^t$ where $B$ is an injective morphism
\[
B:\pi_\eta^*V\to \sO_{C_\alpha}^n
\]
with $V$ a rank $n$ vector bundle on $\P^1$, maximal among those supporting
a map $B$.  This canonical factorization also clarifies issues regarding
singularities.  For instance, as we mentioned in the introduction, the
equation $v(z+q)=-v(z)$ is singular at $-q/2$ as a symmetric equation,
since there are no symmetric solutions which are nonzero and holomorphic at
$-q/2$.  This is nonobvious in terms of $A$, since $A=-1$ has no zeros or
poles here, but becomes clear in terms of $B$, as we find that $B$ must in
fact vanish at every point of the form $-q/2$ (more precisely, at every
fixed point of $\eta$).  We find in general that the points where
$\tau_q^*v=Av$ is singular as a symmetric equation are precisely those
points where $\det(B)=0$.  More generally, the right way to classify
singularities of (symmetric) elliptic difference equations is to consider
the induced cocycle over the ring of ad\`eles; one can show that the
classes of such cocycles are determined by the corresponding elementary
divisors of $B$.

\begin{rem}
  A similar factorization appeared in \cite{isomonodromy}, but the reader
  should be cautioned that they are not quite the same; indeed, the
  factorization of \cite{isomonodromy} involves a partition of the
  singularities in to two subsets, and depends significantly on that
  choice.  It turns out that those matrices (up to transpose) correspond to
  canonical factorizations of cohomologous equations, see Section
  \ref{sec:elldiff}.
\end{rem}

Since $B$ is a map from a pullback, we can use adjunction to relate it to
a map to a direct image: specifying $B$ is equivalent to specifying
\[
\pi_{\eta*}B:V\to \pi_{\eta*}\sO_{C_\alpha}^n.
\]
With this in mind, we can obtain a natural extension of $B$ to a surface
containing $C_\alpha$ (as an anticanonical curve).  Indeed, since $\pi_\eta$
has degree 2, the direct image $\pi_{\eta*}\sO_{C_\alpha}$ is a vector bundle
of degree $2$, and thus we can take the corresponding projective bundle to
obtain a Hirzebruch surface $X = \P(\pi_{\eta*}\sO_{C_\alpha})$.  Note that
$X\cong F_2$, since $\pi_{\eta*}\sO_{C_\alpha}\cong \sO_{\P^1}\oplus
\sO_{\P^1}(-2)$.  Moreover, $X$ contains $C_\alpha$ in a natural way, in such
a way that the induced map from $C_\alpha$ to $\P^1$ is just $\pi_\eta$, and
$C_\alpha$ is anticanonical.  If $\rho:X\to \P^1$ is the corresponding
ruling, and $s_{\min}$ denotes the section of the ruling with minimal
self-intersection ($s_{\min}^2=-2$), then we have a canonical isomorphism
\[
\pi_{\eta*} \sO_{C_\alpha} \cong \rho_*{\cal L} (s_{\min}),
\]
since $\sO_X(s_{\min})$ is the relative $\sO(1)$.  This is just the direct
image under $\rho$ of the restriction map
\[
\sO_X(s_{\min})\to \sO_X(s_{\min})|_{C_\alpha}\cong \sO_{C_\alpha},
\]
where we note that $C_\alpha$ and $s_{\min}$ are disjoint, and we make the
isomorphism canonical by taking the unique global section of
$\sO_X(s_{\min})$ to the unique global section of $\sO_{C_\alpha}$.

In other words, to specify $B:\pi_\eta^*V\to \sO_{C_\alpha}^n$, it is
equivalent to specify its direct image
\[
\rho_*B:V\to \rho_*\sO_X(s_{\min})^n,
\]
where we now think of $B$ as a morphism of sheaves on $C_\alpha\subset X$.
Again using the adjunction between $\rho_*$ and $\rho^*$ gives us a morphism
\[
B:\rho^*V\to \sO_X(s_{\min})^n,
\]
which when restricted to $C_\alpha\subset X$ recovers the original morphism.
Again, we modify this slightly to
\[
B:\rho^*V\otimes \sO_X(-s_{\min})\to \sO_X^n,
\]
which has no effect on the restriction to $C_\alpha$, but is more natural in the
noncommutative setting (and slightly more natural even in the commutative
setting).  In any event, we now have a morphism of vector bundles on the
Hirzebruch surface $X$.  This in turn translates to a questions about
sheaves, via the following result.

\begin{prop}
  Let $\rho:X\to C$ be a ruled surface, with relative $\sO(1)$ denoted by
  $\sO_\rho(1)$, and let $M$ be a coherent sheaf on $X$.  Then the
  following are equivalent.
\begin{itemize}
\item[1.] $M$ is the cokernel of an injective morphism
\[
B:\rho^*V\otimes \sO_\rho(-1)\to \rho^*W
\]
with $V$, $W$ vector bundles of the same rank on $C$.
\item[2.]  $M$ has $1$-dimensional support, $M\otimes \sO_\rho(-1)$ is
  $\rho_*$-acyclic, and $\rho_*M$ is torsion-free.
\end{itemize}
Moreover, if either condition holds, then $B$ is uniquely determined
up to isomorphism by $M$.
\end{prop}

\begin{proof}
  $1\implies 2$: Since $B$ is an injective morphism of vector bundles of
  the same rank, it is an isomorphism on the generic fiber, and thus
  $\supp(M)$ does not contain the generic point of $X$.  It follows that
  $M$ has $\le 1$-dimensional support.  (In fact, $M$ is supported on the
  zero locus of $\det(B)$.)

  Now, since the sheaves $\sO_\rho(d)$ are isomorphic to $\sO_f(d)$ on
  every fiber $f$, we find that $\sO_\rho(d)$ is $\rho_*$-acyclic for $d\ge
  -1$, and has trivial direct image for $d\le -1$.  In particular, we can
  compute the higher direct image long exact sequence associated to the
  short exact sequence
\[
0\to \rho^*V\otimes \sO_\rho(-2)\to \rho^*W\otimes \sO_\rho(-1)\to M\otimes
\sO_\rho(-1)\to 0.
\]
Since $\rho$ has $1$-dimensional fibers, this long exact sequence
terminates after degree 1, and we conclude that $M\otimes \sO_\rho(-1)$ is
$\rho_*$-acyclic.  Similarly, from the untwisted short exact sequence, we
obtain
\[
W\cong \rho_*\rho^*W\cong \rho_*M.
\]
In particular, $\rho_*M$ is torsion-free, and we can recover $B$ as the
kernel of the natural map $\rho^*\rho_*M\to M$.

$2\implies 1$: The condition that $M\otimes \sO_\rho(-1)$ is
$\rho_*$-acyclic implies that if we view $M$ as a family of sheaves on
$\P^1$, then every fiber satisfies $H^1(M_f(-1))=0$.  In particular, every
fiber is $0$-regular in the sense of Castelnuovo and Mumford, and thus $M$
is relatively globally generated \cite{KleimanS:1971}. Since $\rho_*M$ is
torsion-free by assumption, so a vector bundle, it remains only to show
that the kernel of this natural map has the form $\rho^*V\otimes
\sO_\rho(-1)$.  Now, $M$ cannot have any $0$-dimensional subsheaf,
since that would produce a $0$-dimensional subsheaf of $\rho_*M$.  In other
words, $M$ is a pure $1$-dimensional sheaf, and thus has homological
dimension $1$.  In particular, the kernel is a vector bundle (of the same
rank as $W$, since the map is generically surjective), so we can view it as
a flat family of sheaves on $\P^1$.  Since $\rho^*\rho_*M$ and $M$ are
acyclic with isomorphic direct image, it follows that the kernel has
trivial direct image and higher direct image, and thus every fiber of the
kernel has trivial cohomology.  The only sheaves on $\P^1$ with trivial
cohomology are sums of $\sO_{\P^1}(-1)$, and thus the kernel has the form
\[
V'\otimes \sO_\rho(-1)
\]
where $V'$ is a flat family of sheaves on $\P^1$, each fiber of which is a
power of $\sO_{\P^1}$.  In other words, $V'\cong \rho^*V$ for some vector
bundle $V$.
\end{proof}

\begin{rems}
Just as we found $W\cong \rho_*M$, we can also compute $V$ from $M$, since
\[
\rho_*(M\otimes \sO_\rho(-1))\cong V\otimes R^1\rho_*\sO_\rho(-2),
\]
and $R^1\rho_*\sO_\rho(-2)$ is a line bundle on $C$.
\end{rems}

\begin{rems}
  This argument was inspired by the main construction of \cite{BeauvilleA:2000},
  which considered minimal resolutions of sheaves on $\P^n$ for $n>1$; in
  our case, we have a {\em relative} minimal resolution of a family of
  sheaves on $\P^1$.
\end{rems}

Of course, there remain two conditions to translate into conditions on the
sheaf $M$, namely the constraint on the singularities, and the constraint
that $V$ is maximal.  The former is straightforward: specifying the
elementary divisors of $B$ along $C_\alpha$ is equivalent to specifying the
cokernel of $B$ as a morphism of vector bundles on $C_\alpha$, and thus the
singularities are determined by the restriction $M|_{C_\alpha}$.  (In
particular, we have the overall constraint that $M$ must be transverse to
$C_\alpha$, so that $B$ is generically invertible on $C_\alpha$!)  The
latter is somewhat more subtle, but is not too difficult to deal with.

\begin{prop}
Let $\rho:X\to C$ be a ruled surface, and suppose the sheaf $M$ is given by
a presentation
\[
0\to \rho^*V\otimes \sO_\rho(-1)\xrightarrow{B} \rho^*W\to M\to 0,
\]
where $V$ and $W$ are vector bundles of the same rank.
The morphism $B$ extends to a supersheaf $V\subsetneq V'\subset
V\otimes_{\sO_C} k(C)$ iff $M$ has a subsheaf of the form $\sO_f(-1)$ for
some fiber $f$ of $\rho$.
\end{prop}

\begin{proof}
If $B$ extends to $V'$, then the image of $\rho^*V'\otimes \sO_\rho(-1)$
induces a subsheaf of $M$ isomorphic to
\[
\rho^*(V'/V)\otimes \sO_\rho(-1).
\]
Now, $V'/V$ is $0$-dimensional, so contains a subsheaf of the form $\sO_p$
for some closed point $p\in C$.  This $\sO_p$ itself induces a supersheaf
of $V$, and thus a subsheaf of $M$ of the form
\[
\rho^*(\sO_p)\otimes \sO_\rho(-1)\cong \sO_f(-1),
\]
where $f$ is the fiber over $p$.

Conversely, suppose we have an injective map $\sO_f(-1)\to M$, and let $M'$
be the cokernel.  The higher direct image long exact sequences tell us
\[
\rho_*M'\cong \rho_*M
\qquad
R^1\rho_*M'\cong R^1\rho_*M=0
\qquad
R^1\rho_*(M'\otimes \sO_\rho(-1))\cong R^1\rho_*(M'\otimes \sO_\rho(-1))=0,
\]
and thus $M'$ has a presentation of the form
\[
0\to \rho^*V'\otimes \sO_\rho(-1)\to \rho^*W\to M'\to 0.
\]
Since this construction is functorial, we obtain an injective morphism
\[
\rho^*V\otimes \sO_\rho(-1)\to \rho^*V'\otimes \sO_\rho(-1),
\]
thus an injective morphism $\rho^*V\to \rho^*V'$, and by adjunction,
$V\subset V'$ in such a way that $B$ extends.
\end{proof}

There is a dual condition related to relative global generation.

\begin{prop}
Let $\rho:X\to C$ be a ruled surface, and suppose that $M$ is a pure
$1$-dimensional sheaf on $X$.  If $M$ is $\rho_*$-acyclic, then $M$ is
relatively globally generated iff no quotient of $M$ has the form
$\sO_f(-1)$
for some fiber $f$ of $\rho$.
\end{prop}

\begin{proof}
If $M$ is relatively globally generated, then
\[
\Hom(M,\sO_f(-1))\subset \Hom(\rho^*\rho_*M,\sO_f(-1))\cong
\Hom(\rho_*M,\rho_*\sO_f(-1))=0.
\]

For the converse, consider the natural map $\rho^*\rho_*M\to M$, viewed as
a two-term complex.  The terms in the complex are $\rho_*$-acyclic, and
thus the derived direct image of the complex is
\[
\rho_*\rho^*\rho_*M\cong \rho_* M,
\]
so is exact.  On the other hand, there is a spectral sequence converging to
this result in which we first take the cohomology of the complex before
taking higher direct images.  Since $\rho$ has $1$-dimensional fibers, this
spectral sequence stabilizes at the $E_2$ page, and we thus conclude that
the cohomology sheaves of the complex have trivial direct image and higher
direct image.

We thus conclude that if $M$ is not globally generated, then $M$ has a
surjective morphism to a nonzero sheaf $M'$ with $\rho_*M'=R^1\rho_*M'=0$,
so
\[
M'\cong \rho^*\rho_*(M'\otimes \sO_\rho(1))\otimes \sO_\rho(-1)
\]
Now, $\rho_*(M'\otimes \sO_\rho(1))$ cannot be 0, since that would force
$M'=0$.  It thus admits a surjective map to some $\sO_p$, which induces a
surjective map from $M'$ to a sheaf of the form $\sO_f(-1)$.
\end{proof}

\begin{rem}
In fact, it follows from this that if $M$ is pure $1$-dimensional and
$\rho_*$-acyclic, then it is relatively globally generated iff $M\otimes
\sO_\rho(-1)$ is $\rho_*$-acyclic.  Indeed, a surjection $M\to \sO_f(-1)$
induces a surjection
\[
R^1\rho_*(M\otimes \sO_\rho(-1)) \to R^1\sO_f(-2)\cong \sO_{\pi(f)}
\]
making the former sheaf nontrivial.
\end{rem}

Similar conditions apply to $\rho_*$-acyclicity and torsion-freeness of
$\rho_*M$.

\begin{prop}
Let $\rho:X\to C$ be a ruled surface, and let $M$ be a pure $1$-dimensional
sheaf on $X$.  Then $\rho_*M$ is torsion-free iff $\Hom(\sO_f,M)=0$ for all
fibers $f$ of $\rho$, and $M$ is $\rho_*$-acyclic iff $\Hom(M,\sO_f(-2))=0$
for all $f$.
\end{prop}

\begin{proof}
For the first condition, we have
\[
\Hom(\sO_f,M)
\cong
\Hom(\rho^*\sO_{\pi(f)},M)
\cong
\Hom(\sO_{\pi(f)},\rho_*M)
\]
Since $\rho_*M$ is torsion-free iff it has no maps from point sheaves, the
first claim follows.

For the second, the same spectral sequence argument based on the complex
$\rho^*\rho_*M\to M$ tells us that if $M'$ is the cokernel of this natural
map, then $R^1\rho_*M\cong R^1\rho_*M'$ and $\rho_*M'=0$.  Since $M'$ is
supported on finitely many fibers as before, it must have a quotient of the
form $\sO_f(-d)$ for some $d>1$, and thus has a nontrivial morphism to
$\sO_f(-2)$.
\end{proof}

If $\Hom(\sO_f,M)=0$ but $\Hom(\sO_f(-1),M)\ne 0$, then any such morphism
is necessarily injective; similarly, if $\Hom(M,\sO_f(-2))=0$ but
$\Hom(M,\sO_f(-1))\ne 0$, then any such morphism is surjective.  We thus
arrive at the final moduli problem: Classify pure $1$-dimensional sheaves
$M$ on $X$ with specified restriction to $C_\alpha$ and such that
$\Hom(M,\sO_f(-1))=\Hom(\sO_f(-1),M)=0$ for all fibers $f$ of the ruling.

We are also interested in understanding when $W$ is trivial, for which we
have the following numerical condition in the rational case.

\begin{lem}\label{lem:cohom_vanish}
Let $\rho:X\to \P^1$ be a Hirzebruch surface, and let $M$ be a sheaf on $X$
with $1$-dimensional support.  Then the following are equivalent:
\begin{itemize}
\item[(a)] $H^0(M)=H^1(M)=0$
\item[(b)] $M$ is $\rho_*$-acyclic and $\rho_*M\cong \sO_{\P^1}(-1)^n$ for
  some $n\ge 0$.
\end{itemize}
\end{lem}

\begin{proof}
  Since $\rho$ has $1$-dimensional fibers, $R^p\rho_*M=0$ for $p>1$; since
  the generic fiber of $\supp(M)$ over $\P^1$ is $0$-dimensional,
  $R^1\rho_*M$ has $0$-dimensional support.  The
  Leray-Serre spectral sequence
\[
H^p(R^q\rho_*M)\implies H^{p+q}(M)
\]
thus implies isomorphisms
\[
H^0(M)\cong H^0(\rho_*M)\qquad H^2(M)\cong H^1(R^1\rho_*M)=0
\]
and a short exact sequence
\[
0\to H^1(\rho_*M)\to H^1(M)\to H^0(R^1\rho_*M)\to 0.
\]
In particular, $H^0(M)=H^1(M)=0$ iff both $\rho_*M$ and $R^1\rho_*M$ have
vanishing cohomology.  In particular, both must be isomorphic to a direct
sum of line bundles $\sO_{\P^1}(-1)$, and since $R^1\rho_*M$ has
$0$-dimensional support, it must be $0$.
\end{proof}

Thus the only way a 1-dimensional sheaf with $H^*(M\otimes
\rho^*\sO_{\P^1}(-1))=0$ could fail to have a presentation of the standard
form is if $\Hom(M,\sO_f(-1))\ne 0$ for some fiber $f$.  (Even if
it fails this last condition, the image of $\rho^*\rho_*M\to M$ gives us a
subsheaf with standard presentation, and thus a subquotient with standard
presentation satisfying maximality.)

\medskip

When we considered elliptic difference equations above, this was in fact a
simplification: in fact, the difference equations that occur in the theory
of elliptic special functions have {\em theta function} coefficients in
general.  Algebraically speaking, $A$ is really a matrix with coefficients
in ${\cal L}_0\otimes_{\sO_{C_\alpha}}k(C_\alpha)$, where ${\cal L}_0$ is a
degree 0 line bundle equipped with an isomorphism
\[
\eta^*{\cal L}_0\cong {\cal L}_0^*
\]
such that the composition
\[
{\cal L}_0=\eta^*\eta^*{\cal L}_0\cong \eta^*{\cal L}_0^*\cong {\cal L}_0
\]
is the identity.  Again, Hilbert's Theorem 90 allows us to factor ${\cal L}_0$,
though now the nonuniqueness is more significant.  If ${\cal L}_0$ has the
above form, then it can (since we are over an algebraically closed field)
be factored as
\[
\psi:{\cal L}_0\cong \eta^*{\cal L}\otimes {\cal L}^*,
\]
for some line bundle ${\cal L}$.  This line bundle is nonunique in two
respects: the obvious one is that it can be twisted by any power of the
line bundle $\pi_\eta^*{\cal O}_{\P^1}(1)$, but even modulo this, there are
8 possibilities for ${\cal L}$, and for each such choice, 2 possibilities
for the isomorphism $\psi$.  Indeed, we can twist ${\cal L}$ by the degree
$1$ bundle corresponding to any ramification point of $\eta$; if we twist
by them all, this is the same as twisting by $\pi_\eta^*{\cal
  O}_{\P^1}(2)$, except that the isomorphism is multiplied by $-1$.  (In
characteristic 2, the issue of nonuniqueness is somewhat more complicated.
In general, modulo twisting $\pi_\eta^*{\cal O}_{\P^1}(1)$, the
factorizations form a torsor over an abelian group scheme with structure
$\mu_2.\Pic^0(C_\alpha)[2].\Z/2\Z$.)

This nonuniqueness is related to the question of singularities at the four
ramification points: we can no longer canonically distinguish between the
two possible local rank $1$ equations, in order to decide which one should
be viewed as regular.  Since we want to specify the singularity structure,
we should view the factorization of ${\cal L}_0$ as part of the
specification (and indeed, in all of the motivating cases, there is a
natural choice of factorization making the equation regular at the fixed
points of $\eta$ for generic parameters).  With this in mind, we again
obtain a canonical factorization, except now $B$ has the form
\[
B:\pi_\eta^*V\to \pi_\eta^*W\otimes {\cal L}.
\]
Again, two uses of adjunction allow us to extend this to a morphism of
vector bundles on the Hirzebruch surface $\P(\pi_{\eta *}{\cal L})$,
and thus further to the cokernel of that morphism.  The above
considerations extend immediately to the case of general ${\cal L}$.

Note that $\P(\pi_{\eta*}{\cal L})$ is isomorphic to the Hirzebruch surface
$F_1$ iff ${\cal L}$ has odd degree, and is otherwise isomorphic to $F_0$
or $F_2$, with the latter precisely when ${\cal L}$ is a power of
$\pi_\eta^*\sO_{\P^1}(1)$.

We should also note that using the freedom to twist ${\cal L}$ by powers of
$\pi_\eta^*\sO_{\P^1}(1)$, we can assume that ${\cal L}^*$ is represented
by an effective divisor disjoint from the ramification locus.  The
resulting isomorphism ${\cal L}\cong \sO_{C_\alpha}(-D)$ allows us to
translate the problem back to one on $F_2$, but with additional
``apparent'' singularities along $D$.  (These are points where the
obstructions to having symmetric solutions can be gauged away, possibly at
the expense of introducing apparent singularities along other orbits of the
infinite dihedral group.)

\medskip

Although the generic anticanonical curve on $F_2$ is a smooth genus 1 curve
disjoint from $s_{\min}$, there is significant scope for
degeneration.  We can view $F_2$ as the minimal desingularization of a
weighted projective space, and in this way anticanonical curves correspond
to equations of the form
\[
p_0(x,w) y^2 + p_2(x,w) y + p_4(x,w)=0,
\]
with $p_d(x,w)$ homogeneous of degree $d$.  (Here $w$,$x$,$y$ are
generators of degrees $1$, $1$, and $2$ of a graded algebra, and $F_2$ is
the minimal desingularization of $\Proj(k[w,x,y])$.)  The anticanonical
curve is disjoint from $s_{\min}$ iff $p_0\ne 0$, and thus we should consider
equations
\[
y^2 + p_2(x,w) y + p_4(x,w)=0.
\]
On any such curve, we have an involution $y\mapsto -p_2(x,w)-y$ which
exchanges the two points on any given fiber.  Given any such curve, the
translation between matrices $B$ on $C_\alpha$ and matrices on $F_2$ is quite
explicit in terms of coordinates: express $B$ in terms of the coordinates,
and use the equation of $C_\alpha$ to eliminate any term of degree $\ge 2$ in
$y$.  The resulting matrix, every coefficient of which is linear in $y$ and
weighted homogeneous, can now be viewed as a matrix on $F_2$, and is a
canonical extension of $B$.

There are several degenerate cases to consider.  In characteristic not 2,
we can complete the square to make $p_2=0$, and the degenerate cases are
classified by the multiplicities of the zeros of $p_4$ (with an additional
case when $p_2=p_4=0$).  These cases all extend to characteristic 2; there
is one extra case in characteristic 2 ($p_2=0$, $p_4$ is not a square) for
which we do not have a natural difference/differential equation
interpretation, so do not consider below.  Note that over a perfect field,
any such curve is equivalent to the curve $y^2=x w^3$.

\begin{itemize}
\item[211:] $C_\alpha$ is integral, with a single node.  Then $C_\alpha$ is
  isomorphic to $\P^1$ with $0$ and $\infty$ identified, and $\eta$ acts on
  this $\P^1$ as $z\mapsto \beta/z$ for some $\beta$.  Up to a change of
  coordinates on $F_2$, $C_\alpha$ has the equation
\[
y^2-xwy+\beta w^4=0,
\]
with $w(z)=1$, $x(z)=z+\beta/z$, $y(z)=z$.  If we choose an automorphism
$\tau_q:z\mapsto qz$, then the above construction applies to relate sheaves
on $F_2$ to symmetric $q$-difference equations
\[
v(qz) = A(z)v(z)
\]
such that $A\in \GL_n(k(z))$ satisfies $A(\beta/z)=A(z)^{-1}$.  Note that
this is singular at the node unless $A(0)=A(\infty)=1$.

\item[31:] $C_\alpha$ is integral, with a single cusp.  Then $C_\alpha$ is
  identified with $\P^1$ such that the cusp maps to $\infty$ and
  $\eta(z)=\beta-z$ for some $\beta$.  Up to a change of coordinates on
  $F_2$, $C_\alpha$ has the equation
\[
y^2-\beta w^2 y+x w^3= 0,
\]
with $w(z)=1$, $x(z)=z(\beta-z)$, $y(z)=z$.  These correspond to symmetric
ordinary difference equations:
\[
v(z+q)=A(z)v(z)
\]
with $A\in \GL_n(k(z))$ such that $A(\beta-z)A(z)=1$.  The equation is
singular at the cusp unless $A(z)=1+O(1/z^2)$ as $z\to\infty$. (The
symmetry then implies $A(z)=1+O(1/z^3)$.)

\item[22:] $C_\alpha$ is a union of two smooth components (isomorphic to
  $\P^1$) meeting in two distinct points, and $\eta$ swaps the components.
  In suitable coordinates, $C_\alpha$ has the equation
\[
y^2-xwy=0,
\]
with components $y=xw$, $y=0$; we may view $z=x/w$ as a common coordinate
on the two components.  Then any morphism $B$ on $F_2$ as above specifies a
{\em pair} of morphisms
\[
B_1,B_2:V\to \sO_{\P^1}^n,
\]
agreeing at $0$ and $\infty$.  The corresponding $A$ matrix on $C_\alpha$ is
really a pair of inverse matrices, but we may simply view it as a single
matrix $B_2^{-t}B_1^t$ on one component of $C_\alpha$.  In this way, we
obtain a $q$-difference equation on $\P^1$ without any symmetry condition,
and $B_1$, $B_2$ separate the zeros and poles of the equation.
Singularities at the two nodes of $C_\alpha$ arise when $A(0)\ne 1$ or
$A(\infty)\ne 1$ respectively.

\item[4:] $C_\alpha$ is a union of two smooth components which are tangent at
  a single point, and $\eta$ swaps the components; $C_\alpha$ has equation
  $y^2=w^2 y$, up to changes of coordinates.  Again $B$ specifies a pair of
  morphisms, which now agree to second order at $\infty$.  This corresponds
  to ordinary difference equations without symmetry, which are singular at
  $\infty$ unless $A(z)=1+O(1/z^2)$ as $z\to\infty$.

\item[0:] $C_\alpha$ is nonreduced, with equation $y^2=0$ after a change of
  coordinates.  In this case, the degree 2 morphism $C_\alpha\to \P^1$ is no
  longer generically \'etale, so is not the quotient by an involution.
  However, we can now identify $B$ with a pair of maps
\[
B_0:V\to \sO_{\P^1}^n,\qquad
B_\omega: V\to \omega_{\P^1}^n,
\]
giving a canonical factorization of a meromorphic matrix taking values in
$\omega_{\P^1}$.  Since there is a canonical connection on $\sO_{\P^1}$,
we may use this to interpret the meromorphic matrix with values in
$1$-forms as a meromorphic connection.  This, of course, is just the
standard translation between differential equations and sheaves on $F_2$
arising in the usual theory of Hitchin systems.
\end{itemize}

The construction in the nonsymmetric difference equation cases extends to
one for general elliptic difference equations.

\begin{eg}
Consider a general elliptic difference equation $\tau_q^*v=Av$ with $A\in
\GL_n(k(C))$ for some genus 1 curve $C$.  There is a natural factorization
\[
A = B_\infty^{-t} B_0^t
\]
where $B_0,B_\infty:V\to \sO_{C}^n$ and $V$ is a maximal vector bundle
supporting such a factorization.  (The existence of a meromorphic
factorization is trivial (take $B_\infty=1$, $B_0=A^t$), and implies the
existence of a unique maximal $V$ as above.)  The singularity structure of
$A$ then corresponds in a natural way to the cokernels of $B_0$ and
$B_\infty$ (giving zeros and poles respectively).  The pair
$(B_0,B_\infty)$ extends immediately to a morphism of vector bundles on the
ruled surface $E\times \P^1$: just take $zB_\infty+wB_0$, where $(z,w)$ are
homogeneous coordinates on $\P^1$.  We can then recover the pair as the
restriction of this morphism to the anticanonical curve $zw=0$, a union of
two disjoint copies of $C$.  This is essentially just the construction for
the Sklyanin integrable system (see \cite{HurtubiseJC/MarkmanE:2002a}), the
only difference being that the construction in the literature twists by a
line bundle in order to absorb all of the poles, making $B_\infty=1$, but
making the ruled surface more complicated.  (This corresponds to performing
a sequence of elementary transformations centered at the points where the
sheaf $M$ meets the component $w=0$ of the anticanonical curve.)  More
generally, we should allow difference equations on vector bundles, i.e.,
meromorphic (and meromorphically invertible) maps $A:V\to\tau_q^*V$.  If
there is a holomorphic isomorphism $A_0:V\cong \tau_q^*V$, then one can
divide by $A_0$ to again reduce to a sheaf on $E\times \P^1$.  By Atiyah's
classification of vector bundles on smooth genus 1 curves, we find that
$V\cong \tau_q^*V$ whenever $V$ is a sum of indecomposable bundles of
degree 0 (and this is necessary if $q$ has infinite order).  This is again
an open condition on $V$ (for degree 0 bundles, it is equivalent to
semistability), and the isomorphism $A_0$ can be at least partially
globalized (i.e., it exists on some open cover of the open subset); thus,
just as in the rational cases, we can identify large open subsets of the
moduli spaces of sheaves and of difference equations.  (The main
distinction is that the identification is no longer canonical, since $A_0$
is only determined up to scalars.)  Note also that $V$ is semistable of
degree 0 iff there exists some line bundle ${\cal L}$ of degree 0 such that
$H^0(V\otimes {\cal L})=H^1(V\otimes {\cal L})=0$; this gives an analogue
of Lemma \ref{lem:cohom_vanish} for the elliptic case.
\end{eg}

\begin{eg}
Similarly, the Hitchin system corresponds to the analogous factorization
for a meromorphic morphism
\[
A:k(C)\to \omega_C\otimes_{\sO_{C}} k(C),
\]
and gives a sheaf on the anticanonical surface $\P(\sO_C\oplus \omega_C)$.
As before, the construction in the literature essentially absorbs the poles
of $A$ into the structure of the anticanonical surface, but this is just a
sequence of elementary transformations.  Once again, the ``true'' moduli
space (of meromorphic connections on vector bundles) and the moduli space
of sheaves can be identified along large open subsets; in this case, the
requirement is that the vector bundle $V$ admit a holomorphic connection.
This is no longer an open condition (it is equivalent to every
indecomposable summand having degree a multiple of the characteristic
\cite{BiswasI/SubramanianS:2006}), but is implied by open conditions, e.g.,
that $V$ is semistable of degree 0 or stable of degree a multiple of the
characteristic.  It is also implied by the open condition that
$H^0(V\otimes {\cal L})=H^1(V\otimes {\cal L})=0$ for some line bundle
${\cal L}$ of degree $g-1$, though this is no longer equivalent to
semistability.
\end{eg}

Specifying an anticanonical curve on a smooth projective surface is
tantamount to specifying a Poisson structure; more precisely, the structure
is determined up to a scalar multiple, which can be fixed by a suitable
choice of nonzero holomorphic differential on the anticanonical curve.  The
above examples account, up to birational maps respecting the Poisson
structure, for almost every Poisson surface which is not symplectic, since
it was shown in \cite{poisson} that every such surface (apart from an
exotic family of examples in characteristic 2) has the form $\P(\sO_C\oplus
\omega_C)$ in such a way that the section corresponding to $\omega_C$ is
disjoint from the anticanonical curve.  That our surfaces have this
structure is significant since it follows from
\cite{HurtubiseJC/MarkmanE:2002b,poisson} that the moduli space of sheaves
with specified restriction to the anticanonical curve is symplectic (and
the closure inside the moduli space of stable sheaves is Poisson).

The only missing example (up to isomorphism) is the Poisson structure on
the characteristic 2 surface $F_2(\overline{\F_2})$ corresponding to the
anticanonical curve $y^2 = x w^3$.  (The exotic surfaces in characteristic
2 do not have Poisson moduli spaces, so we may feel free to ignore them.)
This is an irreducible cuspidal curve as in case $31$ above, but the degree
2 map to $\P^1$ is not \'etale, making the interpretation as a symmetric
difference equation problematical.  (Of course ordinary difference
equations in finite characteristic are already somewhat problematical,
since the infinite cyclic group acts via a finite quotient.)

\medskip

There are a few other moduli problems that translate to sheaves on
anticanonical rational surfaces that we want to consider.

\begin{eg}
  Let $C_\alpha$ be a smooth genus 1 curve, and let ${\cal L}$ be a line
  bundle on $C_\alpha$ with $\deg{\cal L}\ge 0$.  Consider the problem of
  classifying matrices $B\in \Mat_n(\Gamma({\cal L}))$ such that
  $\det(B)\ne 0$, up to left- and right- multiplication by constant
  matrices.  For any choice of hyperelliptic involution $\eta$, we can
  encode such matrices via sheaves on the Hirzebruch surface
  $\P(\pi_{\eta*}{\cal L})$, essentially as above.  The only additional
  condition we impose is that the bundle $V$ must also be trivial.
  In the case $\deg({\cal L})=3$, we can also interpret the matrix as one
  over $\Gamma(\sO_{\P^2}(1))$, and on $\P^2$, the matrix is the minimal
  resolution of its cokernel, as in \cite{BeauvilleA:2000}.  Note that
  since in this case we want both bundles to be trivial up to twist, we
  need to impose another numerical condition, which is straightforward to
  determine from the formula for $V$ in terms of $M$ that we gave above.
\end{eg}

\begin{eg}
  On the Hirzebruch surface $F_0=\P^1\times \P^1$, which we can view as the
  smooth quadric $xz=yw$ in $\P^3$, we can apply the previous construction
  to the degenerate anticanonical curve $xz=0$ (and the induced line bundle
  of degree $4$).  This has four components (two fibers and two sections),
  forming a quadrangle.  Any linear combination of the coordinates on
  $\P^3$ is determined by its values at the four points of intersection of
  the quadrangle, and thus to specify a linear matrix $B$, it is equivalent
  to specify a quadruple $(B_0,B_1,B_2,B_3)$ of scalar matrices.  If these
  matrices are invertible, then the restriction of $\coker(B)$ to $xz=0$ is
  determined by its restrictions to the four components, and thus by the
  conjugacy classes of the matrices
\[
B_0^{-1}B_1,B_1^{-1}B_2,B_2^{-1}B_3,B_3^{-1}B_0.
\]
In other words, the problem of classifying sheaves on $\P^1\times \P^1$
with the appropriate kind of presentation and restriction to the quadrangle
is equivalent to the problem of classifying quadruples in $\GL_n(k)$ with
specified conjugacy classes and product $1$.  This is the four-matrix case
of the ``multiplicative Deligne-Simpson problem'',
\cite{Crawley-BoeveyW/ShawP:2006}.
\end{eg}

\begin{eg}
Of course, we can obtain the three-matrix version of the multiplicative
Deligne-Simpson problem by insisting that $B_3=B_0$, or in other words that
the sheaf meet the relevant component in $\sO_p^n$ where $p$ is the point
representing $1$.  If we blow this point up and blow down both the fiber
and the section containing it, we obtain a moduli problem on $\P^2$
concerning sheaves with specified restriction to the triangle $xyz=0$.
\end{eg}

There does not appear to be any way to translate more general multiplicative
Deligne-Simpson problems into the Poisson surface framework.  For the
additive problem, the situation is nicer.

\begin{eg}
  The Hirzebruch surface $F_d$ for $d\ge 1$ has a unique section $s_{\min}$
  of negative self-intersection ($s_{\min}^2=-d$).  Given any $d+2$
  distinct fibers $f_0$,\dots,$f_{d+1}$, the divisor $2s_{\min}+\sum_i f_i$
  is anticanonical, so we may take it as our curve $C_\alpha$.  Given any
  $(d+2)$-tuple of matrices $C_0$,\dots,$C_{d+1}$ with $\sum_i C_i=0$, we
  have a natural corresponding matrix with coefficients in
  $\sO_X(s_{\min}+df)$.  Indeed, we can coordinatize $F_d$ in terms of a
  weighted projective space with a generator $y$ of degree $d$, and
  consider the matrix
\[
B = y + C(x,w)
\]
such that $C(x,w)$ is equal to $C_i$ on $f_i$.  Then the restriction of $B$
to $f_i$ is given by the conjugacy class of $C_i$.  In this way, we obtain
the $(d+2)$-matrix additive Deligne-Simpson problem (classifying $(d+2)$-tuples
of matrices with specified conjugacy classes and sum $0$).
\end{eg}

\begin{rem}
This, of course, is closely related to the problem of classifying Fuchsian
differential equations with specified singularity structure; if we blow up
a point of each fiber $f_i$ then blow down $f_i$, we can in this way
eliminate all components of $C_\alpha$ but the strict transform of $s_{\min}$.
If we choose the centers of the elementary transformations carefully, we
can arrange to end up at $F_2$, and case $0$ above.
\end{rem}

\medskip

If we forget the symmetry of a symmetric elliptic difference equation, we
obtain a subspace of the moduli space of all elliptic difference equations,
which we can understand in the following way.  On the surface $C\times
\P^1$, consider the involution $\eta\times (z\mapsto z^{-1})$.  This
preserves our standard choice of anticanonical curve, and more precisely
preserves the corresponding Poisson structure (in contrast to $\eta\times
1$, say, which {\em negates} the Poisson structure).  It follows that this
involution acts on the corresponding moduli space, again preserving the
Poisson structure, and thus the fixed locus inherits a Poisson structure
(at least in characteristic $\ne 2$).  There are some difficulties in
studying symmetric equations from this perspective, however.  One is that,
as we have seen, the notion of singularity should really take into account
the symmetry, but another is that when working with moduli spaces, a fixed
point merely indicates a sheaf which is {\em isomorphic} to its image under
the symmetry.  Since a sheaf only determines a matrix up to a choice of
basis, not every point of the fixed locus actually corresponds to a
symmetric equation.  (The situation is not too dire, though: symmetric
equations form a component of the fixed locus.)  Note that the quotient of
$E\times \P^1$ by the above involution is still an elliptic surface (with
constant $j$ invariant, and with two $I_0^*$ fibers in characteristic not
$2$), and thus must be blown down eight times to reach a Hirzebruch
surface.  (We can arrange to reach the usual $F_2$ constructed from
$(C,\eta)$, in which case the map from the elliptic surface blows up each
fixed point of $\eta$ twice.)  This reflects both the fact that the notion
of singularity changes and the fact that the fixed points of the moduli
space are the equations which are symmetric {\em up to isomorphism}.

Similar comments apply if we try to relate symmetric and nonsymmetric
difference equations in the $q$-difference and ordinary difference cases.
More generally, we could consider any Poisson involution on one of our
Poisson Hirzebruch surfaces.  We find that the most general Poisson
involution (again, in characteristic not 2) is again at the elliptic level,
and is simply given by translation by a $2$-torsion point $p$ of
$\Pic^0(C)$; any other Poisson involution on $F_2$ is a degeneration of
this (on some degenerate curve).  Given any symmetric elliptic difference
equation on the isogenous curve $C/\langle p\rangle$, we can interpret it
as an equation on $C$ (typically with twice as many singularities), and the
corresponding sheaf will be invariant under the Poisson involution.  Since
any Poisson involution degenerates this, it in particular follows that the
embedding of symmetric equations in the moduli space of nonsymmetric
equations is a degeneration of this ``quadratic transformation''.  (So
called because at the bottom, differential level, that is precisely what it
is: performing a quadratic change of variables in the differential
equation.)  Once again, the mismatch between the two notions of singularity
and the fact that equations can be symmetric up to isomorphism without
being symmetric is reflected in the fact that the quotient by the
involution is a (singular) del Pezzo surface of degree 4 with an
$A_1A_1A_3$ configuration of $-2$-curves.  (One of the $-2$-curves comes
from the original $-2$-curve, while the other four come from fixed points
of the involution, two of which are on the original $-2$-curve.)

There may also be some interesting phenomena related to anti-Poisson
involutions of rational surfaces (which can be identified by the fact that
they are hyperelliptic when restricted to the anticanonical curve).  Though
these remain anti-Poisson on the moduli space, they can be combined with a
natural duality operation on sheaves to again obtain a Poisson involution
on the moduli space.  One example of this is the adjoint operation
$A\mapsto A^{-t}$, see Section \ref{sec:elldiff} below.

\section{Blowdown structures on rational surfaces}

In \cite{poisson}, we gave a construction for lifting sheaves (of
homological dimension $\le 1$, so in particular sheaves of pure dimension
$1$) through birational morphisms.  (In the case of the direct image of a
line bundle on a smooth curve, this is the obvious lift to the strict
transform, but the construction applies more generally.)  Moreover, up to
``pseudo-twist'', we can lift any sheaf transverse to the anticanonical
curve to some blowup on which it is {\em disjoint} from the anticanonical
curve.  (We will see that pseudo-twists correspond to certain canonical
gauge transformations of difference equations, so we do not lose much
generality by assuming we have such a lift.)  As a result, we find that we
want to consider sheaves on more general rational surfaces.  

Once we have lifted to a blowup of our Hirzebruch surface, we encounter a
new phenomenon: rational surfaces can be blown down to Hirzebruch surfaces
in multiple ways.  Although this is true in a mild sense for ruled surfaces
of higher genus, in those cases we find that any two blowdowns to
geometrically ruled surfaces are related by a sequence of elementary
transformations (corresponding to a very mild transformation of the
difference/differential equation).  In contrast, rational surfaces no
longer have a canonical rational ruling, and as a result a given sheaf on a
rational surface will tend to have multiple qualitatively different
interpretations as difference equations.  (For instance, we will see that
sheaves can correspond to symmetric or nonsymmetric $q$-difference
equations depending on the choice of blowdown.)

With this in mind, we want to understand the set of possible ways to blow a
rational surface down to a Hirzebruch surface.  It turns out to be useful
to record slightly more information than just a birational morphism to a
Hirzebruch surface, and we thus make the following definition.

\begin{defn}
Let $X$ be a rational surface with $X\not\cong \P^2$.  A (Hirzebruch)
blowdown structure on $X$ is a chain $\Gamma$ of morphisms
\[
X=X_m\to X_{m-1}\to\cdots\to X_0\to \P^1,
\]
such that for $1\le i\le m$, the morphism $X_i\to X_{i-1}$ is the blowup in a
single point of $X_{i-1}$, while the morphism from $X_0\to \P^1$ is a geometric
ruling.  Two blowdown structures will be considered equivalent if they fit
into a commutative diagram
\[
\begin{CD}
X@>>> X_{m-1}@>>>\cdots @>>> X_0@>>> \P^1\\
 @|  @VVV \cdots @. @VVV @VVV\\
X@>>> X'_{m-1}@>>>\cdots @>>> X'_0@>>> \P^1
\end{CD}
\]
\end{defn}

\begin{rem}
  Related structures have been studied in the case that the rational
  surface can be blown down to $\P^2$, \cite{SakaiH:2001}; one also
  considers the related notion of an exceptional configuration (essentially
  the analogue for $\P^2$ of the notion of numerical blowdown structure
  below) \cite{LooijengaE:1981}.  Our considerations below are somewhat
  more general (since not every surface blows down to $\P^2$), but of
  course closely related; for instance \cite{LooijengaE:1981} already saw the
  appearance of the root system $E_{m+1}$.
\end{rem}

Note that in addition to keeping track of a factorization of the birational
morphism, we also keep track of the ruling at the end (but only up to
$\PGL_2$).  Most of the time, of course, the latter provides no
information, since most Hirzebruch surfaces have a unique geometric ruling
(and thus have a unique blowdown structure).  The lone exception is
$\P^1\times \P^1$, and we note that above whenever we obtained $\P^1\times
\P^1$ from a problem of twisted difference equations, this came with a
choice of ruling.

One reason for including the above information in the blowdown structure is
that it allows us to associate a basis of $\Pic(X)$ to any blowdown
structure.  For each monoidal transformation $X_i\to X_{i-1}$, we have an
exceptional curve $e_i$ on $X_i$, and the total transform of this curve
gives us a divisor on $X$, which we also denote $e_i$.  Since
\[
\Pic(X)\cong \Pic(X_0)\oplus \bigoplus_i \Z e_i,
\]
it remains only to give a basis of $\Pic(X_0)\cong \Z^2$.  One basis
element is obvious, namely the class $f$ of the fibers of the ruling.  For
the other, there is also an obvious choice, namely the class $s_{\min}$ of
a section with minimal self-intersection.  This turns out not to be the
best choice for our purposes, however, as it gives us a countable infinity
of different intersection forms to consider.  A slightly different basis
greatly reduces the number of cases.  Define a divisor class
\[
s:=s_{\min} + \lfloor -s_{\min}^2/2\rfloor f.
\]
Since $s_{\min}\cdot f=1$, $f^2=0$, we find that $s\cdot f=1$, and $s^2$ is
either 0 or $-1$, depending on whether $s_{\min}^2$ was even or odd.  With
this in mind, we call a blowdown structure even or odd depending on the
parity of $s_{\min}^2$.  (Note that if $X_0$ comes from a line bundle on a
hyperelliptic genus 1 curve as in the previous section, then this choice of
basis element essentially corresponds to choosing the bundle to have degree
$1$ or $2$, as then $\sO_X(s)$ agrees with the relative $\sO(1)$.)

The basis we obtain has one of two possible intersection forms, depending
on parity.  In the even case, we have
\[
s^2 = 0,\quad s\cdot f=1,\quad f^2=0,
\quad s\cdot e_i=f\cdot e_i=0,\quad
e_i\cdot e_j=-\delta_{ij},
\]
while in the odd case, we have the same, except $s^2=-1$.
The expansion of the canonical class in the basis again only depends on
parity:
\[
K_X = \begin{cases}
-2s-2f+\sum_{1\le i\le m} e_i & \text{$\Gamma$ even}\\
-2s-3f+\sum_{1\le i\le m} e_i & \text{$\Gamma$ odd}.
\end{cases}
\]
And of course in either case we find $K_X^2 = 8-m$.  When $X_0=F_1$, we can
blow it down to $\P^2$, suggesting an alternate basis in the odd case:
replace $f$ by $h=s+f$, the class of a line in $\P^2$.  This gives an
orthonormal basis for $\Pic(X)$, with $K_X=-3h+s+\sum_i e_i$, but makes
the effective cone look rather strange when $X_0=F_{2d+1}$ for $d>0$.

Note that we can recover the blowdown structure from the corresponding
basis for $\Pic(X)$: blow down $e_m$, then the image of $e_{m-1}$, etc.,
and construct a map $X_0\to \P^1$ using $f$.  Of course, not every basis
with the correct intersection form will correspond to a blowdown structure,
but we will eventually give an algorithm for determining when a given basis
(expressed in terms of some original blowdown structure) also corresponds
to a blowdown structure.  In any event, we will define a {\em numerical}
blowdown structure to be a basis of $\Pic(X)$ having the same intersection
form as an even or odd blowdown structure on $X$.

The surface $X_1$ was obtained by blowing up a point of $X_0$, and we find
that the fiber containing that point becomes a pair of $-1$ curves on
$X_1$, of divisor classes $e_1$, $f-e_1$.  We thus obtain an alternate
blowdown structure on $X_1$ by blowing down $f-e_1$, producing $X'_0$
differing from $X_0$ by an elementary transformation.  The basis elements
$f$ and $e_i$ for $i\ge 2$ are unchanged by this transformation, but $e_1$
and $s$ are transformed as follows.
\[
(s',e'_1)
=
\begin{cases}
(s-e_1,f-e_1) & \text{$\Gamma$ even}\\
(s+f-e_1,f-e_1) & \text{$\Gamma$ odd}
\end{cases}
\]
Note that this swaps the even and odd cases, and if we perform the
transformation twice, we end up back at the original blowdown structure.

Another natural way to transform a blowdown structure is to rearrange
blowups.  If the morphism $X_{i+2}\to X_i$ blows up two distinct points of
$X_i$, then we can perform the blowdown in the other order, thus swapping
the basis elements $e_i$ and $e_{i+1}$.  Unlike the elementary
transformation case, this operation is {\em not} always legal, as when
$X_{i+2}\to X_{i+1}$ blows up a point of $e_i$, there is no longer any
choice in how to reach $X_i$.  However, when it applies, it has a
particularly nice action on the basis: it is simply the reflection with
respect to the intersection form in the divisor class $e_i-e_{i+1}$.
Similarly, when $X_0\cong \P^1\times \P^1$, we obtain another blowdown
structure by changing to the other ruling.  This swaps the basis elements
$s$ and $f$, and is the reflection in the divisor class $s-f$.

In this way, we obtain a collection of $m$ reflections in the even case,
$m-1$ in the odd case; if we perform an elementary transformation, the two
sets mostly overlap, but we obtain a total of $m+1$ different reflections
in this way (assuming $m\ge 2$).  The corresponding vectors are linearly
independent, and are given by
\[
s-f,f-e_1-e_2,e_1-e_2,\dots e_{m-1}-e_m
\]
in the even case and
\[
s-e_1,f-e_1-e_2,e_1-e_2,\dots,e_{m-1}-e_m
\]
in the odd case.  Note that each one of these vectors is orthogonal to $K$;
we can see this either by direct calculation or by noting that the
expansion of $K$ depends only on the parity of the blowdown structure, so
had better be invariant under the above reflections.

\begin{lem}
Suppose $m\ge 2$.  Then the above sets of vectors give a basis for the
orthogonal complement of $K$ in $\Pic(X)$.  With respect to the negative of
the intersection form, they form the set of simple roots of a Coxeter group
of type $E_{m+1}$.
\end{lem}

\begin{proof}
  In either case, we have $m+1$ vectors, while $\Pic(X)$ has rank $m+2$, and
  thus we obtain bases of the orthogonal complement over $\Q$.  Since the
  bases are obviously saturated (they are essentially triangular with unit
  diagonal), the first claim follows.

  That the vectors are
  simple roots for a Coxeter system follows from the fact that their inner
  products are nonpositive (i.e., the intersections are nonnegative).  To
  identify the system, note that the $e_i-e_{i+1}$ roots are the simple
  roots of a Coxeter group of type $A_{m-1}$, adjoining $f-e_1-e_2$ extends
  this to $D_m$, and adjoining $s-f$ or $s-e_1$ as appropriate extends one
  of the short legs of the $D_m$ Dynkin diagram.
\end{proof}

\begin{rem}
Note the small $m$ cases
\begin{align}
\notag E_3&=A_1\times A_2\\
\notag E_4&=A_4\\
\notag E_5&=D_5\\
\notag E_9&=\tilde{E}_8,
\end{align}
with $E_6$, $E_7$, $E_8$ as expected.
When $m=1$, we have only the root $s-f$ or $s-e_1$ as appropriate, and when
$m=0$, we have only $s-f$ in the even case, and no roots in the odd case.
\end{rem}

With this in mind, we refer to the given vectors as the simple roots for
the (numerical) blowdown structure.  The corresponding simple reflections
clearly give an action of $W(E_{m+1})$ on the set of numerical blowdown
structures.

\begin{lem}
Suppose $\Gamma$ is a blowdown structure for the rational surface $X$, and
let $\sigma$ be a simple root for $\Gamma$, with corresponding reflection $r_\sigma$.
If $\sigma$ is ineffective, then $r_\sigma\Gamma$ is a blowdown structure.
\end{lem}

\begin{proof}
  To be precise, we mean here that if the numerical blowdown structure
  $\Gamma$ comes from a blowdown structure, then so does $r_\sigma\Gamma$, as long
  as the divisor class $\sigma$ is ineffective.

  Using elementary transformations as appropriate, we may reduce to the
  cases $\sigma=s-f$ and $\sigma=e_i-e_{i+1}$.  If $s-f$ is ineffective on $X$, it is
  certainly ineffective on $X_0$, but then $X_0$ must be $\P^1\times \P^1$
  (if $X_0\cong F_{2d}$ for $d>0$, then $s_{\min}=s-df$ is effective), and
  we have already seen that the reflection gives a blowdown structure.
  Similarly, if $e_i-e_{i+1}$ is ineffective on $X$, it is ineffective on
  $X_{i+1}$, which implies that $X_{i+1}\to X_{i-1}$ blows up two distinct
  points of $X_i$, so that we need merely blow up the points in the
  opposite order.
\end{proof}

\begin{rem}
The roots of the $A_{m-1}$ subsystem act without changing the Hirzebruch
surface $X_0$, by permuting the distinct points being blown up.  Similarly,
the simple reflections of the $D_m$ subsystem leave the rational ruling
invariant.  If we combine those reflections with the action of the
elementary transformation, we obtain a group of type $C_m$ acting on the
different ways to blow down to a Hirzebruch surface compatibly with the
given ruling.
\end{rem}

Given a blowdown structure $\Gamma$, call an element $w\in W(E_{m+1})$ {\em
  ineffective} if there exists a word $w = r_1 r_2\cdots r_l$ with each
$r_i$ a simple reflection such that the corresponding simple root is
ineffective for the relevant blowdown structure, $r_{i+1}r_{i+2}\cdots
r_l\Gamma$.  (That this numerical blowdown structure comes from an actual
blowdown structure follows by an easy induction.)  In particular, if $w$ is
ineffective, then $w\Gamma$ is a blowdown structure.

We thus need to understand the effective simple roots.  By a ``$-d$-curve''
on a rational surface, we mean a smooth rational curve of self-intersection
$-d$.

\begin{lem}\label{lem:fixed_of_simple}
  Let $X$ be a rational surface with blowdown structure $\Gamma$ and
  $K_X^2<8$.  Then any effective simple root $\sigma$ can be decomposed as a
  nonnegative linear combination of $-d$-curves with $d>0$.  There is at
  most one fixed component of $\sigma$ not orthogonal to $\sigma$, and that component
  has self-intersection $\le -2$, with equality only if $\sigma$ is a $-2$-curve.
\end{lem}

\begin{proof}
  If $\sigma=e_i-e_{i+1}$ is effective, then it is the total transform of a $-2$
  curve on $X_{i+1}$.  At each later step in the blowing up process, either
  we blow up a point not on the total transform, in which case the
  decomposition is unchanged, or we blow up a point on the total
  transform, in which case we acquire an additional component $e_j$, and
  the component(s) containing the center of the monoidal transform have
  their self-intersection decreased by $1$.  Thus by induction every
  component is a $-d$-curve for some $d>0$.  We also see that $\sigma$ is
  uniquely effective, so every component is fixed, but only the strict
  transform of the original $-2$ curve is not orthogonal to $\sigma$.  The case
  $\sigma=f-e_1-e_2$ follows by elementary transformation.
  
  For the remaining case, assume for convenience that $\Gamma$ is even, so
  the remaining root is $s-f$.  This is effective precisely when $X_0\cong
  F_{2d}$ with $d>0$, when we can write it in the form
\[
s-f = s_{\min} + (d-1)f.
\]
This same decomposition applies to $X_m$, so that the only fixed components
are those of the total transform of $s_{\min}$, to which the previous
calculation applies.  On the other hand, we could choose the fibers in this
decomposition to all pass through the point blown up on $X_1$, obtaining a
decomposition
\[
s-f = s_{\min} + (d-1)(f-e_1)+(d-1)e_1,
\]
or
\[
s-f = (s_{\min}-e_1) + (d-1)(f-e_1)+de_1
\]
on $X_1$, the latter when the point being blown up is on $s_{\min}$.  The
components of this decomposition are all $-e$-curves for varying $e>0$, and
as before this property is preserved on taking the total transform to $X_m=X$.
\end{proof}

Since reflections in ineffective simple roots take blowdown structures to
blowdown structures, we can define a groupoid (the {\em strict} groupoid of
blowdown structures on $X$) as follows: the objects are the blowdown
structures on $X$, while the morphisms are given by the actions of
ineffective elements of $W(E_{m+1})$.

\begin{thm}
If $X$ is a rational surface with $K_X^2<8$, then the strict groupoid of
blowdown structures on $X$ has precisely two isomorphism classes, one for
each parity of blowdown structure.  (If $K_X^2=8$, the groupoid has only
one isomorphism class.)
\end{thm}

In other words, any two blowdown structures on $X$ with the same parity are
related by a sequence of reflections in ineffective simple roots.  For
$K_X^2<8$, elementary transformations imply that we need only consider the
even parity case.

\begin{proof}
  If $K_X^2=8$, this is obvious, as either there is only one blowdown
  structure or $X\cong \P^1\times \P^1$, and there are two blowdown
  structures related by reflection in $s-f$.  Now, suppose $K_X^2=7$, and
  fix an even blowdown structure on $X$.  To blow $X$ down to a Hirzebruch
  surface, we must blow down a $-1$ curve, which in particular gives a
  divisor class $D$ such that $D^2=D\cdot K_X=-1$.  There are only three
  such divisor classes, namely $D\in \{e_1,s-e_1,f-e_1\}$.  Only the case
  $D=e_1$ could blow down to an even Hirzebruch surface, since the other
  two cases have classes of odd self-intersection in their orthogonal
  complements.  In other words, a surface with $K_X^2=7$ blows down to a
  unique even Hirzebruch surface, and thus the even blowdown structures on
  $X$ are bijective with the even blowdown structures on this Hirzebruch
  surface.  (Note also that $f-e_1$ is always a $-1$ curve, while $s-e_1$,
  related to it by a simple reflection, is either a $-1$ curve or
  decomposes as $s-e_1=(s-f)+(f-e_1)$, depending on whether $s-f$ is
  effective.)

  For $K_X^2<7$, we may induct on $m$, and thus reduce to the question of
  showing that every $-1$ curve on $X$ can be moved to $e_m$ by a sequence
  of reflections in ineffective simple roots.  Again, we may assume we are
  starting from an even blowdown structure, conjugating by elementary
  transformations as appropriate.  Let
\[
E = n s + d f - \sum_i r_i e_i
\]
be the class of the given $-1$ curve.  Note that $n\ge 0$ since $f$ is nef
(it always has a representative which is smooth of self-intersection $0$).

If $E\cdot \sigma<0$ for some simple root, then the root must be ineffective,
since otherwise $E$ would be a fixed component of self-intersection $-1$
not orthogonal to the simple root.  In particular, we can always perform
the corresponding reflection, and this makes the vector
$(n,d,-r_1,\dots,-r_{m-1})$ lexicographically smaller.  Since the
reflections preserving $n$ form a finite group (of type $D_m$), we conclude
that after finitely many reflections in ineffective simple roots, we will
obtain a divisor such that $E\cdot \sigma\ge 0$ for every simple root $\sigma$.
We may also assume $E\cdot e_m\ge 0$, since otherwise $E=e_m$.

For such a divisor, we have the inequalities
\[
d\ge n\ge r_1+r_2;\quad r_1\ge r_2\ge\cdots\ge r_m\ge 0.
\]
Since $(E-e_m)\cdot K_X = 0$, we may also express $E-e_m$ as a linear
combination of simple roots:
\[
E-e_m = n(s-f) + (n+d)(f-e_1-e_2)
+(n+d-r_1)(e_1-e_2)+\!
\sum_{2\le k\le m-1}\! (1+\sum_{k<i} r_i)(e_i-e_{i+1}),
\]
clearly a nonnegative linear combination.  Since
\[
E\cdot (E-e_m) = -1 - r_m <0,
\]
we obtain a contradiction.
\end{proof}

\begin{rems}
We can adapt this to an algorithm for testing whether a given class is a
$-1$-curve (and thus whether a given numerical blowdown structure comes
from an actual blowdown structure): reflect in simple roots with $E\cdot
\sigma<0$ until one of the roots is effective, $E\cdot f<0$, or $E=e_m$.  Then
$E$ is a $-1$-curve iff the last termination condition holds.
\end{rems}

\begin{rems}
Of course, we could obtain a groupoid with a single isomorphism class by
including morphisms of the form $w\epsilon$ where $\epsilon$ is the
elementary transformation, but this is somewhat inconvenient, since the
morphisms no longer correspond directly to elements of a group.
\end{rems}

Any $-2$-curve on $X$ is a (real) root of the root system $E_{m+1}$.  More
precisely, we have the following.

\begin{prop}
  Suppose $D$ is the class of a $-2$ curve.  Then there exists a blowdown
  structure on $X$ for which $D$ is a simple root.
\end{prop}

\begin{proof}
Let $\Gamma$ be an even blowdown structure on $X$, and write
\[
D = ns + df - \sum_{1\le i\le m} r_i e_i
\]
as before.  Since $D\cdot K_X=0$, we can expand $D$ as a linear combination
of simple roots
\[
D = n(s-f) + (n+d)(f-e_1-e_2)
+(n+d-r_1)(e_1-e_2)+\!
\sum_{2\le k\le m-1}\! (\sum_{k<i} r_i)(e_i-e_{i+1}),
\]
and thus find as before that $D\cdot \sigma<0$ for some simple root $\sigma$.
Once more, $D$ would have to be a fixed component of $\sigma$ if $\sigma$ were
effective, and thus either $D=\sigma$ or $\sigma$ is ineffective.  As before, in the
latter case, reflecting makes $D$ lexicographically smaller, so this
process must terminate.
\end{proof}

\begin{rem}
  Again, this translates to an algorithm for testing whether a given
  divisor class is represented by a $-2$-curve, which is formally very
  similar to the algorithm of \cite{KatzNM:1996} for testing whether a
  local system is rigid.
\end{rem}

In the case of a surface with a chosen anticanonical curve, there is a
related groupoid with more morphisms and nontrivial stabilizers.  Call a
simple reflection $s\in S(E_{m+1})$ {\em admissible} for the blowdown
structure $\Gamma$ (and the anticanonical curve $C_\alpha$) if the
corresponding simple root is either ineffective or has intersection 0 with
every component of $C_\alpha$.  Although $s\Gamma$ is no longer a blowdown
structure when $s$ is effective but admissible, we define a modified action
as follows.  If $s$ is ineffective, then $s\cdot \Gamma:=s\Gamma$, while if
$s$ is effective but admissible, then $s\cdot \Gamma:=\Gamma$.

More generally, call an element $w\in W(E_{m+1})$ admissible for $\Gamma$
if there exists a word
\[
w = s_1 s_2 \cdots s_l
\]
for $w$ such that for each $1\le i\le l$, $s_i$ is admissible for
\[
s_{i+1}\cdot s_{i+2}\cdots s_{l-1}\cdot s_l\cdot \Gamma.
\]
In this case, we also call the given word admissible.

\begin{prop}
If $w$ is admissible for $\Gamma$, then every reduced word for $w$ is
admissible for $\Gamma$.  Moreover, if
\[
w = s_1\cdots s_l = s'_1\cdots s'_{l'}
\]
are two admissible words representing $w$, then
\[
s_1\cdot s_2\cdots s_{l-1}\cdot s_l\cdot \Gamma
=
s'_1\cdot s'_2\cdots s'_{l'-1}\cdot s'_{l'}\cdot \Gamma.
\]
\end{prop}

\begin{proof}
  We first note that if $s$ is admissible for $\Gamma$, then it is also
  admissible for $s\cdot \Gamma$; either $s$ is ineffective and remains so,
  or $s$ is a $-2$ curve, and $s\cdot\Gamma=\Gamma$.  Either way, we find
  $s\cdot s\cdot\Gamma = \Gamma$.  In other words, if a reflection occurs twice
  in a row in an admissible word, we can remove the pair without affecting
  admissibility or the final blowdown structure.

  Since any word can be transformed into a reduced word by a sequence of
  braid relations and removal of repeated reflections, and any two reduced
  words are related by a sequence of braid relations, it remains only to
  show that the claim holds for braid relations.  In other words, given a
  braid relation in $W(E_{m+1})$, we need to show that either both sides
  are inadmissible or both sides are admissible and produce the same
  blowdown structure.

  Let $s$, $t$ be simple reflections.  If both are inadmissible, there is
  nothing to prove, so suppose that $s$ is admissible for $\Gamma$.  Then we
  observe that if $t$ is inadmissible for $\Gamma$, then it is also
  inadmissible for $s\cdot\Gamma$.  Thus only the case that $s$ and $t$ are
  both admissible need be considered.  If the braid relation is $st=ts$,
  then we need merely check that the relation holds in each of the four
  cases ($s$ effective or not, $t$ effective or not).

  Thus suppose the braid relation is $sts=tst$.  Let $r_s$, $r_t$ be the
  corresponding simple roots.  If $r_s+r_t$ is ineffective, then
  either $r_s$, $r_t$ are both ineffective (and the braid relation follows
  from the fact that the action agrees with the linear action) or precisely
  one (say $r_t$) is effective.  But then we find that $r_s$ is effective
  in the blowdown structure $t\cdot s\cdot \Gamma$, and thus
\[
s\cdot t\cdot s\cdot \Gamma = t\cdot s\cdot t\cdot\Gamma = t\cdot s\cdot \Gamma.
\]
If $r_s+r_t$ is effective and $r_t$ is effective, then $r_s$ is also
effective.  Indeed, $(r_s+r_t)\cdot r_t=-1$, and thus any representative of
$r_s+r_t$ contains $r_t$ as a component, implying $r_s+r_t-r_t$ effective.
Thus in this case, we have
\[
s\cdot t\cdot s\cdot \Gamma = t\cdot s\cdot t\cdot\Gamma = \Gamma,
\]
since the blowdown structure never changes.

  Finally, we have the case $r_s+r_t$ effective but $r_s$, $r_t$ are
  ineffective.  Relative to the blowdown structure $t\cdot \Gamma$, $r_s$ is
  effective, and thus
\[
s\cdot t\cdot s\cdot t\cdot\Gamma
=
t\cdot s\cdot t\cdot t\cdot\Gamma
=
t\cdot s\cdot \Gamma
\]
(or both sides are undefined, if $r_s$ is inadmissible for $t\cdot \Gamma$).
If both sides are defined, then $s$ is admissible for both blowdown
structures, and thus
\[
t\cdot s\cdot t\cdot \Gamma
=
s\cdot s\cdot t\cdot s\cdot t\cdot\Gamma
=
s\cdot t\cdot s\cdot t\cdot t\cdot\Gamma
=
s\cdot t\cdot s\cdot \Gamma,
\]
and we are done.
\end{proof}

If we use admissible elements in place of effective elements in defining
the groupoid of blowdown structures, the resulting ``weak'' groupoid has
nontrivial stabilizers, a conjugacy class of reflection subgroups of
$W(E_{m+1})$.

\begin{prop}
  The stabilizer of $\Gamma$ in the weak groupoid of blowdown structures is
  the reflection subgroup of $W(E_{m+1})$ generated by reflections in
  $-2$-curves disjoint from the anticanonical curve.
\end{prop}

\begin{proof}
Given a $-2$-curve $v$ disjoint from $C_\alpha$, let $w$ be an effective element
of $W(E_{m+1})$ such that $v$ is a simple root $\sigma$ in $w\Gamma$.  Then 
$r_\sigma$ stabilizes $w\Gamma$, so $r_v=w^{-1} r_\sigma w$ stabilizes
$\Gamma$.

Conversely, consider an admissible reduced word $w$ stabilizing $\Gamma$.
If every reflection in $w$ is ineffective, then $w$ acts linearly, and
since it stabilizes a basis, we have $w=1$.  Otherwise, we can write
\[
w = w_1 r w_2
\]
where $r$ is an effective but admissible simple reflection, $w_2$ is
ineffective, and $\ell(w)=\ell(w_1)+\ell(w_2)+1$.  Since $r$ is effective,
it stabilizes $w_2\cdot \Gamma=w_2\Gamma$.  We thus conclude that we can factor
\[
w = w_1 w_2 (w_2^{-1}r w_2)
\]
where both $w_1w_2$ and $w_2^{-1}r w_2$ are admissible elements stabilizing
$\Gamma$.  The second factor is a reflection in the $-2$-curve
corresponding to $r$ in $w_2\Gamma$ (which is admissible, so disjoint from
$C_\alpha$), while the first factor has length strictly smaller than
$\ell(w)$.  Thus by induction, $w$ can be written as a product of
reflections in $-2$-curves disjoint from $C_\alpha$.
\end{proof}

Note from \cite{HarbourneB:1997} that the $-2$-curves disjoint from
$C_\alpha$ can be determined in the following way: restriction to $C_\alpha$
gives a natural homomorphism $\Pic(X)\to \Pic(C_\alpha)$, and the $-2$-curves
are precisely the simple roots in the system of positive roots in the
kernel of this homomorphism.  This is easy to see from our perspective, as
it reduces to checking when a simple root of $E_{m+1}$ is a $-2$-curve
disjoint from $C_\alpha$.

Given an anticanonical rational surface $X$, there is a natural
combinatorial invariant of blowdown structures, namely how the components
of $C_\alpha$ (which we fix an ordering of) are expressed in terms of the
corresponding basis.  That is, if we fix an ordered decomposition
\[
C_\alpha = \sum_i c_i C_i
\]
where the $C_i$ are the distinct components of $C_\alpha$ (so each $c_i>0$),
then given any blowdown structure, we may associate the sequence of pairs
$(c_i,v_i)$ where $v_i\in \Z^{m+2}$ is the image of $C_i\in \Pic(X)$ under
the isomorphism $\Pic(X)\cong \Z^{m+2}$ corresponding to $\Gamma$.  (If
$C_\alpha$ is integral, this invariant simply distinguishes between even and
odd blowdown structures.)  If we fix $X$ and a decomposition of $C_\alpha$,
this invariant takes on only finitely many values as we vary $\Gamma$.  In
fact, something much stronger holds: if we take the union over all
anticanonical surfaces with a chosen decomposition $C_\alpha$, then there are
only finitely many possibilities for any given value of $\min_i(v_i^2)$.

Indeed, if we put a lower bound on the self-intersections of the components
of $C_\alpha$, then this implies a lower bound on the self-intersections of
any $-d$-curve on $X$ (since any $-d$-curve with $d>2$ has negative
intersection with $C_\alpha$, so is a component).  In particular, this gives
only finitely many possible Hirzebruch surfaces that $X$ can be blown down
to.  On a given Hirzebruch surface, there are only finitely many
combinatorially distinct decompositions of anticanonical curves, and as we
blow up points, the change in invariant only depends on the set of
components containing the point being blown up.

In particular, this combinatorial type splits the weak groupoid of blowdown
structures into finitely many groupoids.  We observe that each of these
groupoids is a quotient groupoid $G/H$ for some $G\subset W(E_{m+1})$,
where $H$ is the group generated by reflections in $-2$-curves disjoint
from $C_\alpha$.  Indeed, whether a simple root is admissible only depends on
the combinatorial type, and thus the admissible elements of $W(E_{m+1})$
preserving the combinatorial type form a group.  This group is certainly
contained in the stabilizer of the sequence of vectors corresponding to the
components of $C_\alpha$, and itself contains a reflection group.

\begin{prop}
  Suppose $\rho$ is a positive root which is orthogonal to every component of
  $C_\alpha$.  Then the corresponding reflection is admissible.
\end{prop}

\begin{proof}
  It suffices to consider the case that $\rho$ is {\em simple} among the
  root system of positive roots orthogonal to every component of
  $C_\alpha$.  Then we claim that there is a blowdown structure in which
  $\rho$ is a simple root of $E_{m+1}$.  If $\rho$ is already simple, this
  is immediate.  Otherwise, let $\sigma$ be a simple root such that
  $\sigma\cdot \rho<0$.  If $\sigma$ is ineffective, we may reflect in
  $\sigma$ and proceed by induction.  Otherwise, $\sigma$ is effective, and
  some component $c$ of $\sigma$ which is not a component of $C_\alpha$
  satisfies $c\cdot \rho<0$.  Since $f\cdot \rho\ge 0$ for every positive
  root, we have $c\ne f$, and thus by the proof of Lemma
  \ref{lem:fixed_of_simple}, $c$ must be a fixed component of $\sigma$.  If
  $c^2=-2$, then it is orthogonal to every component of $C_\alpha$, but
  then the fact that $c\cdot \rho<0$ contradicts simplicity of $\rho$
  unless $\rho=c$, in which case reducing to a simple root of $E_{m+1}$ is
  straightforward.

  Otherwise, by the classification of fixed components of $-2$-curves, we
  find that $c=e_i$ for some $i$, and that $e_j\cdot \sigma=0$ for $j\ge
  i$.  Since the only positive roots satisfying $\rho\cdot e_i<0$ are those of
  the form $e_i-e_j$ for some $j>i$, we conclude that $\rho\cdot
  \sigma=(e_i-e_j)\cdot \sigma=0$, a contradiction.
\end{proof}

\begin{rem}
  In general, the full stabilizer need not be a reflection subgroup.  For
  instance, let $X=X_8$ be a rational elliptic surface with an
  anticanonical curve of Kodaira type $I_3^*$ (corresponding to the root
  system $\tilde{D}_7$).  Since a subsystem of type $D_7$ in $E_8$ has
  trivial stabilizer, we conclude that the stabilizer must be contained in
  the translation subgroup of $W(E_9)=W(\tilde{E}_8)$ in this case.  Any
  translation will add some multiple of $-K_X$ to the different components,
  preserving the property that the relevant linear combination is $-K_X$;
  it follows that the stabilizer contains a corank $7$ subgroup of the
  translation subgroup of $E_8$.  In other words, the stabilizer is
  isomorphic to $\Z$, so is certainly not a reflection subgroup!  (We can
  also directly verify in this case that the generator of the stabilizer is
  admissible.)  The stabilizer also fails to be a reflection subgroup of
  $W(E_9)$ when the anticanonical curve has Kodaira type $I_7$, and in one
  of the two ways it can have Kodaira type $I_8$.  In the $I_8$ case it is
  again isomorphic to $\Z$, while in the $I_7$ case it is isomorphic to
  $W(\tilde{A}_1)\times \Z$.
\end{rem}

\begin{cor}
  If $X$ is a rational surface with an integral anticanonical curve, then
  the weak groupoid of blowdown structures on $X$ is a union of two
  isomorphic quotient groupoids of the form $W(E_{m+1})/H$ where $H$ is the
  group generated by reflections in the $-2$ curves of $X$.
\end{cor}

\begin{rem}
  The reader should be cautioned that a reflection subgroup of an infinite
  Coxeter group need not have finite rank.  Indeed, an example was given in
  \cite[Ex.~2.8]{HarbourneB:1985} of a rational surface with (nodal)
  integral anticanonical curve and infinitely many $-2$-curves, and thus
  the stabilizers in the corresponding groupoid have infinite rank.
\end{rem}

\section{Divisors on rational surfaces}

\subsection{Nef divisors}

Given a rational surface and blowdown structure, one natural question which
arises is whether a given vector corresponds to an effective divisor class,
or one with an integral representative.  For the latter, it will be helpful
to also have an answer to the question of which vectors correspond to
nef divisor classes.  This is complicated in general, but
in the anticanonical case, is quite tractable.  Note that as with the above
algorithms for recognizing $-2$- and $-1$-curves, the algorithms below only
depend on (a) the decomposition of $C_\alpha$ in some initial choice of
blowdown structure, and (b) the kernel of the natural homomorphism
$\Pic(X)\to \Pic(C_\alpha)$.  (The latter is not particularly tractable to
compute in general, but in fact all we really need is the ability to test
membership in the kernel.)

One answer to this question was given in \cite{LahyaneM/HarbourneB:2005}:
the monoid of effective divisors (assuming $K_X^2<8$) is generated by the
$-d$-curves with $d>0$ and $-K_X$.  In principle, this gives a way of
testing whether a vector is nef: simply check that it has nonnegative
intersection with $-K_X$ and every $-d$-curve.  Of course, this is not an
actual algorithm, for the simple reason that a rational surface can have
infinitely many smooth curves of negative self-intersection.  (In fact,
this is the {\em typical} behavior!)

We do, however, obtain the following.  Given an anticanonical rational
surface $X$ and a blowdown structure $\Gamma$ for $X$, define the {\em
  fundamental chamber} to be the monoid in $\Pic(X)$ consisting of classes
having nonnegative intersection with every simple root for $\Gamma$.
Note also that there are only finitely many $-d$-curves on $X$ with $d>2$,
since any such curve must be a component of $-K_X$.

\begin{prop}
  Suppose $X$ is an anticanonical rational surface with $K_X^2\le 6$, and
  let $\Gamma$ be a blowdown structure on $X$.  Let $D$ be a divisor class
  in the fundamental chamber of $\Gamma$.  Then $D$ is nef iff $D\cdot
  C_\alpha\ge 0$, $D\cdot e_m\ge 0$, and $D$ has nonnegative
  self-intersection with every $-d$-curve with $d>2$.  If $C_\alpha^2\ge
  0$, then we can omit the condition $D\cdot C_\alpha\ge 0$.
\end{prop}

\begin{proof}
  That nef divisors satisfy the given conditions is
  obvious, so it remains to show that the given conditions imply that $D$
  is nef.  Since we have assumed $D\cdot C_\alpha\ge 0$, it remains to verify
  that it has nonnegative intersection with every $-d$-curve for $d>0$.  We
  have also assumed this for $d>2$, so only the cases of $-2$- and
  $-1$-curves remain.  Any $-2$-curve is a positive root of $E_{m+1}$, so
  is a nonnegative linear combination of simple roots.  By assumption, $D$
  has nonnegative intersection with every simple root, so nonnegative
  intersection with every positive root.  Similarly, we saw above that any
  $-1$-curve can be written as $e_m$ plus a nonnegative linear combination
  of simple roots, so again $D$ has nonnegative intersection with every
  $-1$-curve.

  If $K_X^2\ge 0$, then we can write $K_X$ as a nonnegative linear combination
  of simple roots and $e_m$, so can omit the corresponding condition.
\end{proof}

\begin{rem}
  Something similar holds when $K_X^2=7$, except that we must also assume
  $D\cdot (f-e_1)\ge 0$.  In any event, we can readily write down the
  effective and nef monoids when $K_X^2=7$.  Indeed, if $X_0\cong F_{2d}$ and
  $X_1\to X_0$ blows up a point of $s_{\min}$ (which we arrange to occur if
  $d=0$), we have
\begin{align}
\text{Eff}(X_1) &= \langle s-df-e_1,f-e_1,e_1\rangle\notag\\
\text{Nef}(X_1) &= \langle f,s+df,s+(d+1)f-e_1\rangle,\notag
\end{align}
while if $d>0$ and $X_1\to X_0$ blows a point not on $s_{\min}$, we have
\begin{align}
\text{Eff}(X_1) &= \langle s-df,f-e_1,e_1\rangle\notag\\
\text{Nef}(X_1) &= \langle f,s+df,s+df-e_1\rangle.\notag
\end{align}
Since $s-df=s_{\min}$, it is clear that the putative generators for
$\text{Eff}(X_1)$ are effective.  Similarly, we find that the putative
generators for $\text{Nef}(X_1)$ have nonnegative self-intersection, and
can be represented by integral divisors, so are nef.  Since in each case
the two bases are dual to each other, they must actually be the effective
and nef monoids.  (The corresponding bases for the monoids relative to an
odd blowdown structure can of course be obtained by an elementary
transformation.)  Note that in either case, $\text{Eff}(X_1)$ is a
simplicial cone generated by $-e$-curves with $e<0$.
\end{rem}

Since nef divisors are effective (\cite[Cor.~II.3]{HarbourneB:1997}), we
also conclude that any class $D$ satisfying the above hypotheses is
effective.  In fact, we can do better: we can give an explicit effective
divisor representing $D$.

\begin{prop}\label{prop:decomp_nef}
With hypotheses as above, $D$ can be written as a nonnegative linear
combination of $-d$-curves and $-K_X$.
\end{prop}

\begin{proof}
Suppose first that $D\cdot e_m=0$, so that $D$ is the total transform of a
divisor on $X_{m-1}$.  If $m>2$, then this divisor on $X_{m-1}$ is itself
in the fundamental chamber, and has nonnegative intersection with
$-K_{X_{m-1}}$.  We can thus decompose it into $-d$-curves and copies of
the anticanonical curve on $X_{m-1}$.  The total transform of a $-d$-curve
is either a $-d$-curve or the sum of a $-(d+1)$-curve and $e_m$, while the
total transform of the anticanonical curve on $X_{m-1}$ is $C_\alpha+e_m$,
and thus the decomposition on $X_{m-1}$ induces a decomposition on $X_m$
which again has the desired form.

Similarly, if $m=2$ and $D\cdot e_2=0$, then we still find that $D$ is the
total transform of a nef divisor on $X_1$, and thus obtain the desired
decomposition of $D$ by expanding it in the basis of the simplicial cone
$\text{Eff}(X_1)$.

Finally, suppose $D\cdot e_m>0$.  Then we claim that $D-C_\alpha$ satisfies the
original hypotheses.  Indeed, if $C$ is a $-d$-curve for $d>2$, then
\[
(D-C_\alpha)\cdot C = D\cdot C + (d-2)> D\cdot C,
\]
while if $\sigma$ is a simple root, then $(D-C_\alpha)\cdot \sigma = D\cdot
\sigma\ge 0$.  In addition, $(D-C_\alpha)\cdot e_m = D\cdot e_m-1\ge 0$.
Finally, we have
\[
(D-C_\alpha)\cdot C_\alpha = D\cdot C_\alpha - C_\alpha^2.
\]
Either $C_\alpha^2<0$, so the inequality becomes stronger, or $C_\alpha^2\ge
0$, and the inequality is redundant.  Either way, $D-C_\alpha$ satisfies all
of the hypotheses, and we obtain an explicit decomposition of the given form.
\end{proof}

\begin{lem}
If $D$ is a nef divisor on the rational surface $X$, then there exists a
blowdown structure (of either parity) such that $D$ is in the fundamental
chamber.  Moreover, the representation of $D$ in the basis corresponding to
such a blowdown structure depends only on the parity.  In addition, if $e$
is a $-1$-curve with $e\cdot D=0$, then the blowdown structure can be chosen
in such a way that $e_m=e$.
\end{lem}

\begin{proof}
  Choose a blowdown structure of the desired parity on $X$.  If $D$ is not
  already in the fundamental chamber, then there exists a simple root
  $\sigma$ such that $D\cdot \sigma<0$.  Since $D$ is nef, $\sigma$ cannot
  be effective, and thus we can apply the corresponding reflection.  Either
  $\sigma$ is in the subsystem of type $D_m$ (which can only occur finitely
  many times in a row, since that subgroup is finite), or it decreases
  $D\cdot f$.  The latter is nonnegative since $f$ is effective, and thus
  the process will terminate after a finite number of steps.

For uniqueness, suppose $D$ is in the fundamental chamber of both $\Gamma$
and $\Gamma'$, two blowdown structures of the same parity.  Then there
exists an ineffective element $w\in W(E_{m+1})$ such that
$\Gamma'=w\Gamma$.  If $w=1$, then we are done; otherwise, there is an
ineffective simple root $\sigma$ of $\Gamma$ such that $w\sigma$ is
negative (the last root in some reduced word for $w$).  Since $D$ is in the
fundamental chamber for both $\Gamma$ and $\Gamma'$, it has nonnegative
intersection with every positive root of either blowdown structure.  Thus
$D\cdot \sigma\ge 0$ since $\sigma$ is positive for $\Gamma$, and $D\cdot
(-\sigma)\ge 0$ since $-\sigma$ is positive for $\Gamma'$.  In other words,
$D\cdot \sigma=0$, and thus the reflection in $\sigma$ does not change the
expansion of $D$ in the standard basis.  The claim follows by induction on
the length of the reduced word for $w$.

Finally, if $\Gamma$ is any blowdown structure with $e_m=e$, then $D\cdot
(e_{m-1}-e_m)\ge 0$, and thus the algorithm for putting $D$ in the
fundamental chamber will never try to apply the corresponding reflection,
so will never change $e_m$.
\end{proof}

\begin{rem}
  Uniqueness is of course a standard fact from Coxeter theory when we
  restrict to $D$ in the root lattice, and the above argument is adapted
  from the standard one.
\end{rem}

This then gives us the desired algorithm for testing whether a divisor is
nef: First check that it has nonnegative intersection with every $-d$-curve
with $d>2$ and with $C_\alpha$, then repeatedly attempt to reflect in simple
roots with $D\cdot \sigma<0$.  If at any step we have $\sigma$ effective, $D\cdot
f<0$ or $D\cdot e_m<0$, then $D$ is not nef; otherwise, we terminate in the
fundamental chamber, and conclude that $D$ is nef.

\subsection{Effective divisors}

A similar algorithm works for testing whether a divisor $D$ is effective.
We assume $K_X^2<7$, since otherwise the effective cone is simplicial, so
testing whether $D$ is effective is just linear algebra.

Again, we start by choosing any blowdown structure for $X$, and if at any
step in the process we obtain a divisor with $D\cdot f<0$, we halt with the
conclusion that our divisor was ineffective.  We perform the following
steps, as specified.
\begin{itemize}
\item[1.] If there exists a component $C$ of $C_\alpha$ such that
  $C^2,D\cdot C<0$, then replace $D$ by $D-C$, and repeat step 1.
\item[2.] If $D\cdot e_m<0$, then replace $D$ by $D+(D\cdot e_m)e_m$
  and go back to step 1.
\item[3.] If $D$ is in the fundamental chamber, conclude that the
  original divisor was effective.  Otherwise, choose the lexicographically
  smallest simple root such that $D\cdot \sigma<0$.  If $\sigma$ is effective,
  replace $D$ by $D-\sigma$ and go back to step 1; otherwise, replace $\Gamma$
  by $r_\sigma\Gamma$ and go back to step 2.
\end{itemize}

To see that this algorithm works, we note as before that $f$ is nef, so any
divisor with $D\cdot f$ is not effective.  Whenever we replace $D$ by $D-C$
in step 1, $C$ is an integral curve of negative self-intersection
intersecting $D$ negatively.  But then $D$ is effective iff $D-C$ is
effective; one direction is obvious, while if $D$ is effective, then $C$ is
a fixed component of $D$.  The same argument applies in step 2, while in
step 3, either $\sigma$ is irreducible (so again the argument applies) or
we have $D\cdot c<0$ for some fixed component $c$ of $\sigma$.  We must
have $c^2\ge -2$, else $c$ would have been removed in step $1$; and
similarly $c\ne e_m$.  But then the classification of fixed components of
effective simple roots lets us find a lexicographically smaller simple root
having negative intersection with $D$.

Since we terminate at a nef divisor in the fundamental chamber, this
algorithm also gives us an explicit decomposition of $D$ as a nonnegative
linear combination of $C_\alpha$ and $-d$-curves.  In this context, we note
the following.

\begin{prop}
  Let $X$ be an anticanonical rational surface with $K_X^2<8$.  Either
  every representative of $-K_X$ is integral, or some representative is a
  nonnegative linear combination of $-d$-curves with $d\ge 1$.
\end{prop}

\begin{proof}
If some representative of $-K_X$ is reducible, then we can write
\[
-K_X=D_1+D_2
\]
for nonzero effective divisors $D_1$, $D_2$, and it suffices to show that
each $D_i$ is linearly equivalent to a nonnegative linear combination of
$-d$-curves.  Since $D_1$ is effective by assumption, we can write
\[
D_1 \sim m(-K_X) + \sum_j c_j C_j
\]
where each $C_j$ is a $-d$-curve for some $d\ge 1$ and all coefficients are
nonnegative.  This implies $D_1+mK_X$ is effective, and thus $-D_2=D_1+K_X =
D_1+mK_X+(m-1)C_\alpha$ is effective, unless $m=0$.  In other words, $D_1$
has a decomposition as required.
\end{proof}

\begin{cor}
  Let $X$ be an anticanonical rational surface, and suppose $K_X^2\notin
  \{0,1,8,9\}$.  Then the effective monoid of $X$ is generated by the
  integral curves of negative self-intersection.
\end{cor}

\begin{proof}
  If $K_X^2<0$, then $C_\alpha$ has negative self-intersection, and is either
  integral or redundant.  For $1<K_X^2<8$, we note that $\Gamma(-K_X)$
  corresponds to a codimension $m$ subspace of $\Gamma(-K_{X_0})$.  On a
  Hirzebruch surface, either every anticanonical curve is reducible (i.e.,
  on $F_d$ for $d>2$), or the reducible anticanonical curves are
  codimension $2$ subvariety of the $8$-dimensional projective space of all
  anticanonical curves.  We are imposing $m\le 6$ linear conditions on this
  projective variety, and thus obtain a nonempty set of anticanonical
  curves on $X$ which are reducible on $X_0$ and thus reducible on $X$.
\end{proof}

\begin{rem}
  Similarly, if $K_X^2=1$ but $X$ has a $-2$-curve, then some anticanonical
  curve is reducible.  Also, in any case the anticanonical divisor is not
  needed to generate the {\em rational} effective cone when $K_X^2=1$,
  since then $-2K_X = (-2K_X-e_7)+e_7$ is a sum of effective divisors.
\end{rem}

We can also adapt the algorithm to compute $h^0(\sO_X(D))$ for an
effective divisor.  Indeed, every step of the algorithm removes a fixed
component of $D$, and thus the resulting nef divisor $D'$ has a natural
isomorphism 
\[
H^0(\sO_X(D'))\cong H^0(\sO_X(D)).
\]
Thus to compute the dimensions of effective linear systems, it remains only
to compute the dimensions of linear systems corresponding to nef divisors
in the fundamental chamber.  So let $D$ be such a divisor class and suppose
$m=0$ or $D\cdot e_m>0$, since otherwise we may as well consider $D$ as a
divisor on $X_{m-1}$.

If $D\cdot C_\alpha>0$, then it follows from
\cite[Thm.~III.1(ab)]{HarbourneB:1997} that $h^1(\sO_X(D))=0$, and thus
we can use Hirzebruch-Riemann-Roch to compute
\[
h^0(\sO_X(D)) = \chi(\sO_X(D)) = \frac{D\cdot (D+C_\alpha)}{2}+1.
\]
This in particular holds whenever $m<8$, since then either $D\cdot
C_\alpha>0$ or $D=0$.

If $m\ge 8$ and $D\cdot C_\alpha=0$, then from the proof of Proposition
\ref{prop:decomp_nef}, we find that $D-C_\alpha$ is also nef.  Now consider
the short exact sequence
\[
0\to \sO_X(D-C_\alpha)\to \sO_X(D)\to \sO_X(D)|_{C_\alpha}\to 0.
\]
From \cite[Thm.~III.1(d)]{HarbourneB:1997}, we find that the natural inclusion
\[
H^0(\sO_X(D-C_\alpha))\subset H^0(\sO_X(D))
\]
is an isomorphism iff the line bundle $\sO_X(D)|_{C_\alpha}$ is
nontrivial.  Since $h^0(\sO_{C_\alpha})=1$, we conclude that
\[
h^0(\sO_X(D))
=
\begin{cases}
h^0(\sO_X(D-C_\alpha))+1 & \sO_X(D)|_{C_\alpha}\cong \sO_{C_\alpha}\\
h^0(\sO_X(D-C_\alpha)) & \text{otherwise.}
\end{cases}
\]
If $m>8$, then $D-C_\alpha$ is a nef divisor with $(D-C_\alpha)\cdot C_\alpha>0$,
so we reduce to the previous case.  If $m=8$, then $D=r C_\alpha$ for some
$r\ge 1$, and thus $C_\alpha$ must be nef.  We deduce that either $C_\alpha$ is
integral or every component of $C_\alpha$ is a $-2$-curve.  Moreover, the
above recurrence tells us that in this case,
\[
h^0(\sO_X(rC_\alpha)) = \lfloor r/r'\rfloor +1,
\]
where $r'$ is the order of the bundle $\sO_X(C_\alpha)|_{C_\alpha}$ in
the group $\Pic(C_\alpha)$.  (In particular, $h^0(\sO_X(rC_\alpha))=1$ if this
bundle is not torsion.)

\begin{rem}
If
\[
D = ns+df-\sum_i r_i e_i
\]
relative to some even blowdown structure, then
\[
\chi(\sO_X(D))
=
(n+1)(d+1)-\sum_i \frac{r_i(r_i+1)}{2}.
\]
This of course corresponds to the fact that $H^0(\sO_X(D))$ is a
subspace of $H^0(\sO_X(ns+df))$ cut out by the appropriate number of
linear conditions.  (If $X$ blows up $m$ distinct points of $X_0$, the
conditions are simply that the curve have multiplicity $r_i$ at the $i$-th
point.)  In principle, one could determine $h^0(\sO_X(D))$ (and in
particular test whether $D$ is effective) using linear algebra, but the
above approach scales better, and largely separates out the combinatorial
influences from the algebraic influences.
\end{rem}

\subsection{Integral divisors}

By Lemma II.6 and Theorem III.1 of \cite{HarbourneB:1997}, there is a
relatively short list of possible ways that a nef class can fail to be
generically integral.  (The integral classes which are not nef are
precisely the $-d$-curves for $d\ge 1$ and the anticanonical divisor, when
this is integral and has negative self-intersection, and we already know
how to recognize those.)  Although the description given there is purely
geometric, it turns out to be easy enough to recognize the different cases
in terms of the representation of the divisor in a fundamental chamber.
Since this representation is unique, we can (and will) figure out how each
case is represented by placing various geometrically motivated constraints
on the blowdown structure, and checking that the result is in the
fundamental chamber.

\begin{rem}
  Note that in characteristic 0, ``generically integral'' and ``generically
  smooth'' are the same on an anticanonical rational surface: a generically
  integral divisor class on a rational surface has at most one base point,
  and if it does, meets $C_\alpha$ at that point with
  multiplicity 1.  Bertini's theorem implies that the generic
  representative is smooth away from the base point, and the intersection
  with $C_\alpha$ implies smoothness there.
\end{rem}

\begin{lem}
  Let $D$ be a divisor on $X$, and suppose $\Gamma$ is an even blowdown
  structure such that $D$ is in the fundamental chamber.  Then $D$ is a
  pencil iff one of the following three cases occurs.
\begin{itemize}
\item[(a)] $D=f$.
\item[(b)] $D=2s+2f-e_1-\cdots-e_7$, $D$ is nef, and
  $\sO_X(2s+2f-e_1-\cdots-e_7-e_k)|_{C_\alpha}\not\cong \sO_{C_\alpha}$ for
  $8\le k\le m$.
\item[(c)] $D=r(2s+2f-e_1-\cdots-e_8)$, where
$\sO_X(2s+2f-e_1-\cdots-e_8)|_{C_\alpha}$ is a line bundle of exact order
$r$ in $\Pic(C_\alpha)$.
\end{itemize}
\end{lem}

\begin{proof}
  A pencil is certainly generically integral (lest $X$ be reducible), so
  nef.  Per \cite[Lem.~II.6]{HarbourneB:1997}, there are three possibilities:
$D^2=0$, $D\cdot C_\alpha=2$; $D^2=D\cdot C_\alpha=1$; or $D^2=D\cdot
C_\alpha=0$.

In the first case, the generic fiber of $D$ has arithmetic genus 0, so $D$
is the class of a rational ruling.  It follows that there exists a blowdown
structure such that $D=f$, and we readily verify that $f$ is in the
fundamental chamber.  (If the blowdown structure we end at is odd, simply
perform an elementary transformation, and note that this preserves the
meaning of $f$.)

For the case $D^2=D\cdot C_\alpha=0$, $D$ gives a quasi-elliptic fibration of
$X$, and we can choose a blowdown structure in which we first blow down any
$-1$-curves contained in fibers.  After doing so, we end up at a relatively
minimal quasi-elliptic surface, which must be $X_8$ for the blowdown
structure.  The only isotropic vectors in $\Pic(X_8)$ are the multiples of
the canonical class, and thus $D=r(2s+2f-e_1-\cdots-e_8)$ for some $r$;
again, this is in the fundamental chamber.  For this to be a pencil, it
must not have any fixed component, so $\sO_X(D)|_{C_\alpha}\cong
\sO_{C_\alpha}$ and $\sO_X(-K_{X_8})|_{C_\alpha}$ has order dividing $r$.
If the order strictly divides $r$, then $\sO_X(D)$ will have more than
$2$ global sections.

For the case $D^2=D\cdot C_\alpha=1$, the linear system is again
quasi-elliptic, now with a base point.  The base point is on the
anticanonical curve, namely the unique point such that
\[
\sO_X(D)|_{C_\alpha}\cong {\cal L}_{C_\alpha}(p).
\]
The fibers of $D$ transverse to $C_\alpha$ are either integral or contain a
single $-1$-curve, while the fiber not transverse to $C_\alpha$ contains
$C_\alpha$.  The residual divisor $D-C_\alpha$ has arithmetic genus $1-r$,
where $r=C_\alpha^2-1$, and thus has at least $r$ connected components, each
of which has negative self-intersection (since it is orthogonal to a class
of positive self-intersection).  It follows that every component has
self-intersection $-1$ and arithmetic genus 0, and thus contains a
$-1$-curve.  We may thus choose a blowdown structure in which we first blow
down those $-1$-curves until eventually reaching $X_7$ and
$D=2s+2f-e_1-\cdots-e_7$.  This is a pencil on $X_7$ precisely when it is
nef, and remains a pencil on $X$ as long as we never blow up the base
point.
\end{proof}

\begin{prop}
  Suppose $D$ is a nef divisor class with $D\cdot C_\alpha\ge 2$,
  and let $\Gamma$ be an even blowdown structure for which $D$ is in the
  fundamental chamber.  Then $D$ is generically integral unless $D=rf$ for
  some $r>1$.
\end{prop}

\begin{proof}
  By \cite[Thm.~III.1(a)]{HarbourneB:1997}, $D$ is base point free, so is
  generically integral unless it is a strict multiple of a pencil.
\end{proof}

The case $D\cdot C_\alpha=0$, which is the most interesting for us in any
event, is the next easiest case to handle.  If $\sO_X(D)|_{C_\alpha}\not\cong
\sO_{C_\alpha}$, then $D$ can only be integral if $D=C_\alpha$ and
$C_\alpha$ is integral (\cite[Thm.~III.1(d)]{HarbourneB:1997}).

\begin{thm}\label{thm:non_integral}
  Let $X$ be an anticanonical rational surface, let $D$ be a nef divisor
  class on $X$ such that $\sO_X(D)|_{C_\alpha}\cong \sO_{C_\alpha}$, and let
  $\Gamma$ be an even blowdown structure such that $D$ is in the
  fundamental chamber.  Then $D$ is generically integral unless one of the
  following two possibilities occurs.
\begin{itemize}
\item[(a)] $D=r(2s+2f-e_1-\cdots-e_8)$, and
  $\sO_X(2s+2f-e_1-\cdots-e_8)|_{C_\alpha}$ is a line bundle of order $r'$
  strictly dividing $r$.  Then the generic representative of $D$ is a
  disjoint union of $r/r'$ curves of genus 1, of divisor class
  $r'(2s+2f-e_1-\cdots-e_8)$.
\item[(b)] $D=r(2s+2f-e_1-\cdots-e_8)+e_8-e_9$ with $r>1$, and
  $\sO_X(2s+2f-e_1-\cdots-e_8)|_{C_\alpha}\cong \sO_{C_\alpha}$.  Then the
  generic representative of $D$ is the union of $r$ divisors of class
  $2s+2-e_1-\cdots-e_8$ (all of genus 1) and a $-2$-curve of class
  $e_8-e_9$.
\end{itemize}
\end{thm}

\begin{proof}
  $D$ is generically integral unless it factors through a pencil or has a
  fixed component.  The first case is precisely (a) above, by the
  classification of pencils.  The second case is described in
  \cite[Thm.~III.1(c)]{HarbourneB:1997}: $D$ has a unique fixed component
  $N$ which is a $-2$-curve, and $D-N$ is a strict multiple of a pencil $P$
  with $P\cdot N=1$.  In particular, there exists a blowdown structure such
  that $P$ is the total transform of some antipluricanonical pencil on
  $X_8$, and ``pluri'' can be ruled out by the fact that $P\cdot N=1$.
  Now, $N$ cannot be contracted by the map $X\to X_8$, since $P$ is still
  base point free on $X_8$; thus $N$ is a rational curve, and since $N\cdot
  P=1$, must be a $-1$-curve.  We can thus further insist that the map
  $X_8\to X_7$ blows down $N$.  Since $N$ is a $-2$-curve on $X$, the map
  $X\to X_8$ blows up a point of $N$ exactly once, and we may insist that
  this is the first point blown up after reaching $X_8$; i.e., that $N$ is
  already a $-2$-curve on $X_9$.  But then $N=e_8-e_9$, and $D$ has the
  claimed form, which we verify is in the fundamental chamber.
\end{proof}

\begin{rems}
For multiplicative Deligne-Simpson problems, a rather more complicated
irreducibility condition was given in \cite{Crawley-BoeveyW/ShawP:2006}.  In
particular, the above result gives a much stronger statement in the case of
$3$- and $4$-matrix multiplicative Deligne-Simpson problems, and it is
natural to wonder if a similarly strong result holds in general.
\end{rems}

The remaining case $C_\alpha\cdot D=1$ can be dealt with in one of two ways.
The easiest is to blow up the intersection with $C_\alpha$, and consider the
strict transform $D'$ of $D$ on $X_{m+1}=:X'$, a divisor class which is
generically disjoint from the new anticanonical curve.  The above
algorithms tell us how to determine the generic decomposition of such a
divisor class: first use the algorithm for testing effectiveness to write
it as a sum of (fixed) $-2$-curves and a nef class in some fundamental
chamber, then use the above result to decompose the latter class.  The
generic decomposition of $D'$ on $X'$ corresponds directly to the generic
decomposition of $D$ on $X$, since $D'$ is a strict transform, and thus
this procedure computes the generic decomposition of $D$.

We can also work out what the nonintegral cases look like in the
fundamental chamber, again using a result of Harbourne,
\cite[Thm.~III.1(b)]{HarbourneB:1997}.  We omit the details.

\begin{prop}\label{prop:non_integral}
  Suppose $D$ is a nef divisor class on $X$ such that $D\cdot C_\alpha=1$,
  and let $\Gamma$ be a blowdown structure for which $D$ is in the
  fundamental chamber.  Then $D$ is generically integral except in the
  following two cases.
\begin{itemize}
\item[(a)] For some $1\le i\le m$, $D\cdot e_i=0$ and
$\sO_X(D-e_i)|_{C_\alpha}\cong \sO_{C_\alpha}$.  The fixed part of $D$ is
the total transform of the minimal such $e_i$.
\item[(b)] $D=r(2s+2f-e_1-\cdots-e_8)+e_8$ and
  $\sO_X(2s+2f-e_1-\cdots-e_8)|_{C_\alpha}\cong \sO_{C_\alpha}$.  The fixed
  part of $D$ is the total transform of $e_8$.
\end{itemize}
\end{prop}

We also mention a necessary condition for a divisor to be integral, related
to the theory of Coxeter groups of Kac-Moody type.  The convention there is
to consider both ``real'' roots (i.e., roots in the usual sense) and
``imaginary'' roots.  The latter are defined as integral vectors whose
orbit intersects the fundamental chamber in a nonnegative linear
combination of simple roots.

\begin{prop}
  Any integral divisor $D$ such that $D\cdot K_X=0$ is a positive root
  (real or imaginary).
\end{prop}

We have already shown this for $-2$-curves (i.e., that $-2$-curves are
positive real roots), while for nef curves it is a consequence of the
following more general fact.  Note that since $e_m\cdot C_\alpha=1$, the fact
that the simple roots are a basis of $C_\alpha^\perp$ implies that together
with $e_m$, they form a basis of $\Pic(X)$.

\begin{prop}
Let $X$ be an anticanonical rational surface with $K_X^2<7$, and
$D$ a nef divisor class on $X$.  Then for any blowdown structure
on $X$, $D$ is a nonnegative linear combination of the simple roots
and $e_m$.  In fact, if we write (for an even blowdown structure)
\[
D = a(s-f)+b(f-e_1-e_2)+\sum_{1\le i<m}c_i(e_i-e_{i+1})+c_m e_m,
\]
then we have the inequalities
\begin{align}
c_2&\ge c_3\ge\cdots\ge c_m\ge 0\notag\\
c_2&\ge b\ge a\ge 0.\notag\\
c_2&\ge c_1\ge 0\notag
\end{align}
\end{prop}

\begin{proof}
The divisor class $e_i$ is effective for all $i$ (since it is the
total transform of a $-1$-curve on $X_i$), and thus $D\cdot e_i\ge
0$.  Taking $i\ge 3$, we conclude that
\[
c_2\ge c_3\ge \cdots \ge c_m.
\]
To see that $c_m\ge 0$, we note that
\[
c_m = (c_m e_m)\cdot C_\alpha = D\cdot C_\alpha\ge 0.
\]
Similarly, the classes $f+s-e_1-e_2$, $s$, and $f$ are effective on
$X_2$, thus on $X$, so taking inner products with $D$ shows
\[
c_2\ge b\ge a\ge 0.
\]
Finally, the effective
classes $f+s-e_2$, $f+s-e_1$ tell us that
\[
c_2\ge c_1\ge 0,
\]
finishing the proof.
\end{proof}

\begin{rems}
Of course, we can weaken the hypothesis ``$X$ anticanonical'' to ``$D\cdot
K_X\le 0$'', since the latter fact was the only way in which we used the
anticanonical curve.
\end{rems}

\begin{rems}
If $D$ is in the fundamental chamber, the same sequences of coefficients
will be convex.
\end{rems}

\section{Moduli of surfaces}

\subsection{General surfaces}

One benefit of considering blowdown structures is that it makes the moduli
problem of rational surfaces much better behaved.  Of course, the standard
approach of choosing an ample bundle also works, but obscures the symmetry
of the situation; in contrast, as we have seen, working with blowdown
structures gives us a (rational) action of the Coxeter group $W(E_{m+1})$.

To construct the moduli stack (an Artin stack) of rational surfaces with
blowdown structures, we first need to construct the moduli stack of
Hirzebruch surfaces.  This is of course essentially just the moduli stack
of rank 2 vector bundles on $\P^1$, so is a standard construction, but it
will be useful to keep in mind the details.  (In particular, the
construction we use is not the usual construction for the moduli stack of
vector bundles; the extra structure we use to rigidify the moduli problem
has a simpler interpretation as a structure on $\P(V)$.)

For any integer $d\ge 0$, we have
\[
\Ext^1(\sO_{\P^1}(d+2),\sO_{\P^1})\cong 
H^1(\sO_{\P^1}(-d-2))\cong 
k^{d+1},
\]
and thus the non-split extensions of $\sO_{\P^1}(d+2)$ by $\sO_{\P^1}$ are
classified up to automorphisms of the two bundles by points of the
corresponding $\P^d$.  By a standard construction, this gives rise to a
canonical extension
\[
0\to \sO_{\P^1}\boxtimes \sO_{\P^d}(1) \to V \to \sO_{\P^1}(d+2)\boxtimes
  \sO_{\P^d}\to 0
\]
of sheaves on $\P^1\times \P^d$, each fiber of which is the corresponding
non-split extension of $\sO_{\P^1}(d+2)$ by $\sO_{\P^1}$.

Let $S_k$ denote the locally closed subspace of $\P^d$ on which the fiber
is isomorphic to $V_{d,k}:=\sO_{\P^1}(k+1)\oplus \sO_{\P^1}(d+1-k)$; this
gives a stratification of $\P^d$ by $S_k$ for $0\le k\le d/2$.  Each
stratum can itself be identified as a moduli space of global sections of
the $V_{d,k}$, namely the moduli space of {\em saturated} global sections
(i.e., generating a subbundle), modulo the action of $\Aut(V_{d,k})$.
Since the generic global section is saturated, we have
\[
\dim(S_k) = \dim(\Gamma(V_{d,k}))-\dim(\Aut(V_{d,k}))
          = d-\max(d-2k-1,0);
\]
in other words, $\dim(S_k)=2k+1$ except that $\dim(S_{d/2})=d$.

Since $\dim(\Gamma(V_{d,k}))=d+2$ is independent of $k$, this gives a flat
map to the moduli problem of vector bundles of the form $V_{d,k}$ for $0\le
k\le d/2$.  Taking the relative $\P$ of the bundle gives a flat map to the
moduli problem of Hirzebruch surfaces; since every Hirzebruch surface
arises in this way for sufficiently large $d$, we find that the moduli
problem of Hirzebruch surfaces is represented by an algebraic stack.  (Note
that if $V$, $V'$ are nonsplit extensions, then $\Hom(V,V')$ is given by a
locally closed subset of global sections of $V'$, so $\Isom(V,V')$ is
indeed a scheme as required, and admits a quotient by $\G_m$ to give
$\Isom_{\P^1}(\P(V),\P(V'))$.)   Note that the stabilizers have the form
$\P\Aut(V_{d,k})\rtimes \Aut(\P^1)$.

This stack has two components (even and odd Hirzebruch surfaces), with
generic fibers isomorphic to $F_0$ and $F_1$ respectively.  In general,
$F_d$ has codimension $d-1$ in the corresponding stack (simply compare
automorphism group dimensions).  Note that the smooth cover corresponding
to $V_d$ classifies pairs $(X,\sigma)$, where $X$ is a Hirzebruch surface
and $\sigma:\P^1\to X$ is an embedding with $\im(\sigma)\cdot f=1$,
$\im(\sigma)^2=d+2$; the map to the moduli stack simply forgets $\sigma$.

To blow up, we proceed as in \cite{HarbourneB:1988}, based on an idea of
Artin.  (There, Harbourne constructed the moduli stack of blowups of
$\P^2$; the extension to Hirzebruch surfaces is straightforward.)  Now, let
${\cal X}_0$ denote the moduli stack of Hirzebruch surfaces, and let ${\cal
  X}_1$ denote the corresponding universal family of rational surfaces,
with structure maps $\pi_0:{\cal X}_1\to {\cal X}_0$, $\rho:{\cal X}_1\to
{\cal X}_0\times \P^1$.  (This last is something of an abuse of notation;
what we really mean is the $\P^1$-bundle over ${\cal X}_0$ over which the
vector bundles were constructed.  Though this was a product over $\P^d$, we
are quotienting by $\Aut(\P^1)$.) Consider now the problem of classifying
surfaces with $K_X^2=7$.  Such a surface is uniquely determined by a pair
$(X_0,p)$ where $X_0$ is a Hirzebruch surface and $p\in X_0$ is a closed
point.  But points on a surface are classified by the universal surface, so
that the moduli space of rational surfaces with blowdown structure such
that $K_X^2=7$ is precisely ${\cal X}_1$.

Now, extend this to a sequence of stacks ${\cal X}_i$ and morphisms
$\pi_i:{\cal X}_{i+1}\to {\cal X}_i$ for all $i\ge 0$ in the following way.
Using the morphism $\pi_{i-1}$, we may construct the fiber product ${\cal
  X}_i\times_{{\cal X}_{i-1}} {\cal X}_i$ and then blow it up along the
diagonal.  Call the resulting blowup ${\cal X}_{i+1}$, and let $\pi_i$ be
the morphism induced by the first projection from the fiber product.

\begin{prop}
  The stack ${\cal X}_i$ represents the moduli problem of rational surfaces
  with blowdown structure and $K_X^2=8-i$.  The universal surface over this
  stack is $\pi_i:{\cal X}_{i+1}\to {\cal X}_i$, and the blowdown structure
  is induced by the maps
\[
{\cal X}_{i+1}\to
{\cal X}_i\times_{{\cal X}_{i-1}}{\cal X}_i
\to
{\cal X}_i\times_{{\cal X}_{i-2}}{\cal X}_{i-1}
\to
\cdots
\to
{\cal X}_i\times_{{\cal X}_0} {\cal X}_1
\to
{\cal X}_i\times \P^1
\]
In addition, for $m\ge 1$, each divisor class $f$, $e_1$,\dots,$e_m$
is represented by a divisor on the universal surface, and there exists a
line bundle of first Chern class $2s$.
\end{prop}

\begin{proof}
This is a simple induction: ${\cal X}_i$ is the universal surface over
${\cal X}_{i-1}$, so also the moduli space of pairs $(X_{i-1},p)$.  To
obtain the universal surface over ${\cal X}_i$, we need to blow up $p$ on
the corresponding fiber, and this is precisely what blowing up the diagonal
does for us.

The claim about divisors is clear for $e_1$,\dots,$e_m$, since these are
just the total transforms of the corresponding exceptional curves.
Similarly, $f-e_1$ is always a $-1$-curve on $X_1$, so gives rise to a
divisor on the universal surface.  To obtain a line bundle of class $2s$,
take the bundle $\rho^!\sO_{\P^1}^{-1}$.
\end{proof}

\begin{rem}
  For $m=0$, there is a small difficulty having to do with the fact that
  the $\P^1$ could be twisted; once $m=1$, we have guaranteed that the
  universal $\P^1$ in the construction has a point.  (Of course, the
  anticanonical bundle on the base $\P^1$ is always defined, so lifts to a
  bundle of class $2f$.)  Similarly, although the original construction on
  $\P^d$ comes with a section of the Hirzebruch surface, ${\cal X}_0$
  forgets that section, and the induced automorphisms can act nontrivially
  on the relative $\sO(1)$.
\end{rem}

\begin{cor}
  The moduli stack of rational surfaces with blowdown structure has two
  irreducible components for each value of $K_X^2\le 8$, and each component
  is smooth of dimension $10-2K_X^2$.
\end{cor}

\begin{proof}
  This is clearly true for ${\cal X}_0$ (since the generic Hirzebruch
  surface has a $6$-dimensional automorphism group and the stack is covered
  by open substacks for which some $\P^k\times \PGL_2$-bundle is isomorphic
  to $\P^d$), and each map $\pi_i$ is smooth of relative dimension 2 (being
  a family of smooth projective surfaces).
\end{proof}

\begin{rem}
  Note that the two components are naturally isomorphic for $K_X^2\le 7$:
  just apply the standard elementary transformation.  Also, the formula for
  the dimension holds for $\P^2$ as well, since
  $\dim\Aut(\P^2)=8=2K_{\P^2}^2-10$.
\end{rem}

The action of simple reflections on blowdown structures clearly extends to
give birational automorphisms of these stacks (since each simple root is
clearly generically ineffective).  The action is undefined when the root is
effective, leading us to wonder what those substacks look like.  It turns
out that any given positive root (simple or not) is effective on a closed
substack of codimension 1.  This is a special case of a more general fact
about flat families of sheaves, which we will have occasion to use again.

\begin{lem}\label{lem:tau_func}
  Let $\pi:X\to S$ be a projective morphism of schemes, and suppose $M$ is a
  coherent sheaf on $X$, flat over $S$.  Suppose moreover that
  $R^p\pi_*M=0$ for $p>1$, and the fibers of $M$ have Euler characteristic
  $0$.  Then the locus $T\subset S$ parametrizing fibers with global
  sections has codimension $\le 1$ everywhere.  Moreover, where
  $T\subsetneq S$, it is a Cartier divisor.
\end{lem}

\begin{proof}
  By \cite{KnudsenFF/MumfordD:1976}, the derived direct image of $M$ can be
  represented by a perfect complex on $S$ starting in degree 0.  Since the
  higher direct images vanish, we can replace the degree $1$ term by the
  kernel of the map to the degree $2$ term to obtain a two-term perfect
  complex.  The Euler characteristic condition implies that the two terms
  have the same rank everywhere, and $R^1\pi_*M$ is supported on the zero
  locus of the determinant of the appropriate map, so has codimension
  $\le 1$.  Semicontinuity implies that the fibers of $M$ have global
  sections precisely along the support of $R^1\pi_*M$.  That this is a
  Cartier divisor follows from the construction of
  \cite{KnudsenFF/MumfordD:1976}: the determinant is the canonical global
  section of the bundle $\det \dR\Gamma(M)^{-1}$.  (Note that the reference
  only shows that the fibers have global sections on the zero locus of the
  canonical global section; the argument above gives the converse as well.)
\end{proof}

\begin{cor}
  For any positive (real) root of $W(E_{m+1})$, the corresponding divisor
  class is effective on a codimension $1$ substack of ${\cal X}_m$.
\end{cor}

\begin{proof}
  Indeed, a positive root has $D^2=-2$, $D\cdot C_\alpha=0$, and thus
  $\chi(\sO_X(D))=0$.  Since there exist surfaces for which no positive
  root is effective and surfaces for which every positive root is
  effective, the substack is nonempty, and not all of ${\cal X}_m$, so has
  codimension $1$.
\end{proof}

\begin{rem}
Similarly, that a surface in ${\cal X}_9$ admits {\em some} anticanonical
curve is a codimension $1$ condition.
\end{rem}

\begin{cor}
  The substack of rational surfaces containing a
  $-e$-curve for some $e\ge 3$ is a countable union of closed substacks of
  codimension $\ge 2$.
\end{cor}

\begin{proof}
  A $-d$-curve on $X_m$ for $d>2$ is either the strict transform of a
  $-d$-curve on $X_0$ or arises from a $-(d-1)$ curve on some $X_k$ for
  $k<m$.  We have already noted that $X_0$ contains $-3$-curves (or worse)
  in codimension $\ge 2$.  For the other case, we observe by induction
  (using the $d=2$ case above) that $X_k$ containing a $-2$-curve is a
  countable union of codimension $1$ conditions (since any effective root
  is a sum of $-2$-curves, and in the absence of $-3$-curves or worse, is
  uniquely effective).  If we blow up a point of an effective root, the
  result will necessarily contain a $-3$-curve, andy any $-3$-curve not
  already present must arise in this way.
\end{proof}

\begin{rem}
  Similarly, for $d>3$, we have a $-d$-curve or worse on a countable union
  of locally closed substacks of codimension $\ge d-1$.  This fails to be
  closed since the $-2$-curve we are turning into a $-d$-curve could
  degenerate into a reducible effective root.  So, for instance, the closed
  substacks corresponding to $-4$-curves also contain configurations with
  two $-3$-curves connected by a chain of $-2$-curves.
\end{rem}

\medskip

If we try to extend the action of simple roots to the entire moduli space,
we encounter a problem along these codimension $\ge 2$ substacks.  For
simplicity in exhibiting the problem, we consider blowups of $\P^2$.
Define maps $p_1$, $p_2$, $p_3:\A^1\to \P^2$ by
\[
p_1(u) = (0,0),\qquad
p_2(u) = (u,0),\qquad
p_3(u) = (0,u^2),
\]
and define two family of surfaces parametrized by $\A^1$: $Y_u$ is the blowup
of $\P^2$ in $p_1$, $p_2$, then $p_3$, while $Z_u$ is the blowup of $\P^2$
in $p_3$, $p_2$, then $p_1$.  (At each step, the maps not already used
extend to the blowup, so give a well-defined point at which to blow up.)
For $u\ne 0$, we have $Y_u\cong Z_u$, since the points are distinct, so the
different blowups commute.  On the other hand $Y_0$ is the blowup of $F_1$
in two {\em distinct} points of the $-1$-curve, while $Z_0$ blows up the
same point of the $-1$-curve twice.  But then $Y_0\not\cong Z_0$, since
$Z_0$ contains a $-2$-curve, and $Y_0$ does not.

Since every rational surface locally looks like $\P^2$, we see that there
is no way to extend the action of $S_3$ on a three-fold blowup to include
the surfaces where the three-fold blowup introduces a $-3$-curve.
This is essentially the only difficulty, however.

\begin{thm}
  Let ${\cal X}^{{\ge}{-}2}_m$, $m\ge 2$, denote the stack parametrizing
  pairs $(X,\Gamma)$ where $X$ is a rational surface with $K_X^2=8-m$ not
  containing any $-d$-curves for $d>2$ and $\Gamma$ is a blowdown structure
  on $X$.  There is a natural action of the Coxeter group $W(E_{m+1})$ on
  this stack, which for every simple root is given by the usual action on
  blowdown structures, where the root is ineffective.
\end{thm}

\begin{proof}
  Since $m\ge 1$, we may use an elementary transformation to identify the
  two components of ${\cal X}^{{\ge}{-}2}_m$, and thus have both an odd and an
  even blowdown structure on $X$.

  Since $X$ has no $-d$-curves for $d>2$, an odd blowdown structure maps it
  to $F_1$, so that we can proceed on to $\P^2$.  Moreover, every
  infinitely near point that gets blown up as we proceed to $X$ is a jet.
  More precisely, the blowdown structure on $X$ is determined by (a) a
  union $F_{m+1}$ of jets on $\P^2$ and (b) a filtration $F_i$ of the
  structure sheaf of this union such that each quotient is the structure
  sheaf of a point.  When a given simple root $e_i-e_{i+1}$ is ineffective,
  the corresponding quotient $F_{i+1}/F_{i-1}$ is supported on two distinct
  points, and the reflection makes the other choice of $F_i$.  This extends
  immediately to the locus where the two points agree; the jet condition
  ensures that the degree 2 scheme parametrizing choices of $F_i$ is
  separable.  In particular, we find that the action of $S_{m+1}$ extends
  to the full stack.

  Similarly, an even blowdown structure corresponds to (a) a choice of
  $F_0$ or $F_2$, (b) a union of jets on this surface (disjoint from the
  $-2$-curve), and (c) a corresponding filtration.  The same argument tells
  us that the corresponding $S_m$ acts.  Since the two subgroups cover the
  full set of simple roots, we obtain
  the desired action of $W(E_{m+1})$.  (The braid relations hold because
  they hold generically.)
\end{proof}

\begin{rems}
  The stack ${\cal X}^{{\ge}{-}2}_m$ is not quite a substack of ${\cal
    X}_m$, in general, as it may be necessary to remove a countable
  infinity of closed substacks.  For instance, we can obtain a $-3$-curve
  in ${\cal X}_9$ by blowing up a point of any $-2$-curve, and thus each of
  the infinitely many positive roots of $E_{8+1}$ produces a different
  component we must remove.
\end{rems}

\begin{rems}
  Note that the simple reflections act trivially on the locus where the
  corresponding simple root is effective (thus a $-2$-curve).  Since that
  locus has codimension $1$, we see that the simple reflections act {\em as
    reflections} on ${\cal X}^{{\ge}{-}2}_m$.
\end{rems}

\begin{rems}
  Note that although the group acts on ${\cal X}^{{\ge}{-}2}_m$, this action
  {\em cannot} extend to the universal surface.  That is, the action
  preserves the isomorphism class of the surface, but the isomorphisms on
  generic fibers degenerate as we approach the bad fibers.  The problem
  here is that the universal surface is itself a moduli space of surfaces,
  but those surfaces could contain $-3$-curves; in other words, the generic
  isomorphisms degenerate precisely on the corresponding $-2$-curves on the
  fiber.
\end{rems}

\begin{rems}
  This also helps quantify the sense in which the moduli stack of surfaces
  is badly behaved when we do not introduce the blowdown structure: even if
  we exclude $-d$-curves for $d>2$, it is the quotient of an Artin stack by
  a discrete group which is infinite when $K_X^2\le 0$.
\end{rems}

\begin{rems}
  Finally, the reader should be cautioned that although we have divisors
  corresponding to the standard bases of $\Pic(X)$ (including $s$, since
  $s$ is represented by a canonical divisor on $X_0$ in the odd case),
  these choices of divisor are not compatible with the action of the
  Coxeter group.  (For instance, the reflection in $e_1-e_2$ changes the
  representation of $f$ from $(f-e_1)+e_1$ to $(f-e_2)+e_2$.)  In
  particular, although the various line bundles are taken to isomorphic
  bundles under the group action, those isomorphisms are not canonical.
\end{rems}

\subsection{Anticanonical surfaces}

When trying to extend the above construction to anticanonical surfaces, we
encounter the difficulty that the dimension of the anticanonical linear
system varies with the rational surface, and this variation depends in
subtle ways on the configuration of $-d$-curves on the surface with $d>2$.
It turns out that this is not too serious an obstruction to constructing
the linear system, however.

\begin{prop}\label{prop:linsys}
  Let ${\cal L}$ be a line bundle on a family $X$ of smooth projective
  varieties over a $d$-dimensional locally Noetherian integral base $S$,
  such that on every fiber, $H^p({\cal L})=0$ for $p\ge 2$.  Then the
  moduli functor $|{\cal L}|$ defined by taking $|{\cal L}|(T)$ to be the
  set of effective divisors on $X|_T$ with $\sO_X(D)\cong {\cal L}$ is
  represented by a locally projective scheme $|{\cal L}|$ over $S$, of
  dimension at least $d+\chi({\cal L})-1$ everywhere locally.  If the base
  is smooth and the dimension is equal to $d+\chi({\cal L})-1$, then
  $|{\cal L}|$ is a local complete intersection.
\end{prop}

\begin{proof}
  This certainly holds (and with equality for the dimension) when ${\cal L}$ is
  acyclic, since then Grauert's theorem tells us that the direct image of
  ${\cal L}$ is a vector bundle, and $|{\cal L}|$ is just the corresponding projective
  bundle.

  More generally, we may as well assume $S$ is Noetherian and affine.  In
  particular, there is an effective divisor $D$ such that ${\cal L}(D)$ is
  acyclic on every fiber, so that we may directly construct $|{\cal
    L}(D)|$.  The long exact sequence of cohomology associated
  to
  \[
  0\to {\cal L}\to {\cal L}(D)\to {\cal L}\otimes \sO_D(D)\to 0
  \]
  then tells us that $H^p({\cal L}\otimes \sO_D(D))\cong H^{p+1}({\cal L})$
  on every fiber, and thus that ${\cal L}\otimes \sO_D(D)$ is acyclic.  The
  dual vector bundle to $\Gamma({\cal L}\otimes \sO_D(D))$ then defines a
  system of equations on $|{\cal L}(D)|$, cutting out those sections
  containing $D$.  Call the resulting closed subscheme $Y$.  Since $Y$ is
  an intersection of $h^0({\cal L}\otimes \sO_D(D))$ hypersurfaces, it has
  dimension at least
  \[
  \dim |{\cal L}(D)| - h^0({\cal L}\otimes \sO_D(D))
  =
  d-1+\chi({\cal L}(D))-\chi({\cal L}\otimes \sO_D(D))
  =
  d-1+\chi({\cal L})
  \]
  everywhere locally.  The restriction to $Y$ of the universal divisor on
  $|{\cal L}(D)|$ has $D$ as a fixed locus, which we may subtract to obtain
  a universal divisor for $|{\cal L}|$ as required.

  The local complete intersection property follows from the fact that $Y$
  is obtained from a smooth scheme by intersecting the same number of
  hypersurfaces as its codimension.
\end{proof}

\begin{rem}
  Since automorphisms of ${\cal L}$ act trivially on $|{\cal L}|$, we can apply this even
  when ${\cal L}$ is merely an isomorphism class of line bundles, with the one
  caveat being that this will make $|{\cal L}|$ a family of Brauer-Severi
  varieties.
\end{rem}

In particular, we may define the stack ${\cal X}^\alpha_m$ of
anticanonical surfaces to be the linear system $|-K|$ on ${\cal
X}_m$.  Of course, we would like to know that this is irreducible,
which will require some more work, as the Proposition only gives
lower bounds on the dimension.

For Hirzebruch surfaces, things are not too difficult to control, as in
each case we can write $-K=D_0+D_1$ where $D_0$ is the divisor of fixed
components and $D_1$ is acyclic:
\begin{align}
F_0&: -K = 0+(2s+2f)\notag\\
F_1&: -K = 0+(2s+3f)\notag\\
F_2&: -K = 0+(2s+2f)\notag\\
F_{2d+1}&: -K = (s-df)+(s+(d+3)f),\quad d\ge 1\notag\\
F_{2d} &: -K = (s-df)+(s+(d+2)f),\quad d\ge 2\notag
\end{align}
In particular, we find in each case that $|-K|$ is flat over the given
locally closed substack, of relative dimension $\chi(D_1)$.  This has the
following curious effect: for all $k\ge 3$, the dimension of the stratum of
${\cal X}^\alpha_0$ corresponding to $F_k$ is $0$, while for $F_0$, $F_1$,
$F_2$ the dimension is $2$, $2$, $1$ respectively.  (Recall that ${\cal
  X}_0$ itself has dimension $-6$, since $F_0$ and $F_1$ have 6-dimensional
automorphism groups.)  Since ${\cal X}^\alpha_0$ has dimension at least $2$
everywhere, we find that it has precisely two irreducible components, as we
would have expected.

\begin{rem}
  This is already a distinct departure from the general case, since now all
  $-d$-curves for $d>2$ are a codimension $2$ phenomenon, not just the
  $-3$-curves.  It is clearer why this should be so for blowups: to obtain
  a $-d$ curve for $d>2$, we need simply blow up a point of a $-(d-1)$
  curve.  Since we already need to blow up a point of the anticanonical
  curve, this is a codimension 0 condition on the blowup!  Thus really the
  question is when the anticanonical curve is reducible, and this is
  codimension 2 (either the curve is reducible on $X_0$, or we must blow up
  a singular point).
\end{rem}

\begin{rem}
  Note that this moduli problem is not formally smooth.  For instance, if
  $X=F_4$ and $C_\alpha$ contains $s_{\min}$ with multiplicity $2$, then the
  anticanonical section $\alpha$ extends to an anticanonical section on an
  open subset of ${\cal X}_0$.  This exhibits a subspace of the tangent
  space to $(X,C_\alpha)$ as a direct sum of the tangent space to $X$ in
  ${\cal X}_0$ and the tangent space to $\alpha$ in
  $\P(H^0(\omega_X^{-1}))$.  But this subspace is larger than the generic
  tangent space!
\end{rem}

To understand ${\cal X}^\alpha_m$ in general, we may proceed by induction in $m$.

\begin{thm}
  The moduli problem of classifying triples $(X,C_\alpha,\Gamma$), where
  $X$ is a rational surface with $K_X^2=8-m$, $C_\alpha\subset X$ is an
  anticanonical curve, and $\Gamma$ is a blowdown structure, is represented
  by an Artin stack ${\cal X}^\alpha_m$.  This stack is a local complete
  intersection of dimension $m+2$, with one irreducible component for each
  parity of blowdown structure, with both components integral.  If $m\ge
  1$, the two components are canonically isomorphic.
\end{thm}

\begin{proof}
  We have already shown this for $m=0$, and we also note that the
  anticanonical curve is generically smooth in that case.  Now, the fibers
  of the natural forgetful map ${\cal X}^\alpha_m\to {\cal X}^\alpha_{m-1}$ are
  straightforward to determine: a point in the fiber just indicates which
  point of the anticanonical curve was blown up.  In other words, ${\cal
    X}^\alpha_m$ is the universal anticanonical curve over ${\cal
    X}^{\alpha}_{m-1}$.  By induction, the latter has two irreducible
  components, both integral, and the generic fiber of either component has
  smooth anticanonical curve.  In particular, we find that the fibers of
  ${\cal X}^\alpha_m$ over ${\cal X}^\alpha_{m-1}$ are all $1$-dimensional, and
  generically integral, so that each component of ${\cal X}^\alpha_{m-1}$
  has integral preimage.  Moreover, we immediately find that $\dim({\cal
    X}^\alpha_m)=m+\dim({\cal X}^\alpha_0)=m+2$ as required.

  That the components are isomorphic for $m>0$ follows by elementary
  transformation as before, and the local complete intersection property
  follows from the fact that
  $
  2m-6 + \chi(-K_X)-1 = m+2.
  $
\end{proof}

The above construction showing that $W(E_{m+1})$ cannot in general act in
the presence of $-3$-curves works equally well in the anticanonical case
(as long as the surface we start with has $K_X^2\ge 3$, so that there is an
anticanonical curve containing any three points).  Thus we introduce the
substack ${\cal X}^{\alpha,{\ge}{-}2}_m\subset {\cal X}^{\alpha}_m$ as
before, by excluding all $-d$-curves for $d>2$.  This is actually a
substack in this case, since as we noted above, once we bound the minimal
section of the Hirzebruch surface, there are only finitely many possible
configurations of components of the anticanonical curve.  Of course, not
having any $-d$-curves for $d>2$ is a very strong condition to impose on an
anticanonical surface: in particular, for $K_X^2<0$, it forces $C_\alpha$
to be integral.  (Indeed, otherwise the components of $C_\alpha$ are smooth
rational curves, at least one of which has negative intersection with
$C_\alpha$, so is a $-d$-curve for $d>2$.)

\begin{prop}
  For $0\le m<8$, ${\cal X}^{\alpha,{\ge}{-}2}_m$ is a $\P^{8-m}$-bundle
  over ${\cal X}^{{\ge}{-}2}_m$, and thus has two smooth components,
  isomorphic if $m>0$.  For $m>8$, ${\cal X}^{\alpha,{\ge}{-}2}_m$ can be
  identified with a closed substack of ${\cal X}^{{\ge}{-}2}_m$.
\end{prop}

\begin{proof}
  For $m>8$, the anticanonical curve is integral if it exists, and since it
  has negative self-intersection is rigid.  Thus there is at most one
  anticanonical curve on a rational surface without $-d$-curves for $d>2$.
  We have seen that this is a closed codimension 1 condition for $m=9$,
  while for $m>9$, it combines the closed conditions that the image in
  ${\cal X}^{{\ge}{-}2}_{m-1}$ is in ${\cal X}^{\alpha,{\ge}{-}2}_{m-1}$ and
  that the point being blown up lies on the anticanonical curve.

For $m<8$, the anticanonical divisor is nef on a surface without
$-d$-curves for $d>2$, and since $C_\alpha^2=8-m>0$, the corresponding line
bundle is acyclic.  But then the linear system is a $\P^{8-m}$-bundle as
required.
\end{proof}

\begin{rem}
The case $m=8$ is more subtle, as the surface could have a unique
anticanonical curve, or could have a $1$-parameter family of anticanonical
curves (making it an elliptic surface with no multiple fibers).
\end{rem}

In any case, since the choice of $C_\alpha$ is independent of the blowdown
structure, the action of $W(E_{m+1})$ extends immediately to ${\cal
  X}^{\alpha,{\ge}{-}2}_m$.  As before, the action does not extend to the
universal surface (it must act linearly on $\Pic(X)$, so does not respect
the effective cone), but it turns out that there is a strong sense in which
it {\em does} act on the {\em line bundles} on the universal surface.

Given any vector $v\in \Z s+\Z f+\sum_i \Z e_i$, we have a corresponding
line bundle ${\cal L}_v$ on the universal surface over ${\cal
  X}^{\alpha,{\ge}{-}2}_m$.  (In general, we only knew this when the
coefficient of $s$ was even, but the assumptions imply that an odd blowdown
structure reaches $F_1$, where $s$ is canonically a divisor, and the claim
for even blowdown structures follows by elementary transformation.)  Of
course, the space of global sections of this bundle can vary wildly with
the surface, and can similarly vary if we replace $v$ by $wv$ for any
element $w\in W(E_{m+1})$.  These are essentially the same phenomenon,
however.  The main problem with the global sections of ${\cal L}_v$ is that
we can have sections of ${\cal L}_v$ on a given surface that do not extend
to neighboring surfaces in the moduli space.  This can be fixed by taking
the direct image sheaf rather than the fiberwise global sections.  This can
cause problems in general, however, which are characterized by the
following result (a strong (albeit specialized) form of semicontinuity).

\begin{lem}
Let $\pi:X\to S$ be a projective morphism, and suppose that $M$ is a
sheaf on $X$, flat over $S$, such that every fiber of $M$ has $H^p=0$ for
$p\ge 2$.  Then for any sheaf $N$ on $S$, we have isomorphisms
\[
\Tor_{p+2}(R^1\pi_*M,N)\cong \Tor_p(\pi_*M,N)
\]
for $p>0$, along with a short exact sequence
\[
0\to \Tor_2(R^1\pi_*M,N)\to \pi_*M\otimes N\to \pi_*(M\otimes \pi^*N)\to
\Tor_1(R^1\pi_*M,N)\to 0
\]
and an isomorphism
\[
R^1\pi_*M\otimes N\cong R^1\pi_*(M\otimes \pi^*N).
\]
In particular, $\pi_*M$ is flat iff $R^1\pi_*M$ has homological dimension
$\le 2$, the fibers of $\pi_*M$ inject in the corresponding spaces of
global sections of $M$ iff $R^1\pi_*M$ has homological dimension $\le 1$,
and the injection is an isomorphism iff $R^1\pi_*M$ is flat.
\end{lem}

\begin{proof}
  As in the proof of Lemma \ref{lem:tau_func}, we find that $\dR\pi_*M$ is
  represented by a two-term perfect complex on $S$.  If $V^0\to V^1$ is
  this complex, then we have a four-term exact sequence
\[
0\to \pi_*M\to V^0\to V^1\to R^1\pi_*M\to 0,
\]
and thus any flat resolution of $\pi_*M$ extends to a flat resolution of
$R^1\pi_*M$.  The claim follows upon tensoring this resolution with $N$ and
observing that
\[
\dR\pi_*M\otimes^\dL N\cong \dR\pi_*(M\otimes^\dL\pi^*N)\cong
\dR\pi_*(M\otimes \pi^*N),
\]
with the last isomorphism following from the fact that $M$ is flat.
\end{proof}

\begin{rem}
  As an example, consider the divisor class $-2K_X$ on the moduli stack of
  anticanonical Hirzebruch surfaces.  This acquires cohomology when the
  surface has a $-d$-curve for any $d\ge 3$; since this locus has
  codimension $2$, $R^1\pi_*(\omega_X^{-2})$ has homological dimension $\ge
  2$.  As a result, the fibers of the direct image sheaf do not inject in
  the spaces of global sections of the fibers.  Similarly, the
  anticanonical bundle itself fails this criterion in the presence of a
  $-d$-curve for $d\ge 4$.
\end{rem}

  By the last claim of the Lemma, we can compute $R^1\pi_*M$ fiberwise, and
  the main contribution comes from hypersurfaces: those where a given
  positive root becomes effective, and those where ${\cal L}_v|_{C_\alpha}$
  has a global section.  Near a generic point of such a hypersurface, we
  find that $\pi_*{\cal L}_v$ is flat and injects fiberwise in the space of
  global sections of ${\cal L}_v$, since $R^1\pi_*{\cal L}_v$ is a flat
  sheaf on a hypersurface, so has homological dimension $\le 1$.  Although
  there could in principle be problems coming from intersections of the
  hypersurfaces, this at least suggests the following result; note that by
  the previous remark, we cannot allow worse than $-2$-curves, even if we
  did not care about the $W(E_{m+1})$ action.

\begin{thm}
  The direct image of any line bundle ${\cal L}_v$ is a flat sheaf ${\cal
    V}_v$ on ${\cal X}^{\alpha,{\ge}{-}2}_m$, and the action of $W(E_{m+1})$
  extends to these sheaves.  More precisely, for any element $w\in
  W(E_{m+1})$, we have an isomorphism
\[
w^*{\cal V}_v\cong {\cal V}_{wv},
\]
defined up to scalar multiplication, and the isomorphisms are compatible,
again up to scalar multiplication.  Moreover, the multiplication map
\[
{\cal V}_v\times {\cal V}_{v'}\to {\cal V}_v\otimes {\cal V}_{v'}\to {\cal
  V}_{v+v'}
\]
induced by
\[
{\cal L}_v\otimes {\cal L}_{v'}\cong {\cal L}_{v+v'}
\]
has no zero divisors.
\end{thm}

\begin{proof}
  First note that if $v$ is not generically effective, then ${\cal V}_v=0$,
  since then no global section of ${\cal L}_v$ on a fiber can extend to an
  open substack of the moduli space.  The generically effective divisors
  form a cone invariant under the action of $W(E_{m+1})$, so the various
  claims are immediate outside this cone.  A generically effective divisor
  will have $v\cdot f\ge 0$, so $(-C_\alpha-v)\cdot f\le -2$, and thus
  $-C_\alpha-v$ cannot be effective.  We thus conclude that $H^2({\cal
    L}_v)=0$ for such a divisor, which is all we need to apply the Lemma.

  It will suffice to show that whenever $v$ is generically effective, the
  group acts and $R^1\pi_*{\cal L}_v$ has homological dimension $\le 1$.
  Indeed, this implies flatness of ${\cal V}_v$, as well as the fact that
  multiplication has no zero-divisors, the latter since the map
\[
\Gamma({\cal L}_v)\times \Gamma({\cal L}_{v'})
\to
\Gamma({\cal L}_{v+v'})
\]
is injective on every fiber.

Now, using an elementary transformation as necessary, we may suppose our
blowdown structure is odd and consider $X$ as an $m+1$-fold blowup of
$\P^2$.  In the corresponding basis of $\Pic(X)$, we have
\[
v = nh - \sum_{0\le i\le m} r_i e_i.
\]
If $r_i=v\cdot e_i<0$ for any $i$, then we have a short exact sequence
\[
0\to {\cal L}_{v-e_i}\to {\cal L}_v\to {\cal L}_v|_{e_i}\to 0.
\]
On the generic fiber, the quotient is a sheaf of negative degree on the
smooth rational curve $e_i$, and thus has no global sections; on the
general fiber, the quotient has $1$-dimensional support, so that the Lemma
applies.  We thus find that $\pi_*({\cal L}_v|_{e_i})=0$ and (since that
certainly injects!) that $R^1\pi_*({\cal L}_v|_{e_i})$ has homological
dimension $\le 1$.  It follows that
\[
{\cal V}_{v-e_i}\cong {\cal V}_v,
\]
and $R^1\pi_*(\sO_X(v))$ has homological dimension $\le 1$ iff
$R^1\pi_*(\sO_X(v-e_i))$ has homological dimension $\le 1$.  By
induction, if we set
\[
v' = nh - \sum_{0\le i\le m} \max(r_i,0)e_i,
\]
then (since this operation respects the action of $S_{m+1}$) it suffices to
prove the claim for $v'$.

Thus suppose $r_i\ge 0$ for $0\le i\le m$, and consider the short exact
sequence
\[
0\to {\cal L}_v\to {\cal L}_{nh}\to Q\to 0.
\]
Since ${\cal L}_{nh}$ is acyclic, ${\cal V}_{nh}$ is flat, and we have
exhibited ${\cal V}_v$ as a subsheaf of this flat sheaf.  Moreover,
the fibers ${\cal V}_v$ inject in $\Gamma({\cal L}_v)$ iff they inject in
the corresponding fibers of ${\cal V}_{nh}$.

Now, $\pi_*Q$ is the kernel of a two-term perfect complex (since $Q$ has
$1$-dimensional support), and is thus a subsheaf of a locally free sheaf.
In particular, $\pi_*Q$ is torsion-free, and thus the map ${\cal V}_v\to
{\cal V}_{nh}$ is determined by its action on the generic fiber.  This
action is clearly $S_{m+1}$-covariant, and thus so is ${\cal V}_v$.
Since the morphism ${\cal V}_v\to {\cal V}_{nh}$ determines the injectivity
condition, we also conclude that if $R^1\pi_*{\cal L}_v$ has
homological dimension $1$, then so does $R^1\pi_*{\cal L}_{w v}$ for any
$w\in S_{m+1}$.

A similar calculation with an even blowdown structure shows that the
corresponding $S_m$ acts on the bundles, and preserves the homological
dimension condition.  The one technicality is that the ambient bundle
${\cal L}_{ns+df}$ need not be acyclic, but we can use $S_{m+1}$-invariance
(conjugated by an elementary transformation) to assume $n\ge d$.

We thus now have full $W(E_{m+1})$-covariance, so that it suffices to prove
the homological dimension claim for $v$ in the fundamental chamber.  Of
course, if $v=0$, then ${\cal L}_0=\sO_X$, and the claim is obvious, so
suppose $v\ne 0$.  If $v\cdot C_\alpha>0$, then ${\cal L}_v$ is acyclic, and
we are done.  Otherwise, consider the short exact sequence
\[
0\to {\cal L}_{v+K}\to {\cal L}_v\to {\cal L}_v|_{C_\alpha} \to 0
\]
Generically, ${\cal L}_v|_{C_\alpha}$ is a nontrivial degree 0 sheaf on a
smooth genus 1 curve, and thus we again find
\[
\pi_*({\cal L}_v|_{C_\alpha})=0
\]
and thus ${\cal L}_v$ satisfies the homological dimension condition iff
${\cal L}_{v-C_\alpha}$ satisfies the homological dimension condition.
\end{proof}

We should note a couple of things here.  First, the argument shows that on
a blowup of $\P^2$, ${\cal V}_v\subset {\cal V}_{(v\cdot h)h}$ whenever
$v\cdot h>0$, with locally free quotient, and similarly ${\cal V}_v\subset
{\cal V}_{(v\cdot f)s+(v\cdot s)f}$ relative to an even blowdown structure.
This fact will guide the noncommutative construction in
\cite{noncomm1}; we will first construct noncommutative analogues of
the ambient bundles, then impose suitable conditions on the generic fiber
and use an analogue of the above argument to prove flatness.

Next, the result allows us to construct a flat family of categories with a
nice action of $W_{E_{m+1}}$.  The objects of the categories are the
vectors $v\in \Z s+\Z f+\sum_i \Z e_i$, while the morphisms from $v$ to
$v'$ are given by ${\cal V}_{v'-v}$, with the natural multiplication maps.
The dimensions of the $\Hom$ spaces in this category are constant as we
vary the choice of anticanonical surface, and the group acts in the obvious
way.  The construction of \cite{noncomm1} will give a noncommutative
deformation of this category, in the case $C_\alpha$ smooth; this will depend
on one additional parameter (a point of $\Pic^0(C_\alpha)$), but will have
the same flatness properties.  The $\Hom$ spaces of the deformation will be
constructed as spaces of elliptic difference operators, and thus there is a
close connection between modules over the deformed category and (symmetric
elliptic) difference equations.  (In particular, every symmetric elliptic
difference equation will have a corresponding module over the deformation
for $m=0$.)  This will be extended in \cite{elldaha} to a two-parameter
deformation of the category with $\Hom$ spaces $S^n({\cal V}_{v'-v})$.

\medskip

In the case $C_\alpha$ integral, we can give a direct construction of the
substack of surfaces with anticanonical curve isomorphic to $C_\alpha$.
(Presumably this can be extended to general curves, but it is unclear what
the precise conditions will be.)

For any connected projective curve $C$ (integral or not) of arithmetic
genus 1, there is a natural moduli problem mapping flatly to the locally
closed substack of ${\cal X}^{\alpha}_m$ where $C_\alpha\cong C$, namely the
problem of classifying triples $(X,\phi,\Gamma)$ where $\phi:C\to X$ embeds
$C$ as an anticanonical curve.  Given such a triple, the restriction
morphism $\phi^*:\Pic(X)\to \Pic(C)$ gives us a sequence of (isomorphism
classes of) bundles $\phi^*(s)$, $\phi^*(f)$ and $\phi^*(e_i)$ for $1\le
i\le m$.  The classes $\phi^*(e_i)$ have degree 1, $\phi^*(f)$ has degree
2, and $\phi^*(s)$ has degree $1$ or $2$ depending on whether $\Gamma$ is
odd or even.

\begin{lem}
  The triple $(X,\phi,\Gamma)$ is determined up to isomorphism of $X$ by
  the classes $\phi^*(s)$, $\phi^*(f)$, and $\phi^*(e_i)$, $1\le i\le m$.
\end{lem}

\begin{proof}
  We may assume $X$ and $C$ are defined over an algebraically closed field,
  so that the given classes actually correspond to line bundles.  Since
  neither of $K_X+f$ or $-f$ is effective, we conclude that
  $H^0(\omega_X(f))=H^2(\omega_X(f))=0$, and Euler characteristic
  considerations imply $H^1(\omega_X(f))=0$.  It follows that we have a
  natural isomorphism $\Gamma(X,\sO_X(f))\to \Gamma(C,\phi^*(f))$, so
  that we can recover the induced map $\rho:C\to \P^1$ from $\phi^*(f)$ (up
  to $\PGL_2$).  Similarly, $\dR\rho_*(\omega_X\otimes {\cal L}_s)=0$, so
  that we have an isomorphism $\rho_*{\cal L}_s\cong \rho_*\phi^*(s)$
  allowing us to recover $X_0$ as the projective bundle of
  $\rho_*(\phi^*(s))$, as expected.

  Now, suppose we have reconstructed the surface $X_k$, and consider the
  direct image of $\phi^*(e_{k+1})$ on the corresponding anticanonical
  curve $C_k$.  This can be identified with the direct image of the sheaf
  $\sO_X(e_{k+1})|_{C_\alpha}$ on $X$, and thus by Corollary 6.7 of
  \cite{poisson} fits into an exact sequence
\[
0\to \sO_{C_k}\to \phi_{k*}\phi^*(e_{k+1})\to \sO_{p_{k+1}}\to 0
\]
where $p_{k+1}$ is the point of $C_k$ that gets blown up on $X_{k+1}$.
We have
\[
\Gamma(X_k,\phi_{k*}\phi^*(e_{k+1}))
=
\Gamma(X_{k+1},\sO_{X_{k+1}}(e_{k+1})|_{C_\alpha}).
\]
Since $e_{k+1}$ is a $-1$-curve on $X_{k+1}$, we find that
$\sO_{X_{k+1}}(e_{k+1})$ is acyclic and uniquely effective, while
\[
H^p(\sO_{x_{k+1}}(e_{k+1}-C_\alpha)) \cong
H^{2-p}(\sO_{x_{k+1}}(-e_{k+1}))^* = 0.
\]
It follows that $\phi_{k*}\phi^*(e_{k+1})$ is uniquely effective, so
determines $p_{k+1}$ and thus $X_{k+1}$.
\end{proof}

For $C$ integral, we readily see that any sequence of invertible sheaves of
the correct degrees will give rise to a valid triple, and thus we can
identify the moduli space of triples with the product
\[
\Pic^2(C)\times \Pic^2(C)\times \Pic^1(C)^m
\qquad
\text{or}
\qquad
\Pic^1(C)\times \Pic^2(C)\times \Pic^1(C)^m,
\]
depending on parity.  This is a principal $\Aut(C)$-bundle over the
corresponding substack of ${\cal X}^{\alpha}_m$, and since the choice of
embedding of $C$ is independent of the choice of blowdown structure, the
action of $W(E_{m+1})$ extends.  Of course, this extension is just the
obvious linear action!  Note also that if we allow $C$ to vary over the
moduli space of smooth genus 1 curves, then this construction gives a
dense open substack of ${\cal X}^\alpha_m$.

For nonintegral curves, there will certainly be an additional constraint on
the degree vectors of the invertible sheaves, since those degrees can be
read off from the combinatorial type of $(X,\Gamma)$ (i.e., the
representations of the components of $C_\alpha$ in the standard basis).  And,
of course, there is the additional difficulty that the moduli stack of
curves has many more pathologies once one allows nonintegral curves,
especially since we do not want to impose any stability conditions.  (For
instance, there are reduced but reducible curves of arithmetic genus 1 that
are not even Gorenstein.)

Along these lines, we note the following constraint on
the anticanonical curve of a rational surface.

\begin{lem}\label{lem:Ca_is_num_conn}
  Let $X$ be a rational surface, and suppose we can write $-K_X=A+B$ with
  $A$, $B$ nonzero effective divisors.  Then $H^1(\sO_A)=H^1(\sO_B)=0$ and
  $A\cdot B=2h^0(\sO_A)=2h^0(\sO_B)>0$.
\end{lem}

\begin{proof}
  Since $A$ and $B$ are nonzero effective divisors, we find that
  $H^0(\sO_X(-A))=H^0(\sO_X(-B))=0$, and thus by duality
  $H^2(\sO_X(-A))=H^2(\sO_X(K_X+B))$=0.  The long exact sequence associated to
  the natural presentation of $\sO_A$ has a piece
\[
H^1(\sO_X)\to H^1(\sO_A)\to H^2(\sO_X(-A)),
\]
and thus $H^1(\sO_A)=0$, with $H^1(\sO_B)=0$ following similarly.  We thus
have $h^0(\sO_A)=\chi(\sO_A)=\frac{1}{2}A\cdot (-K_X-A)=\frac{1}{2}A\cdot
B$, so the remaining claim follows.
\end{proof}

\begin{rem}
We also find $h^1(\sO_X(-A))=-\chi(\sO_X(-A))=\frac{1}{2}A\cdot B-1$.
\end{rem}

This has the following interesting consequence; this was established for
anticanonical curves in $\P^2$ and $\P^1\times \P^1$ in
\cite[Cor.~5.7]{ArtinM/TateJ/vdBerghM:1991}, but given the above Lemma, the
proof carries over directly.

\begin{prop}\label{prop:Pic0_acts}
  Let $(X,C_\alpha)$ be an anticanonical rational surface.  Then there is a
  natural action of the group scheme $\Pic^0(C_\alpha)$ on $C_\alpha$,
  which for an invertible sheaf ${\cal Q}\in \Pic^0(C_\alpha)$ takes a
  point $p\in C_\alpha$ to the unique point $p'$ such that ${\cal
    I}_{p'}\cong {\cal Q}\otimes {\cal I}_p$.
\end{prop}

\begin{rem}
  Similarly, the proof of Proposition 5.9 and Lemma 5.10 op. cit. tells us
  that for any ${\cal Q}\in \Pic^0(C_\alpha)$, the corresponding action
  $\tau_{\cal Q}$ fixes every singular point of $C_\alpha$, stabilizes any
  irreducible component of $C_\alpha$, and for any invertible sheaf ${\cal
    L}$ one has
\[
\tau_{\cal Q}^* {\cal L}\cong {\cal L}\otimes {\cal Q}^{-\chi({\cal L})}.
\]
The last part of Proposition 5.9 op. cit. suggests that if the component
$C$ occurs with multiplicity $m$, then $\tau_{\cal Q}$ restricts to the
identity on $(m-1)C$.
\end{rem}

\subsection{Partitioning the moduli stack}

Since the type and singularities of a difference or differential equation
depends on how the anticanonical curve decomposes and interacts with the
map to $\P^1$, we would like to understand the corresponding decomposition
of ${\cal X}^\alpha_m$.

The simplest part of the decomposition is by whether the curve is smooth or
of multiplicative or additive degeneration.  More precisely, the ``Picard
type'' of an anticanonical rational surface is defined to be one of the
symbols $e$, $*$, or $+$, depending on whether $\Pic^0(C_\alpha)$ is an
elliptic curve, multiplicative group, or additive group.

\begin{lem}
  The anticanonical rational surfaces of Picard type $*$ or $+$ (with
  blowdown structure) form an irreducible codimension $1$ closed substack
  ${\cal X}^{\alpha,\bar{*}}_m$ of ${\cal X}^\alpha_m$.
\end{lem}

\begin{proof}
  If $m>0$, then we may use the line bundle $\sO_{C_\alpha}(e_m)$ to induce
  a map from $C_\alpha$ to a weighted projective space with generators of
  degrees $1$, $2$, $3$, and the image of $C_\alpha$ will be the
  Weierstrass model of $\Pic^0(C_\alpha)$; then ${\cal
    X}^{\alpha,\bar{*}}_m$ is cut out by an invariant equation of degree
  12.

  For $m\le 0$, we need merely observe that if we blow up a point of
  $C_\alpha$, this will not change the Picard type, and thus the condition
  remains closed of codimension 1.
\end{proof}

\begin{rem}
  We could also use the nef divisor $f$ on $X_0$ or $h$ on $\P^2$ and
  reduce to known results on hyperelliptic curves of genus 1 or cubic
  plane curves.  The Weierstrass case is particularly nice over $\Z[1/6]$,
  however, since then we obtain well-defined functions $a_4$, $a_6$ on the
  $G_m$-bundle over ${\cal X}^\alpha_m$ in which we have chosen a Poisson
  structure on $X$ (since that corresponds to a choice of nonzero
  holomorphic differential).  Then ${\cal X}^{\alpha,\bar{*}}_m$ is cut out in
  those coordinates by the equation $-64 a_4^3-432 a_6^2=0$.
\end{rem}

The above works over $\Z$, but the additive case behaves differently in
characteristic $2$ and $3$, and thus we have a slightly weaker result.

\begin{lem}
  Over an algebraically closed field or a $\Z[1/6]$-algebra, the
  anticanonical rational surfaces of Picard type $+$ (with blowdown
  structure) form an irreducible codimension $1$ closed substack
  ${\cal X}^{\alpha,+}_m$ of ${\cal X}^{\alpha,\bar{*}}_m$.
\end{lem}

\begin{rem}
  Over $\Z[1/6]$, the additive Weierstrass curves are cut out from the full
  stack of Weierstrass curves by the equations $a_4=a_6=0$, but over $\F_3$
  the equations have degrees $2$, $12$ and over $\F_2$ they have degrees
  $1$, $12$; the same holds for hyperelliptic and cubic models.
\end{rem}

One issue that arises with the stack ${\cal X}^{\alpha,\bar{*}}_m$ is that
it does not distinguish $0$ and $\infty$.  For instance, in the irreducible
case, there are two branches of $C_\alpha$ at the node, and the limit as we
approach the two points is different.  We can resolve this by taking a
suitable double cover of the moduli space.  Indeed, the Weierstrass model
has a unique singular point, which we may take to be $(0,0)$, giving a
curve of the form
\[
y^2 + a_1 xy = x^3+a_2 x^2.
\]
The two branches at $\infty$ are given by the roots of $b_1^2+a_1 b_1-a_2$,
and thus the double cover has a model of the form
\[
y^2 + c_1 xy = x^3,
\]
where $c_1=a_1+2b_1$ is uniquely determined if we have chosen a nonzero
holomorphic differential on $C_\alpha$.  The double cover is ramified along
${\cal X}^{\alpha,+}_m$, and the equation for the corresponding substack is
$c_1=0$, now valid over $\Z$.  If $C_\alpha$ has multiplicative type with
multiple components, then the double cover distinguishes the two branches
at any given one of the nodes, and that identification can then be carried
along the polygon.

We similarly define the ``combinatorial type'' of an anticanonical rational
surface with blowdown structure to be the (multi)set of pairs $(C_i,m_i)\in
\Pic(X_m)\times \Z$, where $C_i$ ranges over the components of the
anticanonical curve and $m_i$ is the corresponding multiplicity.  Again, if
we have two components with the same divisor class and multiplicity, we
will usually want to distinguish them.  We thus choose an ordering of the
components, which gives us a somewhat different way to describe the type:
if we have $c$ components, then we have a morphism $\phi:\Z^c\to \Pic(X_m)$
taking $e_i$ to $C_i$, and a vector $\mu\in \Z^c$ giving the
multiplicities, with the property that $\phi\mu = C_\alpha$.

The ``type'' of a rational surface is then defined to be the symbol
$(c,\phi,\mu)_\sigma$ where $(c,\phi,\mu)$ gives the combinatorial type and
$\sigma$ gives the Picard type.  Note that most of the type the Picard type
is determined by the combinatorial type: if $\max_i\mu_i>1$, then the
Picard type is $+$, while if $c>3$ and $\max_i\mu_i=1$, then the Picard
type is $*$.  So only the cases with $c\le 3$ and no multiplicities have
any ambiguity.  (And, of course, the only case with Picard type $e$ is
$(1,(-K),(1))_e$.)

It is, in principle, straightforward to determine the set of possible types
for any given $m$, as it is easy to write down a complete (though countably
infinite) list for $m=0$, and each point being blown up either lies on a
single component or on a set of mutually intersecting components.  One also
has the following.  Define the ``dimension'' of a type to be
\[
\dim((c,\phi,\mu)_\sigma)
=
\begin{cases}
  m+1 & \text{$c=1$, $\sigma=*$}\\
  m & \text{$c=1$, $\sigma=+$}\\
  m-1 & \text{$c=2$, $\mu=(11)$, $\sigma=+$}\\
  m-2 & \text{$c=3$, $\mu=(111)$, $\sigma=+$}\\
  2m-6 + \sum_{1\le i\le c} \phi_i\cdot (\phi_i-K)/2, & \text{otherwise.}
\end{cases}
\]

\begin{lem}
  For any type $(c,\phi,\mu)_\sigma$, the anticanonical rational surfaces
  of that type form an irreducible Artin stack of dimension
  $\dim((c,\phi,\mu)_\sigma)$, which maps to a locally closed substack of
  the double cover of ${\cal X}^{\alpha,\bar{*},m}$.
\end{lem}

\begin{proof}
  We proceed by induction on $m$.  For $m=-1$, the result is
  straightforward: since we have labelled the components, we may construct
  the stack inside the product of the linear systems on $\P^2$, and the
  various conditions we want to impose or exclude are closed, with known
  codimension; subtracting $8=\dim\Aut(\P^2)$ then gives the desired result
  for the dimension.  For most cases with $m=0$, we can proceed similarly;
  when the combinatorial type includes a section of nonpositive
  self-intersection $-d$, the surface must be $F_d$, while if
  it otherwise contains a component having negative intersection with
  $s-f$, it must be $F_0$ or $F_1$.  The only remaining cases are
  \[
  (1,(2s+2f),(1))_{e/*/+},\quad\text{and}\quad
  (2,(s+f,s+f),(1,1))_{*/+}.
 \notag
  \]
  The anticanonical curve on such a surface has a natural hyperelliptic
  model, and the surface itself is determined by this model together with a
  class in $\Pic^0$ (which is trivial iff the surface is $F_2$), and thus
  the stack has dimension 1 more than the corresponding stack of
  hyperelliptic curves.

  For $m>0$, the type of $X_m$ determines the type of $X_{m-1}$ as well as
  the set of components the point being blown up lies on.  Irreducibility
  is then almost immediate, as the fibers of the stack parametrizing $X_m$
  over the stack parametrizing $X_{m-1}$ are either
  single points or open subsets of curves.  (Here we use the fact that we
  have chosen branches at the singular points, and thus when blowing up a
  point on a surface of type $(2,\phi,(1,1))_*$, we can distinguish the two
  points.)

  It remains only to show that change in the dimension of the type is the
  same as the change in dimension of the classifying stack.  If neither
  type falls into one of the exceptions in the dimension formula, then the
  change in the dimension of the type is indeed $1$ or $0$ depending on
  whether the point being blown up is on $1$ or $2$ components.  It is then
  straightforward to deal with the exceptions case-by-case.
\end{proof}

In particular, this tells us how many parameters a given type has, modulo
the effect of automorphisms.

Of course, simply knowing the possible types of degenerations of surfaces
(or of difference equations) is only part of the story: in general, we
would like to understand the limiting relations between the different
types.  That is, we would like to know which types appear in the closure in
${\cal X}^\alpha_m$ of the locally closed substack corresponding to a
particular type.  Note that this is somewhat more than just a poset, as a
given type may appear in multiple ways in the closure whenever there are
ambiguities in the labelling.  In particular, we expect the answer to be an
ordered category rather than simply a poset.  (There is a further issue, in
that there are cases in small characteristic in which the intersection of
two closures of types is not a union of types, see below.)

We have not been able to answer this question completely, but do have a
couple of necessary conditions, the simpler of which is as follows.  Define
a (na\"{i}ve) category structure on the set of types as follows.  Let
$(c_1,\phi_1,\mu_1)_{\sigma_1}$, $(c_2,\phi_2,\mu_2)_{\sigma_2}$ be a pair
of types.  If $\sigma_1>\sigma_2$ relative to the order $e>*>+$, then there
are no morphisms between the types, while otherwise a morphism is given by
a linear transformation $\psi:\Z^{c_2}\to \Z^{c_1}$ with nonnegative
coefficients such that $\phi_2=\phi_1\circ \psi$, $\psi(\mu_2)=\mu_1$.
Note that any endomorphism in this category is an automorphism, so this is
indeed an ordered category.

\begin{lem}
  A flat morphism from a dvr to ${\cal X}^\alpha_m$ induces a morphism from
  the type of the (surface corresponding to the) special fiber to the type
  of the generic fiber.
\end{lem}

\begin{proof}
  This is certainly true for the Picard type, so it suffices to consider
  the combinatorial type.  Making a quasi-finite flat base change as
  necessary, we may assume that the irreducible components of the
  anticanonical curve on the generic fiber are geometrically irreducible,
  and then choose an ordering on the components to induce a well-defined
  combinatorial type.  The corresponding family of anticanonical curves
  defines a divisor on the family of surfaces, not containing any fiber.
  As a result, the same applies to each irreducible component of that
  divisor.  Those are in one-to-one correspondence with the components of
  the generic anticanonical curve, and thus (by restriction) induces a
  divisor on the special fiber corresponding to each component of the
  generic anticanonical curve.  Each such divisor is contained in the
  anticanonical curve on the special fiber, and is thus a sum of geometric
  components of that curve.  The linear transformation such that $\psi_i$ is
  the linear combination of components corresponding to the restriction of
  $C_i$ gives a morphism of combinatorial types as required.
\end{proof}

We call a morphism in the category of types ``effective'' if it arises from
a family over a dvr, and ``strongly effective'' if there is a family over a
dvr such that the special fiber is the {\em generic} surface of the given
type.  It is easy to see that if $\psi_1$, $\psi_2$ are morphisms and
$\psi_1$ is strongly effective, then $\psi_2\circ \psi_1$ is (strongly)
effective iff $\psi_2$ is (strongly) effective.

The one technique we have (other than producing an explicit deformation)
for proving strong effectiveness is the following.

\begin{prop}
  Let $\psi:T_1\to T_2$ be a morphism such that for any factorization
  \[
  \begin{CD}
  T_1 @>\psi_1>> T' @>\psi_2>> T_2
  \end{CD}
  \notag
  \]
  of $\psi$ with $T'\not\cong T_2$, $\dim(T')<\dim(T_2)$.  Then $\psi$ is
  strongly effective.
\end{prop}

\begin{proof}
  Let $T_2=(c,\psi,\mu)$, and let $Y$ be the fiber product $\prod_i
  |\psi_i|$ of linear systems on $X_m$.  There is a natural map from $Y$ to
  ${\cal X}^\alpha_m$ given by taking the anticanonical curve to be $\sum_i
  \mu_i D_i$ where $D_i$ is the fiber of $|\psi_i|$.  It follows from
  Proposition \ref{prop:linsys} that $Y$ has dimension at least $2m-6 +
  \sum_{1\le i\le c} \phi_i\cdot (\phi_i-K)/2$.  Moreover, if the
  combinatorial type does not force the Picard type, then imposing the
  corresponding condition introduces one or two more hypersurfaces as
  appropriate.  We thus obtain a stack $Y_{T_2}$ of dimension at least
  $\dim(T_2)$ everywhere locally.

  Now, it is easy to see that this construction defines a functor on the
  category of types, and thus we have a morphism $Y_{T_1}\to Y_{T'}\to
  Y_{T_2}$ for any factorization as described.  Moreover, the moduli stack
  of surfaces of type precisely $T$ embeds as an open substack of
  $Y_T$ for each $T$.

  Now, if $\psi$ were not strongly effective, then the substack
  corresponding to $T_1$ would necessarily meet some component of $Y_{T_2}$
  not containing the substack corresponding to $T_2$.  Any such component
  has dimension at least $\dim(T_2)$, and thus the type $T'$ of its generic
  fiber has at least that dimension.  But the existence of surfaces of type
  $T'$ in $Y_{T_2}$ implies that there is a morphism $T'\to T_2$, and we
  thus obtain a factorization of precisely the form we excluded.
\end{proof}

The following special cases are fairly straightforward, if occasionally
tedious; we omit the details.

\begin{cor}
  Any morphism between multiplicative types is strongly effective.
\end{cor}

\begin{cor}
  Any morphism to a type $(c,\phi,\mu)$ with $c\le 3$, $\max_i \mu_i=1$ is
  strongly effective.
\end{cor}

\begin{cor}
  For $m\le 1$, any morphism is strongly effective.
\end{cor}

Unfortunately, this is not the full story in general.  To see this,
consider the moduli stack of singular del Pezzo surfaces of degree 1.  In
our terms, this is the quotient stack ${\cal X}^{\alpha,\ge -2}_7/W(E_8)$,
and thus in particular inherits a decomposition from the one we have
constructed on ${\cal X}^{\alpha,\ge 2}_7$.  We can thus gain insight into
the geometry of this decomposition by considering the quotient.  In
characteristic 0, a singular del Pezzo surface has an equation of the form
\[
y^2 = x^3 + a_4(x,w)x + a_6(x,w)
\]
where $a_4$, $a_6$ are homogeneous of the given degree.  If we blow up the
base point of the anticanonical linear system, we obtain an elliptic
surface, and the type (mod $W(E_8)$) of the anticanonical curve $w=0$ can
be read off from the Kodaira symbol of the corresponding fiber of the
elliptic surface.  In particular, we can read off a parametrization of the
surfaces of any given type from Tate's algorithm (see specifically the
discussion in \cite[\S IV.9]{SilvermanJH:1994}), and then perform an
elimination to determine the equations satisfied by the coefficients of
$a_4$, $a_6$ in general.  (There is a minor issue, in that the scheme
corresponding to type $I_8$ is reducible, but it is easy enough to factor.)
We thus obtain a collection of $22$ ideals corresponding to the different
possible special fibers, and it is easy enough to both determine the
containment relations between the different ideals and to verify that the
intersection of any two of the closures is a union of closures of types.
We thus obtain a stratification of the moduli stack of del Pezzo surfaces
{\em in characteristic 0}, which we may denote pictorially as in Figure
1.\footnote{These Hasse diagrams were adapted from the \TeX\ code for the
  analogous diagram in \cite{JoshiN/NakazonoN/ShiY:2016}.}  (Here we
specify the root system of the $-2$-curves rather than the Kodaira symbol.)
The open stratum $A_0^e$ is $m+2=9$-dimensional, and each successive column
decreases the dimension by $1$.  For surfaces with $K^2=0$, one obtains the
same diagram, with dimensions increased by $1$ (and affine rather than
finite root systems), since ${\cal X}^{\alpha,\ge -2}_8$ is the universal
$\Pic^0(C_\alpha)$ over ${\cal X}^{\alpha,\ge -2}_7$, and the map on types
coming from blowing down is clearly bijective.

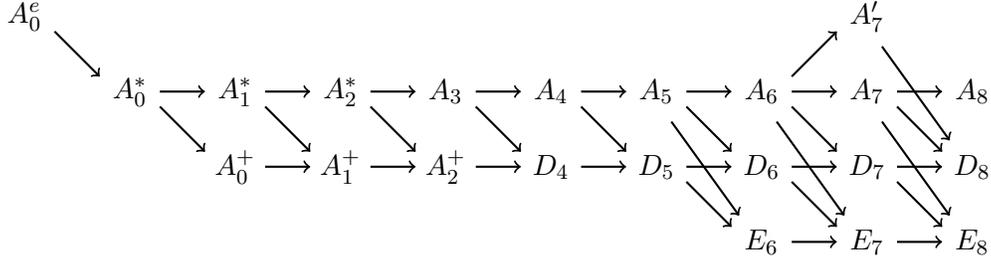
\begin{figure}[t]
\centering
\begin{tikzpicture}[scale = 1]
\begin{scope}
\coordinate (P11s) at (0,0);
\coordinate (P10e) at ($(P11s)-(0.8,0)$);
\coordinate (P10s) at ($(P10e)-(0.6,0)$);
\coordinate (P11e) at ($(P11s)+(0.6,0)$);
\coordinate (P12s) at ($(P11e)+(0.8,0)$);
\coordinate (P12e) at ($(P12s)+(0.6,0)$);
\coordinate (P13s) at ($(P12e)+(0.8,0)$);
\coordinate (P13e) at ($(P13s)+(0.6,0)$);
\coordinate (P14s) at ($(P13e)+(0.8,0)$);
\coordinate (P14e) at ($(P14s)+(0.6,0)$);
\coordinate (P15s) at ($(P14e)+(0.8,0)$);
\coordinate (P15e) at ($(P15s)+(0.6,0)$);
\coordinate (P16s) at ($(P15e)+(0.8,0)$);
\coordinate (P16e) at ($(P16s)+(0.6,0)$);
\coordinate (P17s) at ($(P16e)+(0.8,0)$);
\coordinate (P17e) at ($(P17s)+(0.6,0)$);
\coordinate (P18s) at ($(P17e)+(0.8,0)$);
\coordinate (P18e) at ($(P18s)+(0.6,0)$);
\coordinate (P19s) at ($(P18e)+(0.8,0)$);
\coordinate (P19e) at ($(P19s)+(0.6,0)$);
\coordinate (P21s) at (0,-1);
\coordinate (P20e) at ($(P21s)-(0.8,0)$);
\coordinate (P20s) at ($(P20e)-(0.6,0)$);
\coordinate (P21e) at ($(P21s)+(0.6,0)$);
\coordinate (P22s) at ($(P21e)+(0.8,0)$);
\coordinate (P22e) at ($(P22s)+(0.6,0)$);
\coordinate (P23s) at ($(P22e)+(0.8,0)$);
\coordinate (P23e) at ($(P23s)+(0.6,0)$);
\coordinate (P24s) at ($(P23e)+(0.8,0)$);
\coordinate (P24e) at ($(P24s)+(0.6,0)$);
\coordinate (P25s) at ($(P24e)+(0.8,0)$);
\coordinate (P25e) at ($(P25s)+(0.6,0)$);
\coordinate (P26s) at ($(P25e)+(0.8,0)$);
\coordinate (P26e) at ($(P26s)+(0.6,0)$);
\coordinate (P27s) at ($(P26e)+(0.8,0)$);
\coordinate (P27e) at ($(P27s)+(0.6,0)$);
\coordinate (P28s) at ($(P27e)+(0.8,0)$);
\coordinate (P28e) at ($(P28s)+(0.6,0)$);
\coordinate (P29s) at ($(P28e)+(0.8,0)$);
\coordinate (P29e) at ($(P29s)+(0.6,0)$);
\coordinate (P31s) at (0,-2);
\coordinate (P31e) at ($(P31s)+(0.6,0)$);
\coordinate (P32s) at ($(P31e)+(0.8,0)$);
\coordinate (P32e) at ($(P32s)+(0.6,0)$);
\coordinate (P33s) at ($(P32e)+(0.8,0)$);
\coordinate (P33e) at ($(P33s)+(0.6,0)$);
\coordinate (P34s) at ($(P33e)+(0.8,0)$);
\coordinate (P34e) at ($(P34s)+(0.6,0)$);
\coordinate (P35s) at ($(P34e)+(0.8,0)$);
\coordinate (P35e) at ($(P35s)+(0.6,0)$);
\coordinate (P36s) at ($(P35e)+(0.8,0)$);
\coordinate (P36e) at ($(P36s)+(0.6,0)$);
\coordinate (P37s) at ($(P36e)+(0.8,0)$);
\coordinate (P37e) at ($(P37s)+(0.6,0)$);
\coordinate (P38s) at ($(P37e)+(0.8,0)$);
\coordinate (P38e) at ($(P38s)+(0.6,0)$);
\coordinate (P39s) at ($(P38e)+(0.8,0)$);
\coordinate (P39e) at ($(P39s)+(0.6,0)$);
\coordinate (P41s) at (0,-3);
\coordinate (P41e) at ($(P41s)+(0.6,0)$);
\coordinate (P42s) at ($(P41e)+(0.8,0)$);
\coordinate (P42e) at ($(P42s)+(0.6,0)$);
\coordinate (P43s) at ($(P42e)+(0.8,0)$);
\coordinate (P43e) at ($(P43s)+(0.6,0)$);
\coordinate (P44s) at ($(P43e)+(0.8,0)$);
\coordinate (P44e) at ($(P44s)+(0.6,0)$);
\coordinate (P45s) at ($(P44e)+(0.8,0)$);
\coordinate (P45e) at ($(P45s)+(0.6,0)$);
\coordinate (P46s) at ($(P45e)+(0.8,0)$);
\coordinate (P46e) at ($(P46s)+(0.6,0)$);
\coordinate (P47s) at ($(P46e)+(0.8,0)$);
\coordinate (P47e) at ($(P47s)+(0.6,0)$);
\coordinate (P48s) at ($(P47e)+(0.8,0)$);
\coordinate (P48e) at ($(P48s)+(0.6,0)$);
\coordinate (P49s) at ($(P48e)+(0.8,0)$);
\coordinate (P49e) at ($(P49s)+(0.6,0)$);
\node at ($(P10s)-(0.4,0)$){$A_0^e$};
\node at ($(P21s)-(0.4,0)$){$A_0^*$};
\node at ($(P22s)-(0.4,0)$){$A_1^*$};
\node at ($(P23s)-(0.4,0)$){$A_2^*$};
\node at ($(P24s)-(0.4,0)$){$A_3$};
\node at ($(P25s)-(0.4,0)$){$A_4$};
\node at ($(P26s)-(0.4,0)$){$A_5$};
\node at ($(P27s)-(0.4,0)$){$A_6$};
\node at ($(P28s)-(0.4,0)$){$A_7$};
\node at ($(P29s)-(0.4,0)$){$A_8$};
\node at ($(P18s)-(0.4,0)$){$A_7'$};
\node at ($(P32s)-(0.4,0)$){$A_0^+$};
\node at ($(P33s)-(0.4,0)$){$A_1^+$};
\node at ($(P34s)-(0.4,0)$){$A_2^+$};
\node at ($(P35s)-(0.4,0)$){$D_4$};
\node at ($(P36s)-(0.4,0)$){$D_5$};
\node at ($(P37s)-(0.4,0)$){$D_6$};
\node at ($(P38s)-(0.4,0)$){$D_7$};
\node at ($(P39s)-(0.4,0)$){$D_8$};
\node at ($(P47s)-(0.4,0)$){$E_6$};
\node at ($(P48s)-(0.4,0)$){$E_7$};
\node at ($(P49s)-(0.4,0)$){$E_8$};
%
%
%
%
\draw [->, thick] (P21s)--(P21e);
\draw [->, thick] (P22s)--(P22e);
\draw [->, thick] (P23s)--(P23e);
\draw [->, thick] (P24s)--(P24e);
\draw [->, thick] (P25s)--(P25e);
\draw [->, thick] (P26s)--(P26e);
\draw [->, thick] (P27s)--(P27e);
\draw [->, thick] (P28s)--(P28e);
\draw [->, thick] (P32s)--(P32e);
\draw [->, thick] (P33s)--(P33e);
\draw [->, thick] (P34s)--(P34e);
\draw [->, thick] (P35s)--(P35e);
\draw [->, thick] (P36s)--(P36e);
\draw [->, thick] (P37s)--(P37e);
\draw [->, thick] (P38s)--(P38e);
\draw [->, thick] (P47s)--(P47e);
\draw [->, thick] (P48s)--(P48e);
\draw [->, thick] ($(P27s)+(0,0.2)$)--($(P17e)-(0,0.2)$); 
\draw [->, thick] ($(P10s)-(0,0.2)$)--($(P20e)+(0,0.2)$); 
\draw [->, thick] ($(P21s)-(0,0.2)$)--($(P31e)+(0,0.2)$); 
\draw [->, thick] ($(P22s)-(0,0.2)$)--($(P32e)+(0,0.2)$); 
\draw [->, thick] ($(P23s)-(0,0.2)$)--($(P33e)+(0,0.2)$); 
\draw [->, thick] ($(P24s)-(0,0.2)$)--($(P34e)+(0,0.2)$); 
\draw [->, thick] ($(P25s)-(0,0.2)$)--($(P35e)+(0,0.2)$); 
\draw [->, thick] ($(P26s)-(0,0.2)$)--($(P36e)+(0,0.2)$); 
\draw [->, thick] ($(P27s)-(0,0.2)$)--($(P37e)+(0,0.2)$); 
\draw [->, thick] ($(P28s)-(0,0.2)$)--($(P38e)+(0,0.2)$); 
\draw [->, thick] ($(P38s)-(0,0.2)$)--($(P48e)+(0,0.2)$); 
\draw [->, thick] ($(P36s)-(0,0.2)$)--($(P46e)+(0,0.2)$); 
\draw [->, thick] ($(P37s)-(0,0.2)$)--($(P47e)+(0,0.2)$); 
\draw [->, thick] ($(P26s)-(0.2,0.4)$)--($(P46e)+(0.1,0.35)$);
\draw [->, thick] ($(P27s)-(0.2,0.4)$)--($(P47e)+(0.1,0.35)$);
\draw [->, thick] ($(P18s)-(0.2,0.4)$)--($(P38e)+(0.1,0.35)$);
\draw [->, thick] ($(P28s)-(0.2,0.4)$)--($(P48e)+(0.1,0.35)$);
\end{scope}
\end{tikzpicture}
\caption{The natural stratification of ${\cal X}_7^{\alpha,\ge -2}/W(E_8)$
  over $\Z[1/6]$.}
\end{figure}
\begin{figure}[t]
\centering
\begin{tikzpicture}[scale = 1]
\begin{scope}
\coordinate (P11s) at (0,0);
\coordinate (P10e) at ($(P11s)-(0.8,0)$);
\coordinate (P10s) at ($(P10e)-(0.6,0)$);
\coordinate (P11e) at ($(P11s)+(0.6,0)$);
\coordinate (P12s) at ($(P11e)+(0.8,0)$);
\coordinate (P12e) at ($(P12s)+(0.6,0)$);
\coordinate (P13s) at ($(P12e)+(0.8,0)$);
\coordinate (P13e) at ($(P13s)+(0.6,0)$);
\coordinate (P14s) at ($(P13e)+(0.8,0)$);
\coordinate (P14e) at ($(P14s)+(0.6,0)$);
\coordinate (P15s) at ($(P14e)+(0.8,0)$);
\coordinate (P15e) at ($(P15s)+(0.6,0)$);
\coordinate (P16s) at ($(P15e)+(0.8,0)$);
\coordinate (P16e) at ($(P16s)+(0.6,0)$);
\coordinate (P17s) at ($(P16e)+(0.8,0)$);
\coordinate (P17e) at ($(P17s)+(0.6,0)$);
\coordinate (P18s) at ($(P17e)+(0.8,0)$);
\coordinate (P18e) at ($(P18s)+(0.6,0)$);
\coordinate (P19s) at ($(P18e)+(0.8,0)$);
\coordinate (P19e) at ($(P19s)+(0.6,0)$);

\coordinate (P21s) at (0,-1);
\coordinate (P20e) at ($(P21s)-(0.8,0)$);
\coordinate (P20s) at ($(P20e)-(0.6,0)$);
\coordinate (P21e) at ($(P21s)+(0.6,0)$);
\coordinate (P22s) at ($(P21e)+(0.8,0)$);
\coordinate (P22e) at ($(P22s)+(0.6,0)$);
\coordinate (P23s) at ($(P22e)+(0.8,0)$);
\coordinate (P23e) at ($(P23s)+(0.6,0)$);
\coordinate (P24s) at ($(P23e)+(0.8,0)$);
\coordinate (P24e) at ($(P24s)+(0.6,0)$);
\coordinate (P25s) at ($(P24e)+(0.8,0)$);
\coordinate (P25e) at ($(P25s)+(0.6,0)$);
\coordinate (P26s) at ($(P25e)+(0.8,0)$);
\coordinate (P26e) at ($(P26s)+(0.6,0)$);
\coordinate (P27s) at ($(P26e)+(0.8,0)$);
\coordinate (P27e) at ($(P27s)+(0.6,0)$);
\coordinate (P28s) at ($(P27e)+(0.8,0)$);
\coordinate (P28e) at ($(P28s)+(0.6,0)$);
\coordinate (P29s) at ($(P28e)+(0.8,0)$);
\coordinate (P29e) at ($(P29s)+(0.6,0)$);

\coordinate (P31s) at (0,-2);
\coordinate (P31e) at ($(P31s)+(0.6,0)$);
\coordinate (P32s) at ($(P31e)+(0.8,0)$);
\coordinate (P32e) at ($(P32s)+(0.6,0)$);
\coordinate (P33s) at ($(P32e)+(0.8,0)$);
\coordinate (P33e) at ($(P33s)+(0.6,0)$);
\coordinate (P34s) at ($(P33e)+(0.8,0)$);
\coordinate (P34e) at ($(P34s)+(0.6,0)$);
\coordinate (P35s) at ($(P34e)+(0.8,0)$);
\coordinate (P35e) at ($(P35s)+(0.6,0)$);
\coordinate (P36s) at ($(P35e)+(0.8,0)$);
\coordinate (P36e) at ($(P36s)+(0.6,0)$);
\coordinate (P37s) at ($(P36e)+(0.8,0)$);
\coordinate (P37e) at ($(P37s)+(0.6,0)$);
\coordinate (P38s) at ($(P37e)+(0.8,0)$);
\coordinate (P38e) at ($(P38s)+(0.6,0)$);
\coordinate (P39s) at ($(P38e)+(0.8,0)$);
\coordinate (P39e) at ($(P39s)+(0.6,0)$);

\coordinate (P41s) at (0,-3);
\coordinate (P41e) at ($(P41s)+(0.6,0)$);
\coordinate (P42s) at ($(P41e)+(0.8,0)$);
\coordinate (P42e) at ($(P42s)+(0.6,0)$);
\coordinate (P43s) at ($(P42e)+(0.8,0)$);
\coordinate (P43e) at ($(P43s)+(0.6,0)$);
\coordinate (P44s) at ($(P43e)+(0.8,0)$);
\coordinate (P44e) at ($(P44s)+(0.6,0)$);
\coordinate (P45s) at ($(P44e)+(0.8,0)$);
\coordinate (P45e) at ($(P45s)+(0.6,0)$);
\coordinate (P46s) at ($(P45e)+(0.8,0)$);
\coordinate (P46e) at ($(P46s)+(0.6,0)$);
\coordinate (P47s) at ($(P46e)+(0.8,0)$);
\coordinate (P47e) at ($(P47s)+(0.6,0)$);
\coordinate (P48s) at ($(P47e)+(0.8,0)$);
\coordinate (P48e) at ($(P48s)+(0.6,0)$);
\coordinate (P49s) at ($(P48e)+(0.8,0)$);
\coordinate (P49e) at ($(P49s)+(0.6,0)$);
\node at ($(P10s)-(0.4,0)$){$A_0^e$};

\node at ($(P21s)-(0.4,0)$){$A_0^*$};
\node at ($(P22s)-(0.4,0)$){$A_1^*$};
\node at ($(P23s)-(0.4,0)$){$A_2^*$};
\node at ($(P24s)-(0.4,0)$){$A_3$};
\node at ($(P25s)-(0.4,0)$){$A_4$};
\node at ($(P26s)-(0.4,0)$){$A_5$};
\node at ($(P27s)-(0.4,0)$){$A_6$};
\node at ($(P18s)-(0.4,0)$){$A_7$};
\node at ($(P28s)-(0.4,0)$){$A_7'$};
\node at ($(P29s)-(0.4,0)$){$A_8$};

\node at ($(P32s)-(0.4,0)$){$A_0^+$};
\node at ($(P33s)-(0.4,0)$){$A_1^+$};
\node at ($(P34s)-(0.4,0)$){$A_2^+$};
\node at ($(P35s)-(0.4,0)$){$D_4$};
\node at ($(P36s)-(0.4,0)$){$D_5$};
\node at ($(P37s)-(0.4,0)$){$D_6$};
\node at ($(P38s)-(0.4,0)$){$D_7$};
\node at ($(P39s)-(0.4,0)$){$D_8$};
\node at ($(P47s)-(0.4,0)$){$E_6$};
\node at ($(P48s)-(0.4,0)$){$E_7$};
\node at ($(P49s)-(0.4,0)$){$E_8$};

\draw [->, thick] (P21s)--(P21e);
\draw [->, thick] (P22s)--(P22e);
\draw [->, thick] (P23s)--(P23e);
\draw [->, thick] (P24s)--(P24e);
\draw [->, thick] (P25s)--(P25e);
\draw [->, thick] (P26s)--(P26e);
\draw [->, thick] (P27s)--(P27e);

\draw [->, thick] (P35s)--(P35e);
\draw [->, thick] (P36s)--(P36e);
\draw [->, thick] (P37s)--(P37e);
\draw [->, thick] (P38s)--(P38e);

\draw [->, thick] (P47s)--(P47e);
\draw [->, thick] (P48s)--(P48e);

\draw [->, thick] ($(P27s)+(0,0.2)$)--($(P17e)-(0,0.2)$);
\draw [->, thick] ($(P18s)-(0,0.2)$)--($(P28e)+(0,0.2)$);

\draw [->, thick] ($(P24s)-(0,0.2)$)--($(P34e)+(0,0.2)$);
\draw [->, thick] ($(P25s)-(0,0.2)$)--($(P35e)+(0,0.2)$);
\draw [->, thick] ($(P26s)-(0,0.2)$)--($(P36e)+(0,0.2)$);
\draw [->, thick] ($(P27s)-(0,0.2)$)--($(P37e)+(0,0.2)$);
\draw [->, thick] ($(P28s)-(0,0.2)$)--($(P38e)+(0,0.2)$);

\draw [->, thick] ($(P36s)-(0,0.2)$)--($(P46e)+(0,0.2)$);
\draw [->, thick] ($(P37s)-(0,0.2)$)--($(P47e)+(0,0.2)$);
\draw [->, thick] ($(P38e)+(0.3,-0.3)$)--($(P48e)+(0.3,0.3)$);

\draw [->, thick] ($(P26s)-(0.2,0.4)$)--($(P46e)+(0.1,0.35)$);
\draw[->, thick] ($(P28s)-(0.2,0.2)$) .. controls ($(P38s)+(0.3,0)$) .. ($(P48s)+(-0.1,0.2)$);
\draw[->, thick] ($(P29s)-(0.2,0.2)$) .. controls ($(P39s)+(0.1,0)$) .. ($(P49s)+(-0.1,0.2)$);

\draw [->, thick] ($(P10s)-(0,0.2)$)--($(P20e)+(0,0.2)$); 
\draw [->, thick] ($(P21s)-(0,0.2)$)--($(P31e)+(0,0.2)$); 
\draw [->, thick] ($(P22s)-(0,0.2)$)--($(P32e)+(0,0.2)$); 
\draw [->, thick] ($(P23s)-(0,0.2)$)--($(P33e)+(0,0.2)$); 
\draw [->, thick] (P32s)--(P32e);
\draw [->, thick] (P33s)--(P33e);
\draw [->, thick] (P34s)--(P34e);

\draw [->, thick] ($(P18s)-(0.2,0.4)$)--($(P38e)+(0.1,0.35)$);

\end{scope}
\end{tikzpicture}
\caption{The na\"{i}ve poset of types in ${\cal X}_7^{\alpha,\ge -2}/W(E_8)$.}
\end{figure}
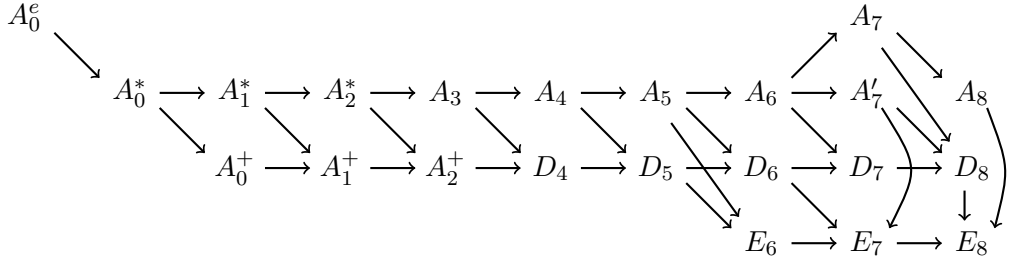

Now, the components of the anticanonical curve of a del Pezzo surface of
degree 1 all have negative self-intersection, so are rigid; as a result, if
there is a morphism at all between two such types, it is unique.  We find
in particular that the corresponding subcategory of the category of types
is simply the natural poset of root subsystems of $E_8$ (as extended by the
Picard type), ordered by inclusion.  However, the resulting diagram (Figure
2) is slightly different (essentially from \cite{SakaiH:2001}, except that
we have added the missing arrow from $A_7$ to $D_8$; note that $A_7$ and
$A_7'$ have swapped positions in the diagram to avoid arrows from the top
row to the bottom row in each case).
In particular, we find that although the combinatorics suggests that there
should be degenerations
\[
A'_7\to E_7,\quad
A'_7\to E_8,\quad
 D_8\to E_8,\quad
 A_8\to E_8,
\]
the corresponding morphisms are not effective.  Note that the three cases
corresponding to vertical arrows cannot possibly be {\em strongly}
effective, since the dimensions of the corresponding substacks are the same.

In fact, the situation is even worse than this suggests: any surface of
type $E_8$ in characteristic 3 can be obtained as the reduction mod 3 of a
surface of type $A_8$ over a suitable $3$-adic field.  (Similar statements
hold for the other three cases in characteristic 2.)  Even if one restricts
ones attention to equicharacteristic deformations, there are still issues:
in characteristic 3, one can obtain {\em some} surfaces of type $E_8$ as
limits from type $A_8$.  As a result, our decomposition of the moduli stack
of del Pezzo surfaces is not a stratification in characteristic 3 or over
$\Z$, as the closure of type $A_8$ meets type $E_8$ in a proper substack.

  Indeed, over a field, a surface with an $A_8$ singularity has the form
  \[
  y^2 +(a_{11}t+a_{10}u) xy + a_{30}u^3 y = x^3,
  \]
  up to changes of basis in $x$ and $y$ (but with no such changes of basis
  in $t$ and $u$ required).  The discriminant of this surface has the form
  \[
  c_1^2 t^3 u^9 + 3 c_1c_2 t^2 u^{10} + 3 c_2^2 t u^{11} + c_3 u^{12}
  \]
  for suitable functions $c_1$, $c_2$, $c_3$ of the parameters.  The
  discriminant of an $E_8$ surface in characteristic $\notin \{2,3\}$ has
  degree precisely 2 in $t$, but such a discriminant is not in the Zariski
  closure of the above set of discriminants, and thus there can be no
  degeneration from $A_8$ to $E_8$ in such cases.  On the other hand,
  consider the del Pezzo surface
\[
y^2 = x^3 - \frac{243t^2-54tu-u^2}{4}x^2 - \frac{3u^2(27t-5u)}{2}x
-\frac{tu^4(27t-4u)}{4}
\]
over the rationals.  This has an $A_8$ singularity at $y=x=u=0$, but
modulo $3$ becomes the surface
\[
y^2 = x^3 + u^2 x^2 + t u^5,
\]
which now has an $E_8$ singularity at $y=x=u=0$.  A characteristic 3
surface with an $A_8$ singularity at $u=0$ has discriminant of
the form $u^9(at^3+bu^3)$, and thus cannot degenerate to the above surface
of discriminant $-tu^{11}$.  However, the family of surfaces
\[
y^2 = x^3 + v^3(tv-u)^2 x^2 + u^2(u-vt)(u-v^3t) x + u^4 t (u+v^3t)
\]
over $\Spec(\F_3[v])$ generically has an $A_8$ singularity at $u=0$, but at
$v=0$ becomes the surface
\[
y^2 = x^3 + u^4 x + u^5 t
\]
with an $E_8$ singularity.  Every smooth fiber of the latter surface has
$j$-invariant $0$, so is a supersingular curve.  (There is one more
isomorphism class of surfaces with an $E_8$ singularity, namely the
quasielliptic surface $y^2=x^3+u^5 t$, but this is a degeneration of the
$j=0$ case.)

Let us consider what a surface of type $A_8$ would look like if we
considered it relative to a smooth anticanonical curve $C'$.  To obtain
type $A_8$, two things must happen: all roots of $A_8$ must vanish in
$\Pic(C')$, but also every root {\em not} in the subsystem must {\em not}
vanish.  Indeed, of some root not in $A_8$ were to vanish, then every root
of $E_8$ would vanish, and the surface would have type $E_8$ instead.
Since the lattice $\Lambda_{A_8}$ has index 3 in $\Lambda_{E_8}$, we see
that the image of $\Lambda_{E_8}$ in $\Pic^0(C')$ is a $3$-torsion
subgroup.  But in characteristic not $3$, it is impossible to degenerate a
nontrivial $3$-torsion point to the identity.  (In the $3$-adic case, we
can take the $3$-torsion point to be in the kernel of reduction, while in
the equicharacteristic case, we may take the special fiber of $C'$ to be
supersingular.)

Of course, the above argument is quite ad hoc, but it turns out that the
obstruction generalizes.  First, we note that the claim about degeneration
of torsion points holds for curves of arbitrary genus.

\begin{lem}\label{lem:order_r_closed}
  Let $R$ be a dvr with field of fractions $K$ and residue field $k$.
  Let $C_R$ be a smooth curve over $R$, and let ${\cal L}$ be a line bundle on
  $C_R$ such that ${\cal L}_K$ has exact order $r$ in $\Pic(C_K)$.  If $r\in R^*$,
  then ${\cal L}_k$ also has exact order $r$.
\end{lem}

\begin{proof}
  We first note that ${\cal L}_k^r\cong \sO_{C_k}$, since this is a closed
  condition.  But the assumption on $r$ implies that $\Pic^0(C_R)[r]$ is
  \'etale, and thus the closed subscheme $\Pic^0(C_R)[d]$ is also open for
  any divisor $d|r$.  It follows that ${\cal L}_k^d\not\cong \sO_{C_k}$ for every
  proper divisor $d|r$.
\end{proof}

\begin{rem}
  This also holds if $R$ is equicharacteristic and $C_k$ is ordinary, since
  then the complement of $\Pic^0(C_R)[p^l]$ in $\Pic^0(C_R)[p^{l+1}]$ is
  again both closed and open.  The claim fails in the remaining cases with
  $\ch(k)|r$, however.
\end{rem}
  
Now, a key fact about the curve $C'$ we used above was that it was
orthogonal to all of the roots in the relevant root systems.  In
particular, we can view $C'$ as a curve on the surface obtained by
contracting the $-2$-curves.  Moreover, $C'$ is an {\em ample} divisor on
that singular surface, and the claim boiled down to showing that the image
of the Picard group of the minimal desingularization in $C'$ contained a
copy of $\Lambda_{A_8}^\perp/\Lambda_{A_8}$.

Note that in Lemma \ref{lem:order_r_closed}, we only consider the Picard
groups of the two fibers; as a result, to apply the result, we need only
understand surfaces over fields.  Moreover, base changing to a finite
extension of $R$ has no effect on the order of $L_K$ or $L_k$, and thus we
can take a limit to a valuation ring in which both the residue field and
the field of fractions are algebraically closed.

With that in mind, let $Y$ be a normal surface over an algebraically closed
field $k$, with minimal desingularization $\tilde{Y}$.  To fix ideas,
suppose for the moment that we have a rational map $\phi:Y\to \P^1$ such
that the locus of indeterminacy is a single smooth point of $Y$ and
$\phi^*\sO_{\P^1}(1)$ is ample.  Then blowing up the corresponding point of
$\tilde{Y}$ gives a surface $\tilde{X}$ with a morphism to $\P^1$.  Now,
let $D$ be a divisor class on $\tilde{X}$ such that the restriction of $D$
to the generic fiber of $\tilde{X}$ is principal.  If we choose a function
$f$ with that divisor (which is uniquely determined modulo $k(\P^1)^*$),
then $D-\div(f)$ is certainly linearly equivalent to $D$, but now has
trivial restriction to the generic fiber.  This implies that $D-\div(f)$ is
a supported on a finite set of fibers, and is thus a linear combination of
components of fibers.

Now, suppose $C$ is a component of a fiber.  If $C$ does not meet the
exceptional curve of $\tilde{X}\to \tilde{Y}$, then its image in $Y$ is
disjoint from the generic fiber of $\phi$.  Since $\phi^*\sO_{\P^1}(1)$ is
ample, this implies that the image of $C$ must be a single point, and is
thus one of the singular points of $Y$.  We thus see that the components of
fibers split into two classes: those which are contracted in $Y$, and those
which meet the exceptional curve.  Since the exceptional curve meets each
fiber precisely once, we see that every fiber contains exactly one
component meeting the exceptional curve.

Suppose $C_1$,\dots,$C_n$ are the contracted components, and let $C$ be any
other component.  If $F$ is the fiber containing $C$, then $F-C$ is a sum
of components not meeting the exceptional locus, and thus we have
\[
F-C\in \Z\langle C_1,\dots,C_n\rangle.
\]
Now, suppose we are given a divisor class $D\in \Pic(\tilde{Y})$ such that
$rD\in \Z\langle C_1,\dots,C_n\rangle$ for some integer $r\ge 1$, and
suppose that $D$ has trivial restriction in the Picard group of the generic
fiber of the pencil.  Then the pullback of $D$ to $\tilde{X}$ is a linear
combination of components of fibers, and therefore has an expression of the
form
\[
D \sim \sum_i a_i C_i + \phi^*Z
\]
where $a_i\in \Z$ and $Z$ is a divisor on $\P^1$.  But this implies that
\[
(D-\sum_i a_i C_i)\cdot C_j = 0
\]
for all $j$.  By Mumford's criterion for contractibility
\cite{MumfordD:1961}, the intersection form of $C_1$,\dots,$C_n$ is
negative definite, and we thus find that
\[
D=\sum_i a_i C_i\in \Z\langle C_1,\dots,C_n\rangle.
\]
In other words, the map from $\Q\langle C_1,\dots,C_n\rangle\cap
\Pic(\tilde{Y})$ to the Picard group of the generic fiber has kernel
precisely $\Z\langle C_1,\dots,C_n\rangle$.

The main difficulty in applying the above argument in general is the
requirement that the pencil have a single base point, as this requires in
particular that $\langle C_1,\dots,C_n\rangle^\perp\subset \Pic(\tilde{Y})$
contains smooth (and non-rigid) curves of self-intersection $1$.  We would
thus like to extend the argument to deal with larger base loci.  The
difficulty, of course, is that fibers of the pencil can then have multiple
components meeting the base locus.

Let $Y$,$\tilde{Y}$ be as above, but now let ${\cal L}_0$ be an arbitrary
very ample line bundle on $Y$.  By Bertini's theorem, the corresponding
linear system contains smooth curves, so let us fix such a curve $C_0$.  A
further application of Bertini to $C_0$ lets us choose a curve $C_1$ in the
linear system meeting $C_0$ in $C_0^2$ distinct points, and we can then
choose $C_2$ meeting $C_0$ in the complement of $C_0\cap C_1$, so that the
linear system spanned by $C_0$, $C_1$, $C_2$ is base-point-free.  Let
$P_{12}$ denote the pencil through $C_1$ and $C_2$, and let $P$ denote the
pencil through $C_0$ and the generic point of $P_{12}$.  Now, although $P$
has base points, they are in general defined over an extension field of the
field $k(\P^1)$ over which $P$ is defined.  In fact, we have the following.

\begin{lem}
  The base points of $P$ are defined over the separable closure of
  $k(\P^1)$, and form a single orbit under the action of the absolute
  Galois group $\Gal(k(\P^1))$.
\end{lem}

\begin{proof}
  That the splitting field of the base scheme of $P$ is separable follows
  from the fact that $C_0\cap C_1$ is reduced, and thus the same holds if
  we replace $C_1$ by the generic fiber of $P_{12}$.  The linear
  system $P_{12}|_{C_0}$ is base-point-free and thus determines a morphism from
  $C_0$ to $\P^1$.  The base points of $P$ are then just the preimage of
  the generic point under this morphism, and thus the base scheme is
  $\Spec(k(C_0))$.  Since $k(C_0)$ is a field, transitivity is immediate.
\end{proof}

Now let $\tilde{X}$ be obtained from $\tilde{Y}$ by blowing up the base
locus of $P$.  If $D$ is a divisor on $\tilde{Y}$ {\em defined over $k$}
which is principal on the generic fiber of $P$, then, just as before, its
preimage in $\tilde{X}$ is linearly equivalent to an integer linear
combination of components of fibers of $P$.  Let $D'$ be such a
representative, and let $\sigma$ be in the absolute Galois group of
$k(\P^1)$.  Then $\sigma D'-D'\sim \sigma D-D$ is principal, but disjoint
from the generic fiber, and is thus the divisor of a function pulled back
from the base of the pencil.  But then since $k(\P^1)$ has trivial Brauer
group, we can add such a divisor to $D'$ to make it Galois-stable.
In particular, $D'$ is a sum of Galois-orbits of components of fibers.

The key insight now is that the components which are contracted on $Y$ are
contained in $k$-rational fibers.  If $C$ is a component of such a fiber
which is {\em not} contracted on $Y$, then its image in $\tilde{Y}$ meets
the base locus of $P$.  But then by transitivity of the Galois group on the
base points, we conclude that the image on $\tilde{Y}$ of the sum of the
Galois orbit of $C$ actually {\em contains} the base locus.  Then the
resulting Galois stable curve is simply the strict transform of the
corresponding fiber of $P$ on $Y$, and may thus be expressed as a linear
combination of the preimage of the fiber and the components of the preimage
of the singular points.

Thus if $C_i$ are the contracted components, we obtain an expression
\[
D' = \sum_i a_i C_i + \sum_j F_j + \sum_l D_l
\]
where each $F_j$ is a fiber and each $D_l$ is an effective divisor
supported on a non-rational fiber.  But then
\[
D'\cdot C_j = \sum_i a_i C_i\cdot C_j,
\]
and we can argue as before.  We thus obtain the following result.

\begin{prop}\label{prop:Pic_injectivity}
  Let $Y$ be a normal surface over an algebraically closed field $k$, with
  minimal desingularization $\tilde{Y}$ and exceptional curves
  $C_1$,\dots,$C_n$.  Let $\Lambda=\Z\langle C_1,\dots,C_n\rangle\subset
  \Pic(\tilde{Y})$, and let $\Lambda^+\subset \Pic(\tilde{Y})$ be the
  subgroup consisting of line bundles such that ${\cal L}^r\in \Lambda$ for
  some $r\in \Z$ which is invertible in $k$.  If $C_0$ is a smooth very
  ample curve on $Y$, then the restriction map $\Pic(\tilde{Y})\to
  \Pic(C_0)$ induces an exact sequence
  \[
  0\to \Lambda\to \Lambda^+\to \Pic(C_0)
  \]
\end{prop}

\begin{proof}
  The above argument shows that this holds for the generic curve in some
  pencil containing $C_0$, and then the result follows by Lemma
  \ref{lem:order_r_closed}.
\end{proof}

In the following result, note that by \cite{ArtinM:1974} any flat family of
surfaces with sufficiently nice rational singularities admits a uniform
minimal desingularization over an \'etale cover of the base.  In our
specific application, we start with a family of smooth surfaces, so there
is no issue.

\begin{prop}
  Let $R$ be a dvr and $Y/\Spec(R)$ be a projective scheme such that the
  fibers are normal rational surfaces, and suppose that
  $\tilde{Y}/\Spec(R)$ is a fiberwise minimal desingularization of $Y$.  Let
  $C_1$,\dots,$C_n$ be the components of the exceptional locus of
  $\tilde{Y}_k$, and $C'_1$,\dots,$C'_m$ the components of the exceptional
  locus of $\tilde{Y}_K$.  Let $a_{ij}$ be the multiplicity of $C_j$ in the
  special fiber of the closure of $C'_i$, and let $\vec{a}_i\subset \Z^n$
  be the corresponding collection of vectors.  Then the quotient
  \[
  (\Q\langle \vec{a}_1,\dots,\vec{a}_m\rangle \cap \Z^n)
  /
  (\Z\langle \vec{a}_1,\dots,\vec{a}_m\rangle)
  \]
  is trivial if $\ch(k)=0$ and a $p$-group if $\ch(k)=p$.
\end{prop}

\begin{proof}
  Both $Y_k$ and $Y_K$ have the property that their generic hyperplane
  section is smooth, and thus there is a hyperplane defined over $R$ such
  that both fibers of the corresponding section $C_0$ are smooth.  For any
  vector in $\Q\langle \vec{a}_1,\dots,\vec{a}_m\rangle\cap \Z^n$, let $D$
  be the corresponding linear combination of $C_1,\dots,C_n$.  Since $Y$
  has rational fibers, the invertible sheaf $\sO_{Y_k}(D)$ extends uniquely
  from $Y_k$ to $Y$, and some power of $\sO_{Y_K}(D)$ will have
  divisor class in $\Z\langle C'_1,\dots,C'_m\rangle$.  The restriction
  $\sO_{C_0}(D)$ is trivial on the special fiber of $C_0$, and thus
  by Lemma \ref{lem:order_r_closed} is trivial (or has $p$-power order) on
  the generic fiber.  The result follows from Proposition
  \ref{prop:Pic_injectivity}.
\end{proof}

To apply this to anticanonical surfaces, we need to understand when
configurations of anticanonical components can be contracted.

\begin{lem}
  Let $X$ be an anticanonical rational surface, and let $C_1$,\dots,$C_l$
  be a sequence of anticanonical components such that the intersection
  matrix of $C_1,\dots,C_l$ is negative definite and not every component
  appears.  Then the curve $\cup_i C_i$ is contractible.
\end{lem}

\begin{proof}
  First note that if any of $C_1$,\dots,$C_l$ are $-1$-curves, then we can
  simply blow down that curve and apply the Lemma to $X_{m-1}$.  We may
  thus assume that $C_i^2\le -2$ for $1\le i\le l$.  Now, let
  $m_1,\dots,m_s$ be the multiplicities of the anticanonical components,
  and consider the divisor class $Z^+=\sum_{1\le i\le l} m_i C_i$.  Then
  for $1\le i\le l$,
  \[
  Z^+\cdot C_i = (C_\alpha-\sum_{l<j\le s} m_j C_j)\cdot C_i
  = (2+C_i^2) - \sum_{l<j\le s} m_j (C_j\cdot C_i)
  \le 0,
  \]
  while
  \[
  (Z^+)^2 = \sum_{1\le i\le l}
    m_i(2+C_i^2) - \sum_{\begin{substack} 1\le i\le l\\ l<j\le
        s\end{substack}} m_im_j (C_j\cdot C_i)
    <0,
  \]
  since $C_\alpha$ is connected.  It follows that the {\em fundamental
    cycle} \cite{ArtinM:1966} of $\cup_i C_i$ satisfies $Z=\sum_i r_i C_i$
  with $1\le r_i\le m_i$, and thus in particular both $Z$ and $C_\alpha-Z$
  are effective.  It follows from Lemma \ref{lem:Ca_is_num_conn} that every
  connected component of $Z$ has arithmetic genus 0, and then
  \cite{ArtinM:1962,ArtinM:1966} tells us that every connected component of
  $\cup_i C_i$ contracts to a rational singularity.  The result essentially
  follows immediately.  (Artin assumes the configuration is connected, but
  the proof carries over.)
\end{proof}

\begin{cor}
  Let $\psi:T_1\to T_2$ be a morphism of types, suppose $C_1,\dots,C_l$ is
  a contractible configuration of anticanonical components of $T_1$, let
  $\Lambda_1\subset \Pic(X_m)$ be the corresponding sublattice, and let
  $\Lambda_2\subset \Lambda_1$ be the sublattice coming from those
  components of $T_2$ which are supported on $C_1,\dots,C_l$ (relative to
  $\psi$).  If $\Lambda_1/\Lambda_2$ has torsion of degree prime to the
  characteristic, then $\psi$ is ineffective.
\end{cor}

This in particular explains the obstructions for del Pezzo surfaces of
degree 1 (and corresponding obstructions for rational surfaces with $K^2=0$
and no $-d$-curves with $d>2$, the subject of \cite{SakaiH:2001}), as well
as giving obstructions in other cases.  For instance, let $T_2$ be the type
of (quartic del Pezzo) surfaces with anticanonical curve decomposition:
\[
C_\alpha
=
(s-e_1-e_4)+
(f-e_1-e_2)+
(s-e_2-e_3)+
(f-e_3-e_4)+
(e_1)+
(e_2)+
(e_3)+
(e_4),
\]
and let $T_1$ correspond to surfaces with
\[
C_\alpha
=
2(s-f)+4(f-e_1-e_2)+3(e_1-e_2)+6(e_2-e_3)+5(e_3-e_4)+4(e_4).
\]
Since the components of $T_1$ are linearly independent, it is easy to
verify that there is a morphism $\psi:T_1\to T_2$.  On the other hand, the
only components of $T_2$ supported on the complement of $e_4$ are the first
four.  The corresponding sublattice of $\Z^5$ is not saturated (it contains
$2e_2-2e_4$ but not $e_2-e_4$), and thus such a morphism can only
correspond to a degeneration when the special fiber has characteristic 2.
Here we should note that if we blow up a generic point on each component of
the generic fiber (and then blowup a corresponding point downstairs), then
there are no subsets of the anticanonical components to which the
obstruction applies.  Thus to apply this obstruction fully, one must in
principle consider not only the types themselves but also all possible
blowdowns.  It is unclear if the resulting obstruction is effective\dots

\medskip

For surfaces with $K^2=0$, although the computer calculation only showed
that Figure 1 was valid in characteristic 0, it is not too hard to extend
it to $\Z[1/6]$.  If $E$ is a fixed elliptic curve, then del Pezzo surfaces
(with blowdown structure) with anticanonical curve $E$ are classified by
maps $\phi\in \Hom(\Lambda_{E_8},E)$.  If $R\subset E_8$ is an
indecomposable subsystem, then there is an anticanonical curve of type $R$
if every root in $\ker\phi$ is contained in $R$.  In particular, if
$\Lambda_R$ is saturated in $\Lambda_{E_8}$, then the generic point in
$\Hom(\Lambda_{E_8}/\Lambda_R,E)$ will give rise to such a surface.
Moreover, we may verify by direct computation that the generic surface with
such an anticanonical fiber has non-constant $j$-invariant, at which point
dimension considerations tell us that the generic surface arises from a map
$\Hom(\Lambda_{E_8}/\Lambda_R,E)$.  It follows easily that if $R\subset R'$
with $\Lambda_R$ saturated, then the corresponding degeneration is strongly
effective in any characteristic.  In particular, every arrow in Figure 1
remains strongly effective over any field.

One is thus left to consider degenerations from $A'_7$, $A_8$, $D_8$.
Since $\Lambda_{E_8}/\Lambda_{D_8}\cong \Z/2\Z$, and
$\Lambda_{E_8}/\Lambda_{A_8}\cong \Z/3\Z$, such surfaces are classified by
points of order precisely 2 or 3 on $E$ as appropriate, and thus over
$\Z[1/6]$ cannot be degenerated.  Similarly, $A'_7$ surfaces with a fiber
isomorphic to $E$ are labelled by a $2$-torsion point and a point in $E$,
and thus degenerate to the surface of type $D_8$ labelled
by the same $2$-torsion point, but not to surfaces of type $E_7$ or $E_8$
over $\Z[1/6]$.

This also makes it relatively straightforward to construct the exotic
degenerations in small characteristic, as one need simply construct an
appropriate family of curves $E$ with a $2$- or $3$-torsion point
degenerating to the identity.  (This is how the above examples of $A_8\to
E_8$ degenerations were produced.)

\subsection{Surfaces and singularities of equations}

Now that we have introduced the decomposition of the moduli stack into
types, it is natural to ask what a given type implies about the
corresponding difference/differential equations.  The type of a blown up
surface translates into information about how the direct image of a sheaf
with given Chern class meets the anticanonical curve on $F_2$, and we
argued above that this controls the singularities of the equation.  Though
it is clear that the sheaf must meet the anticanonical curve on $F_2$ at
any singular point of the equation, and that the equation is singular at
any point in the support of the intersection, we have not so far been clear
about how the local behavior of the equation at the singularity is recorded
in the sheaf structure of the intersection.  For any given equation, of
course, one can simply turn it into a sheaf and keep track of the different
blowups; however, when classifying degenerations of a given elliptic
scenario, one generally obtains simply a list of possible types of surface,
and would then like to translate that into information about the
singularities of the equation.

We restrict our attention to the case that the sheaf on $F_2$ can be
separated from the anticanonical curve by a sequence of blowups and minimal
lifts.  This is a relatively mild restriction, since by \cite{poisson} any
sheaf which is transverse to the anticanonical curve can be put into that
form by a sequence of ``pseudo-twists'', which on difference equations
correspond to canonical gauge transformations.  In particular, the cases
which cannot be so separated necessarily involve ``apparent''
singularities, in that there is some gauge transformation making its
singularities milder.  We similarly assume that the support of the sheaf
does not contain a fiber, as again this leads to apparent singularities.

More precisely, since singularities are local in nature, we are interested
in singularities at some fixed point $x\in \P^1$, and only need to be able
to separate the sheaf from the anticanonical curve in a neighborhood of
that fiber.  We are thus given a blowup of $F_2$ in which every point blown
up is on the same fiber, and a sheaf $M$ on that blowup such that
$x\notin \rho(\pi(\supp(M)\cap C_\alpha))$.
If $M$ has first Chern class $ds+d'f-r_1e_1-\cdots-r_me_m$, then the
intersection of $\pi_*M$ with the anticanonical curve near $\rho^{-1}x$ is
independent of $d'$, and thus the question is to translate the remaining
data into local information about the equation.  Equivalently, we are given
the intersection numbers of $c_1(M)$ with the various components of the
preimage of $x$.

Again, since singularities are local, we should actually replace the
surface by a formal neighborhood of the fiber.  Helpfully, the intersection
nubmers are still well-defined even when $M$ is only a sheaf on the formal
completion of the fiber, since the various components are still proper
curves.  Now, on the formal neighborhood, $M$ has a natural direct sum
decomposition, with one summand for each point of intersection.  Moreover,
at each point of intersection, $M$ has finitely many branches.  It follows
that $M$ is an extension of invertible sheaves on unibranched curves, which
since we are working in a formal neighborhood, are isomorphic to the
structure sheaves of their supports.  Moreover, the cases in which the
branch is tangent to the relevant exceptional component are limits of cases
in which there are multiple branches, and thus we reduce (up to issues of
limits not changing the restriction to $C_\alpha$) to the case that the
branch meets the exceptional component simply.  Thus, up to the extension
and limit problem, we reduce to the case that $M$ is the structure sheaf of
the image of an appropriate map from $\Spec(k[[t]])$ to the formal
completion of the fiber.

When applying this to the problem of classifying degenerate equations,
there is one complication, however: at the elliptic level, we allowed {\em
  twisted} equations, and thus the sheaf we are given may not be on $F_2$.
Although in principle one could deal with this by working with connections
twisted by suitable line bundles, this will not give equations in the form
one usually wants to consider.  To obtain a difference or differential
equation with rational function coefficients, we need to choose a divisor
representing the twisting line bundle.  This is equivalent to choosing a
section of the Hirzebruch surface which is transverse to the anticanonical
curve.  Given such a choice, we can translate sheaves into equations by
first performing a sequence of elementary transformations to make the
chosen section disjoint from the anticanonical curve, at which point we
will be in the $F_2$ scenario where the translation is as described above.
It is straightforward to see that changing the chosen section has the
effect of gauging by a {\em scalar} function, the solution of an
appropriate first-order equation.  That is, if $M$ is a sheaf on $F_2$ and
we apply the above translation using a section $s'$ other than the
$-2$-curve, then the result is to gauge by the solution of the equation
corresponding to $\sO_{s'}$.  In particular, once we have understood
singularities in the $F_2$ case, we will be able to understand the general
case, up to an overall scalar gauge transformation freedom (which we can
also understand).

Note also that once we have restricted $M$ to a single branch (or even just
a single point of intersection), we may feel free to blow down any
component of the fiber which is a $-1$-curve not meeting $M$.  After doing
so, we obtain a surface such that every point we blow up after the first is
a point on the previous exceptional curve, as otherwise the final
exceptional locus would contain disjoint $-1$-curves.  Moreover, any time
we blow up a point which is on just one anticanonical component, the type
of surface that results is not affected by the choice of point.  In
particular, once the most recent $-1$-curve is not an anticanonical
component, we should stop, as future blowups of points on that curve will
give the same type as more general blowups on the anticanonical component
it meets.  We thus restrict our attention to types of surfaces in which
every blowup but the first blows up a point of an exceptional component and
every $-1$-curve but the last is an anticanonical component.  Call such a
type ``minimal''.  Note that we may also insist that our curve be disjoint
from the strict transform of the fiber, as again otherwise we can deform it
to one meeting a more general point of the relevant $-1$-curve.

In this way, we reduce to the following problems:
\begin{itemize}
  \item[(1)] Given a minimal type, what do the resulting maps from
    $\Spec(k[[t]])$ to $F_2$ look like?  (And can this be inverted?)
  \item[(2)] Given a map from $\Spec(k[[t]])$ to $F_2$ coming from a
    minimal type, what is the local structure of the corresponding
    equation?
  \item[(3)] What is the effect of taking limits from the typical case of a
    minimal type to more special cases?
\end{itemize}
We will not consider (3) here (except in Section \ref{sec:elldiff} for the
elliptic case), as for most applications it is the typical behavior that
matters.  Presumably, the only effect of taking such limits is to make
changes on the order of the $o()$ terms below.

Of the remaining cases, the simpler to deal with is (2), subject to the
question of which maps come from minimal types.  Supposing for the moment
that the fiber of interest is $x=\infty$ (which will be the main
interesting case in any event), then the map to $F_2$ takes the form
$(y,x,w)=(Y(t),1,W(t))$ for suitable power series $Y$ and $W$.  Moreover,
$W(t)$ cannot be identically 0, since then the image would be contained in
the fiber.  The key point now is that we may view the formal neighborhood
of the fiber as a $\P^1$-bundle over $\Spec(k[[w]])$, and perform our usual
calculations to convert between sheaves on the $\P^1$-bundle and morphisms
on the anticanonical curve.  In our case, the sheaf is the structure sheaf
of an affine curve, and is thus simply $k[[t]]$ viewed as a $k[[w]]$-module
via the map $w\mapsto W(t)$.  Since this is a totally ramified extension,
$k[[t]]$ is a free $k[[w]]$-module with basis $(1,t,\dots,t^{\ord(W)-1})$,
and the action $M_t$ of multiplication by $t$ in this basis is the
companion matrix of the minimal polynomial of $t$ over $k[[w]]$.  The
associated morphism of vector bundles on $F_2$ is then simply $B:=y -
Y(M_t)\in \Mat_n(k[[w]][y])$.  Note that in characteristic 0, we may change
variables so that $W(t)=t^{\ord(W)}$.

If the point $(0,1,0)$ we are blowing up is a typical (i.e., not ramified
for the involution) smooth point of the fiber, then the condition that the
spectral curve be disjoint from the strict transform of the fiber forces
$\ord(W)\le \ord(Y)$.  Indeed, the spectral curve contains $(0,1,0)$ with
multiplicity $\min(\ord(W),\ord(Y))$ and meets the fiber with multiplicity
$\ord(W)$, so after blowing up meets $f-e_1$ with multiplicity
$\ord(W)-\min(\ord(W),\ord(Y))$, which must vanish.  We then find that the
spectral curve meets the original anticanonical curve with multiplicity at
least $\ord(W)$, which must therefore be $1$.  This, of course, translates
directly to the case of a simple singularity: $A$ vanishes to order $1$ at
$(0,1,0)$ and has a simple pole at the image under the involution.

The next simplest case to consider is actually the differential case, where
we assume characteristic 0.  Here the anticanonical curve is $y^2=0$, and
the reduced curve has a natural parametrization $(0,1,1/z)$, so that in
terms of the parameter we have $z=W(t)^{-1}=t^{-\ord(W)}$.  We again find
that to avoid the strict transform of the fiber, we must have $\ord(W)\le
\ord(Y)$.  The resulting Higgs bundle then has the form $A(z)^{-t} =
-Y(M_t)|_{w=1/z}$.  Note that we may view the resulting equation as one
over the field of Puiseux series; if $W(t)=t^a$, $\ord(Y(t))=b$, then the
equation has the form
\[
\frac{v'(z)}{v(z)} = z^{b/a} \sum_{0\le l} c_l z^{-l/a-2},
\]
where $t^b/Y(t) = \sum_{0\le l} c_l t^l$.  (Here the factor of $z^{-2}$
comes from the fact that we need to work relative to a differential which
is holomorphic at $z=\infty$.)  This should be compared with the
classification of the local behavior of differential equations
\cite{vandenEssenA/LeveltAHM:1992}, though we will see below that the
combinatorics of the surface actually encodes more subtle information about
which coefficients are nonzero.

The nonsymmetric $q$-difference case is the next simplest, as the
anticanonical curve is $y^2-xwy=0$, and thus we can easily parametrize both
branches:
\[
(y,x,w) = (0,1,1/z)\quad\text{and}\quad (y,x,w) = (z,1,z);
\]
here we have parametrized so that the singular point is $\infty$ on one
branch and $0$ on the other branch.  We again have $\ord(W)\le \ord(Y)$,
and the equation has
\[
A(t) = \frac{Y(t)}{Y(t)-W(t)}\bigg|_{t\mapsto M_t}.
\]
(Here we use the fact that we only care about the conjugacy class of $M_t$
over $k[[w]]=k[[1/z]]$, and $M_t$ and its transpose are conjugate.)

The nonsymmetric difference case is analogous, with the parametrizations
now being
\[
(y,x,w) = (0,1,1/z)\quad \text{and}\quad (y,x,w)=(z^2,1,z),
\]
of the anticanonical curve $y^2-w^2y=0$.  Again, $\ord(W)\le \ord(Y)$, and
the equation is
\[
A(t) = \frac{Y(t)}{Y(t)-W(t)^2}\bigg|_{t\mapsto M_t}.
\]

The symmetric $q$-difference case is only slightly more complicated.  The
two branches on the anticanonical curve of the corresponding double cover
correspond to the parametrizations
\[
(y,x,w) = (\eta z/(z^2+\eta)^2,1,z/(z^2+\eta))\quad\text{and}\quad
(y,x,w) = (z^3/(z^2+\eta)^2,1,z/(z^2+\eta))
\]
of the irreducible anticanonical curve $y^2-xwy+\eta w^4 = 0$.
The key point is that although $w$ is a quadratic function of $z$, the
corresponding extension of $k[[w]]$ is unramified.  We may thus let $Z(t)$
be the non-holomorphic solution of
\[
Z(t)+\frac{\eta}{Z(t)} = \frac{1}{W(t)}
\]
in $k((t))$ and view $M_t$ as a matrix over $k[[1/z]]$ via $w\mapsto
z/(z^2+\eta)$, to obtain an equation of the form
\[
A(t)
=
\frac{\eta Z(t)/(Z(t)^2+\eta)^2-Y(t)}
     {Z(t)^3/(Z(t)^2+\eta)^2-Y(t)}\bigg|_{t=M_t}.
\]
Note again that $\ord(W(t))=-\ord(Z(t))\le \ord(Y(t))$.

The symmetric difference case is more complicated, as now the
anticanonical curve is ramified over the base of the
ruling, with natural parametrizations (assuming characteristic 0)
\[
(y,x,w) = (1/z^3,1,1/z^2)\quad\text{and}\quad (y,x,w) = (-1/z^3,1,1/z^2)
\]
of the anticanonical curve $y^2 = w^3 x$.  Note that as before, we have
$\ord(W(t))\le \ord(Y(t))$, and since the characteristic is 0 we may assume
that $W(t)=t^a$.  Let $Y(t) = \sum_{b\le l} c_l t^l$, with $c_b\ne 0$.
Then the eigenvalues of $Y(M_t)$ are the Puiseux series of the form
\[
\sum_{b\le l} c_l z^{-2l/a},
\]
one for each of the $a$ roots $z^{1/a}$, and thus the eigenvalues of $A(z)$
are the different values of
\[
\frac{1/z^3-\sum_{b\le l} c_l z^{-2l/a}}
     {-1/z^3-\sum_{b\le l} c_l z^{-2l/a}}.
\]
Note that if $a$ is odd, then the eigenvalues are all $\pm 1$ at
$z=\infty$, depending on the sign of $3-2b/a$.

\medskip

For (1), the answer is in one sense trivial: given a (spectral) map from
$\Spec(k[[t]])$ to $F_2$, we simply repeatedly blow up the image of the
closed point until the closed point is no longer on the anticanonical
curve.  This is reasonably straightforward, but as stated tells us little
about the inverse problem.  After each blowup, we obtain a patch of $X_m$
containing the image of the special fiber, and a morphism from that patch
to $X_{m-1}$.  Apart from the parametrized spectral curve itself, the other
main piece of information on each patch is the equation of the
anticanonical curve.  A key observation is that after the first few
blowups, there are at most two (local) components of the anticanonical
curve, and those components (which may have multiplicity) meet
transversely.  In other words, after a few steps, we can coordinatize the
resulting patch so that the anticanonical curve has the equation
$u^{m_1}v^{m_2}=0$, where the most recent exceptional curve is $u=0$.  (If
$m_2=0$, there is some ambiguity in the coordinate $v$, but we always at
least choose it so that the point being blown up is the origin.)  Thus the
question factors into understanding the first few steps and understanding
this general scenario.  Note that since we are assuming that we never blow
up a point of an exceptional curve which is not an anticanonical component,
we may assume that $m_1>0$.

Let $u(t)$, $v(t)\in t k[[t]]$ be the parametrization of the spectral curve
in this patch.  Note that we may assume that $u(t)^{m_1}v(t)^{m_2}$ is not
identically 0, since otherwise the spectral curve would not be transverse
to the anticanonical curve.  After blowing up the origin, there are three
possibilities for the new affine patch.  The new exceptional curve is
$\hat{u}=0$, and the other relevant component (if any) of the anticanonical
curve is $\hat{v}=0$.
\begin{itemize}
  \item[(a)] If $\ord(u)>\ord(v)$, then the new patch has coordinates
    $\hat{u}(t) = v(t)$, $\hat{v}(t)=u(t)/v(t)$, with
    $\hat{m}_1=m_1+m_2-1$, $\hat{m}_2=m_1$.
  \item[(b)] If $\ord(u)<\ord(v)$, then the new patch has coordinates
    $\hat{u}(t) = u(t)$, $\hat{v}(t)=v(t)/u(t)$, with
    $\hat{m}_1=m_1+m_2-1$, $\hat{m}_2=m_2$.
  \item[(c)] If $\ord(u)=\ord(v)$, then the new patch has coordinates
    $\hat{u}(t)=u(t)$, $\hat{v}(t)=v(t)/u(t)-\alpha$, where $\alpha=\lim_{t\to
    0} v(t)/u(t)$, $\hat{m}_1=m_1+m_2-1$, $\hat{m}_2=0$.
\end{itemize}
Note that if $m_2=0$, then (b) and (c) give the same combinatorial type, so
we merge them into a case (bc).  Geometrically, these cases correspond to
the case that the spectral curve is tangent to $u=0$, tangent to $v=0$, or
tangent to neither, respectively.  Every time we hit case (c) or (bc) we
pick up a parameter $\alpha$, which is in $k^*$ when $m_2\ne 0$ and $k$
when $m_2=0$.

If $m_1=1$, $m_2=0$, then the intersection with the anticanonical curve
depends only on the order of vanishing $\ord(u)$ (which by our genericity
assumptions must be $1$).  Note here that the choice of point to blow up on
the exceptional curve was made at the {\em previous} step, and thus there
are no more parameters (since this blowup separates the spectral and
anticanonical curves).  We find the following by an easy induction.
\begin{itemize}
  \item[(1)] If $m_1,m_2>0$, then the intersection of the image of the
    structure sheaf with $C_\alpha$ depends only on
    $
    u(t)+o(t^{m_1\ord(u)+(m_2-1)\ord(v)})$ and
    $v(t)+o(t^{(m_1-1)\ord(u)+m_2\ord(v)})$.
  \item[(2)] If $m_1>0$, $m_2=0$, then it depends only on
    $
    u(t)+o(t^{m_1\ord(u)})$ and $v(t) + o(t^{(m_1-1)\ord(u)})$.
\end{itemize}
In each case, we either have $(m_1,m_2)=(1,0)$ or have enough information
to determine in which case (a--c) we are in, and in case (c) what the value
of the parameter is.  Moreover, in each case, we find that we know the new
coordinates to at least the correct precision.

Note in particular that if $(m_1,m_2)=(1,1)$, then the combinatorial type
depends only on $\ord(u)$ and $\ord(v)$.  Moreover, we see that the blowing
up process is essentially just performing the subtraction form of Euclid's
algorithm for computing $\gcd(\ord(u),\ord(v))$, and the additional data
determining the intersection is just the constant term of
\[
u(t)^{\ord(v)/\gcd(\ord(u),\ord(v))}v(t)^{-\ord(u)/\gcd(\ord(u),\ord(v))}.
\]

The nonreduced cases (i.e., $\max(m_1,m_2)>1$) are more complicated to deal
with, and we can only give a satisfying description in characteristic 0.
The advantage of working in characteristic 0 is that we can take roots of
power series.  In particular, we know $u$ to at least enough precision to
know its leading term, and thus we may choose a new coordinate $s$ such
that $u(s) = s^{\ord(u)}$, and find that we now need to specify only the
coordinate $v(s)$, to precision $o(s^{(m_1-1)\ord(u)+m_2(\ord(v))})$.  If
$m_2>0$, we could instead reparametrize so that $v(s)=s^{\ord(v)}$, so that
the only data is $u(s)+o(s^{m_1\ord(u)+(m_2-1)\ord(v)})$.  Note that
knowing
\[
v(s) + o(s^{(m_1-1)\ord(u)+m_2(\ord(v))})
\]
tells us
\[
v(s)^{1/\ord(v)} + o(s^{1+(m_1-1)\ord(u)+(m_2-1)(\ord(v))}),
\]
and thus tells us the inverse function to the same precision.  In other
words, when both parametrizations are defined, the information we obtain is
the same in either case.

To proceed further, we need to understand better how the two
parametrizations are related.  To begin with, let $H\subset \N$ be a
submonoid (i.e., $0\in H$ and $H$ is closed under addition).  A power
series supported on $H$ is then an element of $k[[t]]$ in which the
coefficient of $t^l$ is 0 unless $l\in H$.

\begin{lem}
  For any submonoid $H\subset \N$, the power series supported on $H$ form a
  ring, in which the series with nonzero constant terms are units.  If
  $f\in k[[t]]$ is supported on $H$ with $f(0)=0$, then for any other element
  $g\in k[[t]]$, $g\circ f$ is supported on $H$.
\end{lem}

\begin{proof}
  That the power series supported on $H$ form a ring is trivial, and thus
  any polynomial in such a power series is also supported on $H$.  Since
  the set of power series supported on $H$ is closed under formal limits,
  the claim regarding composition follows.  That the reciprocal of a series
  with nonzero constant term is still supported on $H$ then follows by
  plugging in $f-f(0)$ into the appropriate geometric series.
\end{proof}

More surprisingly, we have the following.

\begin{lem}
  For any submonoid $H\subset \N$, the power series with $f(0)=0$,
  $f'(0)\ne 0$ and $f(t)/t$ supported on $H$ form a group under composition.
\end{lem}

\begin{proof}
  Suppose $g$ is such a series, with
  \[
  g(t) = \sum_{n\in H} c_n t^{n+1},
  \]
  $c_0\ne 0$.  Then
  \[
  g(f(t)) = \sum_{n\in H} c_n (f(t))^{n+1}
  = \sum_{n\in H} c_n t t^n (f(t)/t)^{n+1}.
  \]
  Since $(f(t)/t)^{n+1}$ is supported on $H$ and $t^n$ is supported on $H$,
  we find that these series are indeed closed under composition.

  It remains only to show that the compositional inverse of such a series
  is again of this form.  Let $f(t)$ be such a series, and let $g(t)$ be
  its compositional inverse.  Suppose $g(t)/t$ is not supported on $H$, and
  let $n$ be the minimal nonnegative integer not in $H$ such that the
  coefficient of $t^{n+1}$ in $g(t)$ is nonzero.  We may thus write
  \[
  g(t) = g_0(t) + c_n t^{n+1} + o(t^{n+1}),
  \]
  where $g_0(t)/t$ is supported on $H$.  We then have
  \[
  t = g(f(t)) = g_0(f(t)) + c_n t^{n+1} (f(t)/t)^{n+1} + o(t^{n+1}),
  \]
  so that
  \[
  1-\frac{g_0(f(t))}{t} = c_n t^n (f(t)/t)^{n+1} + o(t^n).
  \]
  The left-hand side is supported on $H$, and thus the coefficient of $t^n$
  on the left vanishes, while on the right it equals
  \[
  c_n f'(0)^{n+1}\ne 0,
  \]
  giving the desired contradiction.
\end{proof}

In particular, the monoid generated by the exponents of the nonzero terms
of $s^{-\ord(v)}v(s)$ in the $u(s)=s^{\ord(u)}$ parametrization will be the
same as that coming from the other parametrization, when both
parametrizations are valid, and this monoid will be unchanged under blowup
as long as we remain in scenarios (a) and (b).  Unfortunately, scenario (c)
can change the monoid in somewhat unpredictable ways.  Luckily, we can
enlarge the monoid without losing information about the blowup.

Call a subset $H\subset \N$ ``arithmetic'' if it contains $0$ and for all
$n\in H$, also contains $n+\gcd(\{l:l\in H|l\le n\})$.  Such a set is
certainly a monoid, and we can still consider the arithmetic set generated
by the exponents of nonzero terms of a power series.  Note that an
arithmetic set has a minimal set of generators: $d_1$ is the minimal
nonzero element, $d_2$ the minimal element not a multiple of $d_1$, $d_3$
the minimal element not a multiple of $\gcd(d_1,d_2)$, etc.

We associate an arithmetic set to our configuration $u(t),v(t)$ as follows.
If $m_1,m_2>0$, then we change coordinates so that $u(s)=s^{\ord(u)}$, and
take the arithmetic set generated by $\gcd(\ord(u),\ord(v))$ and the
exponents of the nonzero terms of $s^{-\ord(v)}v(s)$.  (Here, by
convention, we suppose that all terms past the required precision are
nonzero.)  If $m_1>0$, $m_2=0$, then we take the same coordinate change,
but now the arithmetic set is generated by $\ord(u)$ and the exponents of
nonzero terms of $v(s)$.

For $m_1,m_2>0$, we find that blowing up has no effect on the arithmetic
set, while for $m_2=0$, scenario (bc) replaces the arithmetic set by
$(H-\ord(u))\cap \N$,
while scenario (a) replaces it by
$
H\cup \gcd(\ord(u),\ord(v))\N$.
Moreover, for $m_2=0$, we can detect which of (a) or (bc) we are in knowing
only $\ord(u)$ and $H$; scenario (a) is the case that $H$ contains a
nonzero element less than $\ord(u)$.

\medskip

Let us now see how this information gets used for the different types of
equations.  The $q$-difference cases are now the simplest, as the
anticanonical curve can now already be written as $uv=0$ in suitable
(formal) coordinates.  For the nonsymmetric $q$-difference case, the two
coordinates are $y$ and $w-y$, so the surface determines $a:=\ord(Y(t))$,
$b:=\ord(W(t)-Y(t))$, and
\[
\gamma:=\lim_{t\to 0}\frac{Y(t)^b}{(W(t)-Y(t))^{a}}.
\]
Moreover, the fact that we end up with multiplicity $1$ implies that
$\gcd(a,b)=1$.  If we write $Y(t) = \alpha t^a (1+O(t))$, $W(t)-Y(t)=\beta
t^b (1+O(t))$, then we have three cases.  If $a<b$, then $W(t) = \alpha
t^a (1+O(t))$, so we may as well reparametrize to make $\alpha=1$,
$W(t)=t^a$, and thus the equation has the symbolic form
\[
\frac{v(qz)}{v(z)} = \gamma^{1/a} z^{-1+b/a} (1+O(z^{-1/a})),
\]
where we translate Puiseux series to matrices as above.
Similarly, in the case $a>b$, we have $W(t) = \beta t^b (1+O(t))$,
and again may take $\beta=1$ to give
\[
\frac{v(qz)}{v(z)} = \gamma^{1/b} z^{1-a/b} (1+O(z^{-1/b})).
\]
Finally, if $a=b=1$, then $W(t) = (\alpha+\beta) (1+O(t))$, and the equation is
\[
\frac{v(qz)}{v(z)} = \gamma (1+O(z^{-1})).
\]
This is essentially the same as the formal classification of $q$-difference
equations given in \cite{RamisJ-P/SauloyJ/ZhangC:2013}.  We obtain the same
three cases in the symmetric $q$-difference case, as the needed
reparametrizations do not affect the leading terms of the numerator and
denominator of $A$.

For the differential case, we start with $(m_1,m_2)=(2,0)$ with $u=y$,
$v=w$.  The above description of the spectral curve thus has the form
\begin{align}
  Y(t) &= u(t) = t^a\\
  W(t) &= v(t) = \sum_{0\le l\le a-b} c_l t^{l+b} + o(t^a),
\end{align}
with $c_b\ne 0$, where the corresponding arithmetic set is determined by
the degrees where the $\gcd$s of the exponents up to that point (with $a$)
drop.  Note that the assumption that we meet the last $-1$-curve with
multiplicity 1 implies that the $\gcd$ actually reaches $1$ before we hit
$o(t^a)$.  Moreover, our assumption of disjointness from $f-e_1$ implies
$b\le a$, and thus we could instead reparametrize to have
\[
Y(t) = t^a(\sum_{0\le l\le a-b} e_l t^l+o(t^{a-b})),\quad
W(t) = t^b,
\]
where the $e_l$ are determined from $c_l$ and generate the same arithmetic
set.  This gives an equation of the form
\[
\frac{zv'(z)}{v(z)} = \sum_{0\le l\le a-b} f_l z^{(a-l)/b-1} +o(1),
\]
where again the arithmetic set is the same.  Note that we can record the
combinatorial information in the equivalent form of specifying the leading
exponent and each later exponent where the common denominator increases.
Also, although the translation from the parameters of the surface to the
coefficients $f_l$ is fairly complicated, there is a one-to-one
correspondence between those parameters and the potentially nonzero
coefficients of the Puiseux series, such that each $f_l$ depends only the
parameters up to that point, and is degree 1 in the parameter corresponding
to $l$, unless the parameter lies in $k^*$ in which case it is proportional
to an appropriate power of the parameter.  In particular, if we are
considering two irreducible singularities that share the first few steps of
the blowup, then the corresponding Puiseux series must agree in precisely
the first $N$ terms (among those allowed to be nonzero), where $N$ is the
number of shared parameters.

For instance, suppose we are given an irreducible singularity of the form
\[
\frac{zv'(z)}{v(z)}
=
f_0 z^3 + f_6 z^2
+f_{12}z+f_{14}z^{2/3}+f_{15}z^{1/2}+f_{16}z^{1/3}
+f_{17}z^{1/6}+f_{18} + o(1).
\]
Assuming $f_0$, $f_{14}$, $f_{15}$ are nonzero, then the combinatorial data
is determined by the sequence $(3,2/3,1/2)$, and the corresponding moduli
space can be identified with $(k^*)^3\times k^5$.  The corresponding
sequence of $(m_1,m_2)_{a/b/c}$ is:
\begin{align}
&(2,0)_a(1,2)_b(2,2)_b{\bf (3,2)_c}{\bf (4,0)_{bc}}{\bf (3,0)_{bc}}\notag\\
&(2,0)_a(1,2)_b{\bf (2,2)_c}\notag\\
&(3,0)_a{\bf (2,3)_c}{\bf (4,0)_{bc}(3,0)_{bc}(2,0)_{bc}(1,0)_{bc}},\notag
\end{align}
where we have indicated the steps introducing parameters in bold.  We
should recall here that the subscript on a tuple $(m_1,m_2)_{a/b/c}$
determines the {\em next} point to be blown up.  Note that the three
parameters from $k^*$ are
\[
-\frac{1}{f_0}, \frac{f_0^6}{f_{14}}, \frac{f_{14}^9}{9 f_0^{14}f_{15}^2},
\]
and the remaining parameters in $k$ are even more complicated functions of
the coefficients.  The anticanonical curve on the resulting surface has the
decomposition
\begin{align}
&2(s+f-e_1-e_2-e_3-e_4)+(e_1-e_2)+2(e_2-e_3)+3(e_3-e_4)+4(e_4-e_5)+3(e_5-e_6)\notag\\
{}+{}&2(e_6-e_7-e_8-e_9)+(e_7-e_8)+2(e_8-e_9)\notag\\
{}+{}&3(e_9-e_{10}-e_{11})+2(e_{10}-e_{11})+4(e_{11}-e_{12})+3(e_{12}-e_{13})+2(e_{13}-e_{14})+(e_{14}-e_{15}).\notag
\end{align}
Note that although setting
a coefficient other than $f_0$, $f_{14}$, $f_{15}$ to 0 has no effect on
the combinatorics, setting one of the critical coefficients to 0 can make
significant changes.  For instance, the generic subcase with $f_{15}=0$ has
one fewer blowup, so this does not correspond to a degeneration of surfaces
as considered above.

The one complication in the nonsymmetric difference case is that we do not
start in a configuration $u^{m_1}v^{m_2}=0$, since the two branches are
tangent.  We must blow up twice in order to achieve this, and this leads to
several cases.  Note first that blowing up a general point of $e_1$ yields
an equation
\[
\frac{v(z+1)}{v(z)} = 1+\alpha/z + o(1/z),
\]
$\alpha\ne 0$, which leaves the case that we blow up the triple
intersection.  This yields an anticanonical decomposition
$(s+f-e_1-e_2)+(s+f-e_1-e_2)+(e_1)+2(e_2)$, and there are four
possibilities for the next blowup: a general point $\lambda$ of $e_2$, or
the intersection with one of the other three components.  The resulting
patches, with the equation of the anticanonical curve and the map to the
original coordinates are:
\begin{itemize}
\item[(1)] $u^2=0$, where $y=u^2 (v+\lambda)$, $w=u$.
\item[(2)] $u^2 v = 0$, where $y=u^2 v$, $w=u$.
\item[(3)] $u^2 v = 0$, where $y=u^2 (v+1)$, $w=u$.
\item[(4)] $u^2 v =0$, where $y=u^2 v$, $w=uv$.
\end{itemize}
In case (1), if we take $u(t) = t^a$, $v(t) = \sum_{b\le l\le a} c_l t^l+o(t^a)$,
then the resulting equation has the symbolic form
\[
\frac{f(z+1)}{f(z)}
=
\frac{\lambda+v(z^{-1/a})}{\lambda-1+v(z^{-1/a})} + o(z^{-1})
=
\frac{\lambda}{\lambda-1} + \sum_{b\le l\le a} d_l z^{-l/a} + o(z^{-1}),
\]
where $d_b\ne 0$ and the common denominator increases at the same places as
$v(z^{-1/a})$.  In case (2), we now know $v(t)$ to relative precision
$o(t^a)$, and now the equation is
\[
\frac{f(z+1)}{f(z)}
=
\frac{v(z^{-1/a})}{v(z^{-1/a})-1}(1+o(z^{-1}))
=
\sum_{b\le l\le a+b} d_l z^{-l/a} + o(z^{-b/a-1}),
\]
where again the common denominator increases at the same points as
$v(z^{-1/a})$.  Case (3) is analogous, and simply gives the dual equation.
Finally, for case (4), we want to parametrize so that $u(t)v(t)=t^a$,
and have $v(t) = \sum_{b\le l\le a} c_l t^l + o(t^a)$ with $c_b\ne 0$,
giving the equation
\[
\frac{f(z+1)}{f(z)}
=
\frac{1}{1-v(z^{-1/a})}
=
1 + \sum_{b\le l\le a} d_l z^{-l/a} + o(z^{-1}),
\]
now with $a>1$, since $e_3$ is a multiple component of the fiber in this
case.  Again, these agree with the formal classification of difference
equations \cite{TurrittinHL:1960,PraagmanC:1983}, apart from the
description of the combinatorial data.

The symmetric difference case is, naturally, even more complicated, as now
we may need to blow up three times to have just a pair of transverse
branches to consider.  The simplest case is that we blow up a general point
of $e_1$, which gives an equation of symbolic form
\[
\frac{f(z+1)}{f(z)} = \exp(\alpha/z+O(1/z^3)).
\]
(Here we ignore the effect that the shift has on the symmetry condition.)
Similarly, blowing up the point of tangency then a general point of $e_2$
gives an equation of symbolic form
\[
\frac{f(z+1)}{f(z)} = -\exp(\alpha/z+O(1/z^3)).
\]
Otherwise, we are blowing up the triple intersection, and again have four cases.
\begin{itemize}
\item[(1)] $u^2=0$, where $y=u^3 (v+\lambda)^2$, $w=u^2(v+\lambda)$.
\item[(2)] $u^2 v = 0$, where $y=u^3 v^2$, $w=u^2 v$.
\item[(3)] $u^2 v = 0$, where $y=u^3 (v+1)^2$, $w=u^2(v+1)$.
\item[(4)] $u^2 v =0$, where $y=u^3 v$, $w=u^2v$.
\end{itemize}

In case (1), we reparametrize so that $u(t)^2(\lambda+v(t))=t^{2a}$,
corresponding to $t=z^{-1/a}$, giving symbolic equations of the form
\[
\frac{f(z+1)}{f(z)}
=
\frac{1-\sqrt\lambda}{1+\sqrt\lambda}
\exp(\sum_{1\le l\le a} c_l z^{-l/a}+o(1/z)).
\]
Of course, the symmetry means that these equations come in pairs,
corresponding to the fact that the full equation has even order.
Case (3) is the next simplest, as we may still reparametrize so that
$u(t)^2(1+v(t))=t^{2a}$, giving again pairs of symbolic forms typified by
\[
\frac{f(z+1)}{f(z)}
=
\frac{\sqrt{1+v(z^{-1/a})}-1}{\sqrt{1+v(z^{-1/a})}+1}
=
\sum_{b\le l\le a+b} c_l z^{-l/a} + o(z^{-1-b/a})
\]
For (2) and (4), the natural parametrization is $z^{-2}=u(t)^2v(t) =
t^{2a+b}$, where $u$ and $v$ have order $a$ and $b$ respectively, and are
known to relative precision $o(t^a)=o(z^{-2a/(2a+b)})$.  Case (4) then
gives an equation of the form
\[
\frac{f(z+1)}{f(z)}
=
\exp(\sum_{0\le l\le a} c_l z^{-(b+2l)/(2a+b)} + o(1/z)).
\]
If $b$ is odd, this equation satisfies the symmetry condition, while if $b$
is even, then the overall order is even and the symbolic equations come in
pairs.

\medskip

It is worth noting in each case how the corresponding canonical
isomonodromy transformations behave.  Note that when performing the
isomonodromy transformation corresponding to twisting by $e_l$, the effect
is to perform the $d$-th iterate of the twist by $e_m$, where $d$ is the
multiplicity of $e_m$ in the pullback of $e_l$.  We thus need only consider
the twist by $e_m$.  The result is then to replace the sheaf by the sheaf
of functions vanishing at that point, and is thus easily seen to correspond
to gauging by $M_t$ (which we may view symbolically as the appropriate
power of $z$).  In the $q$-difference cases, gauging by
$z^{-1/a}+o(z^{-1/a})$ multiplies the leading coefficient by $q^{-1/a}$;
since the true parameter is the $a$-th power of the leading coefficient,
this shifts the parameter by $q$ as expected.  In the nonsymmetric
ordinary difference case, gauging by $z^{-1/a}+o(z^{-1/a})$ multiplies $A$
by $1-1/az+o(1/z)$, and similarly for the symmetric case.  And of course in
the differential case gauging by $z^{-1/a}$ adds $-1/a+o(1)$ to $zf'/f$.

In the differential case, we of course also have continuous isomonodromy
transformations, which in the Puiseux form correspond to equations
\[
\frac{f_u}{f} = \frac{a}{l+a} z^{l/a+1}
\]
giving
\[
\frac{d}{du} A = z^{l/a}.
\]
In particular, for a minimal type with $l$ parameters, we obtain $l-1$ such
deformations.  (There is also a deformation changing the location
of the singularity.)

In the nonsymmetric difference case, if we write the symbolic equation as
\[
\frac{f(z+1)}{f(z)} = z^{b/a} \exp(g(z^{-1/a}))
\]
then the equation
\[
f_u(z) = \frac{a}{a+l} z^{1+l/a} f(z)
\]
gives
\[
\frac{d}{du} g(z^{-1/a}) = z^{l/a}(1+o(1/z))
\]
again giving continuous deformations for every parameter but the last.
The symmetric difference case is analogous; the only change is that the
continuous equation needs to preserve the symmetry, but this is easy to
arrange.

If we take into account the symmetries (e.g., the $\PGL_2$ symmetry in the
differential case), we find in general that the net number of {\em local}
continuous isomonodromy deformations is 0 if the anticanonical curve is
reduced, and otherwise can be expressed as 
$1-C_\alpha\cdot (C_\alpha-C_\alpha^{\text{red}})$.
This follows by an easy induction: blowing up a point of the smooth locus
of $C_\alpha$ subtracts $e_m$ from $C_\alpha^{\text{red}}$, and similarly for
blowing up a point on two components, while blowing up a point on a single
component with nontrivial multiplicity leaves $C_\alpha^{\text{red}}$ alone.  One
thus reduces to the cases in which $C_\alpha$ first becomes nonreduced:
$C_\alpha=2(s+f)$, which has $-3$ deformations due to the symmetries,
$C_\alpha=(s+f-e_1-e_2)+(s+f-e_1-e_2)+(e_1-e_2)+2(e_2)$, with no continuous
deformations, and
$C_\alpha=(2s+2f-2e_1-e_2-e_3)+(e_1-e_2-e_3)+(e_2-e_3)+2(e_3)$, likewise.

The fact that the result depends only on the geometry suggests that there
should be a geometric explanation of these continuous isomonodromy
deformations, just as the discrete isomonodromy deformations come from
twists by line bundles.  Of course, this assumes that the local
deformations can in fact be glued together to form global deformations,
though the results of \cite{symmPVI} suggest that this is indeed the case.

\section{Moduli of sheaves on surfaces}

Let $(X,C_\alpha)$ be an anticanonical rational surface.  Say a coherent
sheaf on $X$ has {\em integral support} if its $0$-th Fitting scheme is an
integral curve on $X$, and it contains no $0$-dimensional subsheaf.

\begin{thm}\label{thm:mod_irr}
  Let $(X,C_\alpha)$ be an anticanonical rational surface over an
  algebraically closed field of characteristic $p$, let $D$ be a divisor
  class with generic representative an integral curve disjoint from
  $C_\alpha$, and let $r$ be the largest integer such that $D\in r\Pic(X)$.
  Then the moduli problem of classifying sheaves $M$ on $X$ with integral
  support, $c_1(M)=D$, $\chi(M)=x$, and $M|_{C_\alpha}=0$ is represented by
  a quasiprojective variety $\Irr_X(D,x)$ of dimension $D^2+2$, with a
  symplectic structure induced by any choice of nonzero holomorphic
  differential on $C_\alpha$.  Moreover, $\Irr_X(D,x)$ is unirational if
  the generic representative of $D$ has no cusp, separably unirational if
  $p=0$ or $\gcd(x,r,p)=1$, and rational if $x\bmod r\in \{1,r-1\}$.
  Finally, if $\gcd(x,r)=1$, then there exists a universal sheaf over
  $\Irr_X(D,x)$.
\end{thm}

\begin{proof}
  Quasiprojectivity follows from the standard GIT construction: for any
  choice of stability condition, a sheaf with integral support is stable.
  The symplectic structure follows from the fact that sheaves with integral
  support are simple (have no nonscalar endomorphisms) together with the
  results of \cite{poisson} (see also
  \cite{HurtubiseJC/MarkmanE:2002b,BottacinF:1998}).  This requires a
    choice of Poisson structure on $X$, or equivalently a choice of nonzero
    holomorphic differential on $C_\alpha$; the symplectic structure on the
    moduli space scales linearly with the choice of differential.

  Now, the typical sheaf in the moduli space corresponds to a pair $(C,M)$
  where $C$ is an integral curve of class $D$ (and disjoint from $C_\alpha$)
  and $M$ is a torsion-free sheaf on $C$.  If $g=D^2/2+1$, then
  $\Gamma(\sO_X(D))$ has dimension $g+1$, and thus the integral curves
  in the linear system form an open subset of a $\P^g$.  We also compute
  that $C$ has arithmetic genus $g$; the fiber over the point corresponding
  to $C$ is a compactification of $\Pic^{x+g-1}(C)$, so has dimension $g$.  (As
  we might expect from a natural fibration of a symplectic scheme by
  half-dimensional subschemes, this is a Lagrangian fibration.)

  If $x=1$, then $\deg(M)=g$, and thus the generic such sheaf has a unique
  global section.  The quotient by the corresponding trivial subsheaf is
  supported on $g$ points, thus giving a birational correspondence with the
  punctual Hilbert scheme $X^{[g]}$.  Since symmetric powers of rational
  surfaces are rational varieties, it follows that $\Irr_X(D,1)$ is
  rational.

  More generally, since $D/r$ is a primitive element of the Picard lattice
  of $X$, there exists a divisor $D'$ such that $D\cdot D'=r$.  In
  particular, we can twist by powers of this divisor to obtain isomorphism
  $\Irr_X(D,x)\cong \Irr_X(D,r+x)$ for any $x$.  Similarly, the duality
  morphism $M\mapsto \sHom(M,\omega_C)$ on invertible sheaves can be
  defined globally (and extended to torsion-free sheaves) by $M\mapsto
  \sExt^1(M,\omega_X)$, and gives an isomorphism $\Irr_X(D,x)\cong
  \Irr_X(D,2g-2-x)$.  Since $2g-2=D^2$ is a multiple of $r^2$, the
  rationality claim follows.

  For unirationality, note that since $2g-2$ is a multiple of $r^2$, $g-1$
  is a multiple of $r$, and thus $\Irr_X(D,2-g)$ (classifying sheaves of
  degree 1) is rational.  Since the generic sheaf is an {\em invertible}
  sheaf on the generic curve, we can take its $d$-th power and thus obtain
  a rational map $\Irr_X(D,2-g)\to \Irr_X(D,d+1-g)$.  Since the generic
  curve $C_{\text{gen}}$ is integral, this map is dominant unless
  $C_{\text{gen}}$ is cuspidal and $d$ is 0 in $k$.   This implies
  unirationality in the noncuspidal case for any $d\ne 0$; the case $d=0$
  reduces to the case $d=g-1$ by twisting.  When $\gcd(d,p)=1$, the
  multiplication by $d$ map is separable, and again we may feel free to add
  multiples of $r$ to make this happen.

  Finally, we note that the obstruction to the existence of a universal on
  $\Irr_X(D)$ is given by a class in $H^2(\Irr_X(D),\G_m)$, which can be
  constructed as follows.  There certainly exists a universal sheaf
  \'etale-locally, and we may twist by a line bundle to ensure that this
  universal sheaf is acyclic; i.e., that $x\ge g$.  Then, since the fibers
  of the universal sheaf are simple, the endomorphism ring of the direct
  image descends to an Azumaya algebra on $\Irr_X(D)$, the class of which
  is the desired obstruction.  This Azumaya algebra has degree $x$, and
  thus the obstruction has order dividing $x$.  Since this is true for any
  twist of sufficiently large degree, we also find that the order of the
  obstruction divides $x+r$, and thus that it divides $\gcd(x,r)=1$.  In
  particular, if $x$ and $r$ are relatively prime, there is no obstruction.
\end{proof}

\begin{rems}
  The generically cuspidal case can of course only occur in finite
  characteristic, but can certainly occur there, say if $D$ is the class of
  a fiber in a rational quasi-elliptic surface.
\end{rems}

\begin{rems}
  Often in the literature, one restricts ones attention to the subscheme
  where $C$ is not just integral but smooth, making the fibers of the
  Lagrangian fibration abelian varieties.  Of course, this is problematical
  in finite characteristic, where there may not be any smooth curves in the
  linear system.  In addition, since singularity is a codimension 1
  condition, this removes an entire hypersurface from the moduli space,
  based on a condition which is rather unnatural from the difference
  equation perspective.  (Indeed, as we mentioned, difference equations
  correspond most naturally to sheaves on noncommutative surfaces, and
  there the notion of support fails altogether.  In contrast, the failure
  of integrality corresponds to reducibility of the equation in a suitable
  sense.)  For instance, in the generic $2$-dimensional case, both the
  surface and the moduli space are elliptic surfaces, and there are 12
  fibers where the support is singular.  Similarly, there are 12 points of
  the moduli space where the sheaf is not invertible on its support, again
  an odd condition in terms of difference equations.
\end{rems}

\begin{rems}
  It seems likely that the condition $\gcd(x,r)=1$ for the existence of a
  universal sheaf is necessary.  Some condition is needed, see the remark
  following Proposition \ref{prop:ell_surf_rat} below.
\end{rems}

The rational case $x=1$ is particularly nice for another reason: although
the definition of stability generally requires the choice of an ample
bundle, it turns out that when $\chi=1$, this choice is irrelevant.  One
finds in this case $M$ is stable iff any nonzero quotient of $M$ has
positive Euler characteristic (and there are no strictly semistable
sheaves).  We thus find that $\Irr_X(D,1)$ extends naturally to a
projective moduli space.  This space is no longer symplectic, but since
every sheaf in the space is stable, thus simple, it still inherits a
Poisson structure.  This Poisson variety has smooth symplectic leaves
determined by the quasi-isomorphism class of the complex $M\otimes^{\dL}
\sO_{C_\alpha}$, see \cite{poisson}.  In particular, the open subvariety
where $M|_{C_\alpha}=0$ is still smooth and symplectic.

For our purposes, the most natural case is $x=D\cdot f$.  Indeed, the sheaf
corresponding to a difference equation comes from a sheaf on a Hirzebruch
surface with presentation
\[
0\to \rho^*V\otimes \sO_\rho(-1)\to \sO_X^n\to M_0\to 0.
\]
If we twist by $-f$, then both sheaves in the resolution have vanishing
cohomology, and thus $H^*(M_0(-f))=0$; conversely, by Lemma
\ref{lem:cohom_vanish}, any sheaf with $H^*(M_0(-f))=0$ at the least has a
canonical subsheaf with a presentation of the above form.  (In Section $2$,
we imposed the additional open conditions
$\Hom(M,\sO_f(-1))=\Hom(\sO_f(-1),M)=0$ for all $f$; ignoring those
conditions gives us a natural partial compactification.)

Of course, we do not have a sheaf on a Hirzebruch surface, but rather a
sheaf on some blowup of the Hirzebruch surface.  However, we have the
following fact, by the same spectral sequence argument as Lemma
\ref{lem:cohom_vanish}.

\begin{lem}
Let $\pi:X\to X_0$ be a birational morphism of smooth projective surfaces,
and let $M$ be a $1$-dimensional sheaf on $X$.  Then $H^0(M)=H^1(M)=0$
iff $M$ is $\pi_*$-acyclic and $H^0(\pi_*M)=H^1(\pi_*M)=0$.
\end{lem}

In other words, a sheaf $M$ on $X$ induces a difference equation (up to
twisting by $-f$) iff $H^0(M)=H^1(M)=0$.  (Again, it could fail to be the
natural sheaf associated to a difference equation, but this can be avoided
by imposing the additional conditions
$\Hom(M,\sO_g(-1))=\Hom(\sO_g(-1),M)=0$ for any smooth rational curve $g$
contained in a fiber.)  We are thus led to consider the space
$\Irr_X(D,0)$.  Once again, the stability condition turns out to be
independent of the choice of ample bundle: a $1$-dimensional sheaf $M$ on
$X$ with $\chi(M)=0$ is stable iff any proper nontrivial subsheaf has
negative Euler characteristic, and similarly for semistability.  Since we
need semistable sheaves, we do not immediately inherit a Poisson structure,
although this will certainly exist on the complement of the semistable
locus.

It remains only to consider the condition $H^0(M)=H^1(M)=0$.  By Lemma
\ref{lem:tau_func}, this is the complement of a codimension $1$ condition
on any family of $1$-dimensional sheaves, cutting out a Cartier divisor.
Of course, this is only well-defined outside the semistable locus, but the
generic sheaf is integral, and thus stable, so we still obtain a
well-defined divisor on the projective moduli space.  On the integral
locus, this divisor is just the canonical theta divisor in the relative
$\Pic^{g-1}$, while in general, it is the zero locus of a canonical
global section of a canonical line bundle $\det \dR\Gamma(M)^{-1}$.

We can obtain a whole family of such divisors by noting that for any vector
$v\in D^\perp$, we can twist $M$ by $\sO_X(v)$ without affecting the Euler
characteristic, thus obtaining rational automorphisms of the projective
moduli space (these are only rational, since twisting can affect stability;
but this is an automorphism on the integral locus).  In particular, we
obtain in this way a canonical global section of $\det
\dR\Gamma(M(v))^{-1}$, which we call a ``tau function'' by analogy with
\cite{ArinkinD/BorodinA:2009}.  It is of course a misnomer to call it a
function (just like a theta function is not an algebraic function), but it
is at the very least a convenient way of describing divisors on the moduli
space.  (Similarly, the divisor on the moduli stack of surfaces where a
given divisor class is a $-2$-curve can also be viewed as a tau function,
as can theta functions themselves.)

The moduli space is $0$-dimensional when $D$ is a $-2$-curve, and in that
case, $M$ is uniquely determined by its Chern classes and the constraint
that it be disjoint from $C_\alpha$; moreover, its image on $X_0$ is
similarly determined by its bidegree and its intersection with $C_\alpha$
(\cite[Prop.~8.9]{poisson}).  Since the moduli spaces are symplectic, the
next interesting case is the $2$-dimensional case $D^2=D\cdot K_X=0$.  The
only such divisors in the fundamental chamber are the classes $-rK_{X_8}$
for integer $r$, such that $\omega_{X_8}|_{C_\alpha}$ has exact order $r$ in
$\Pic(C_\alpha)$.  Now, in that case, $X_8$ is itself an quasi-elliptic
surface, and $D$ is the class of a fiber.  The corresponding moduli spaces
are just the relative Picard varieties of this quasi-elliptic surface.

\begin{prop}\label{prop:ell_surf_rat}
  Let $\psi:X\to \P^1$ be a relatively minimal rational quasi-elliptic
  surface.  Then for any integer $x$, the relative $\Pic^x$ of $X$ over
  $\P^1$ is a rational surface.
\end{prop}

\begin{proof}
  This is essentially a result of
  \cite[Prop.~5.6.1]{CossecFR/DolgachevIV:1989}.  To be precise, that
  Proposition shows that a relatively minimal quasi-elliptic surface is
  rational iff it has at most one multiple fiber (and that ``tame'') and
  its relative Jacobian is rational.  Since the relative Jacobian of the
  relative $\Pic^x$ is isomorphic to the original relative Jacobian, and
  the relative $\Pic^x$ cannot turn non-multiple fibers into multiple
  fibers (or tame multiple fibers into wild multiple fibers), the claim
  follows.
\end{proof}

\begin{rem}
  It is worth noting that although the moduli space here is rational, there
  is in general no universal sheaf, even if we restrict to an open subset
  (or the function field) of the moduli space.  In particular, when $x=0$,
  it follows from Bhatt's appendix to \cite{KrashenD/LieblichM:2008} that
  the obstruction can be identified with the class of $X$ as a torsor over
  the relative Jacobian.  In particular, we find in that case that the
  obstruction is nontrivial whenever $r>1$.
\end{rem}

Since our surfaces are anticanonical, it makes sense to ask which rational
surface one obtains in this way.  That is, if we start with a blowdown
structure on $X$ and a section of the anticanonical linear system, is there
a natural way to choose a blowdown structure on the (minimal proper regular
model of the) relative $\Pic^x$ such that we can compute the new
anticanonical curve and the new morphism $\Z^{10}\to \Pic(C_\alpha')$?  Note
that on a (quasi-)elliptic rational surface, the line bundle
$\omega_X|_{C_\alpha}$ has finite order (say $r$), and thus determines a
subgroup of degree $\gcd(x,r)$.  Since the corresponding bundle for the
relative $\Pic^x$ has order $r/\gcd(x,r)$, the resulting surface should
depend only on the composition $\Z^{10}\to\Pic(C_\alpha)$ with the quotient
by this subgroup.  We will show this in the case $\gcd(x,r)=r$, and give an
explicit description of the surface, in the following section.

Past the $2$-dimensional cases, it no longer makes much sense to ask which
variety we obtain (since birational geometry is extremely complicated, even
for $4$-folds).  It is fairly straightforward to write down divisor classes
giving such moduli spaces, however; for instance the $4$-dimensional moduli
spaces correspond (in the fundamental chamber of an even blowdown
structure) to one of
\[
2s+3f-\sum_{1\le i\le 10}e_i.\notag
\qquad\text{or}\qquad
4s+4f-2\sum_{1\le i\le 7} e_i-e_8-e_9
\]
The former is always generically integral (assuming of course that the
corresponding line bundle is trivial on $C_\alpha$), while the latter is
generically integral unless $e_8-e_9$ is a $-2$-curve.  Both of these cases
extend to a sequence of moduli problems of dimension $2g$ for arbitrary
$g>1$.  The first case extends to the general problem of classifying
second-order problems with simple singularities (related to generalized
Garnier-type systems); for the second case, see Section \ref{sec:CMspace}
below.

In general, for any dimension bigger than $2$, there are only
finitely many possibilities for the representative of $D$ in the
fundamental chamber.  Indeed, we may write any divisor $D$ in the
fundamental chamber in the form
\[
D = D_7 +
c(2s+2f-e_1-\cdots-e_8)
-\sum_i \lambda_i e_{8+i},
\]
where $D_7\in \langle s,f,e_1,\dots,e_7\rangle$
is in the fundamental chamber and $\lambda$ is a partition with all parts
at most $c$.  The constraint $D\cdot C_\alpha=0$ becomes
\[
|\lambda| = D_7\cdot C_\alpha
\]
and we have
\[
D^2 = D_7^2 + 2 c D_7\cdot C_\alpha - \sum_{1\le i}\lambda_i^2
    \ge D_7^2 + c D_7\cdot C_\alpha.
\]
Since every generator of the fundamental chamber for $m=7$ has positive
intersection with $C_\alpha$, there are only finitely many pairs $(D_7,c)$
such that $D_7\ne 0$ and $D_7^2 + c D_7\cdot C_\alpha\le d-2$, where $d$ is
our desired value for $D^2+2$.  (The pairs with $D_7=0$ correspond to
the $D^2=0$ case.)

\section{Moduli of sheaves on rational (quasi-)elliptic surfaces}

Let $(X,\Gamma,C_\alpha)$ be an anticanonical rational surface with
blowdown structure $\Gamma$, and suppose that $X$ is (quasi-)elliptic, so
that the divisor $rC_\alpha$ is the class of a fiber of a genus 1 pencil
for some positive integer $r$.  Since $-rK$ is the class of a pencil, we
see that the line bundle $\omega_X^r|_{C_\alpha}$ is trivial.  Moreover, if
$\omega_X|_{C_\alpha}$ had order $s$ strictly dividing $r$, then
$sC_\alpha$ would already have multiple sections, contradicting the
assumption on $rC_\alpha$.  We thus find that in this scenario,
$\omega_X|_{C_\alpha}$ has exact order $r$.  Per Proposition
\ref{prop:Pic0_acts}, such a torsion bundle induces an automorphism of
$C_\alpha$ of order $r$, and we can quotient by this automorphism to obtain
a new curve $C'_\alpha$.  The pair $(X,\Gamma)$ is determined from the map
$\Lambda_{10}\to \Pic(C_\alpha)$; if we compose with the degree-preserving
map $\Pic(C_\alpha)\to \Pic(C'_\alpha)$, we obtain a new map
$\Lambda_{10}\to \Pic(C'_\alpha)$.  This new map has the same combinatorial
structure as the original map (twisting by $\omega_X|_{C_\alpha}$ preserves
degrees, so the automorphism preserves components), and thus itself arises
from a unique triple $(X',\Gamma',C'_\alpha)$.

\begin{thm}\label{thm:Picr}
  Suppose $C_\alpha$ is reduced.  Then the surface $X'$ constructed in this
  way is the minimal proper regular model of the relative $\Pic^r$ of $X$,
  in such a way that the fiber corresponding to $rC_\alpha$ is $C'_\alpha$.
\end{thm}

\begin{proof}
  Assume for the moment that the sublattice of $\Lambda_{E_8}$
  corresponding to $C_\alpha$ is saturated; this excludes only two cases,
  namely $\Lambda_{A_7}\subset \Lambda_{E_7}\subset \Lambda_{E_8}$ and
  $\Lambda_{A_8}\subset \Lambda_{E_8}$, which we will discuss below.  (The
  remaining unsaturated sublattice $\Lambda_{D_8}$ corresponds to a
  nonreduced $C_\alpha$.)  Together with the hypothesis that $C_\alpha$ is
  reduced, this is equivalent to assuming that $\Pic(X)$ is generated by
  the $-1$-classes meeting each component of $C_\alpha$ positively.  In
  particular, this ensures that the homomorphism $\Pic(X)\to \Pic(C_\alpha)$
  is determined by its restriction to such classes.  A $-1$-class is always
  uniquely effective, the transversality condition implies that the
  corresponding curve meets $C_\alpha$ in a single smooth point, and that
  point in turn determines the image in $\Pic^1(C_\alpha)$.  In particular,
  we can reconstruct $X$ from $C_\alpha$ and the configuration of points in
  which $-1$-classes meet $C_\alpha$, and similarly for $X'$.  As a result,
  to prove the theorem, we will simply need to show that the minimal proper
  regular model of the relative $\Pic^r$ has the correct special fiber, and
  has the relevant $-1$-classes, meeting $C'_\alpha$ in the correct points.

Rather than study the minimal proper regular model directly, we instead
consider the corresponding moduli space of semistable sheaves with first
Chern class $c_1(M)=-rK$ and Euler characteristic $\chi(M)=r$.  Unlike the
cases $\chi\in \{-1,0,1\}$ discussed earlier, in this case the stability
condition depends nontrivially on the choice of ample divisor $\sO_X(1)$.
There are only finitely many divisors with $D$, $-rK-D$ effective (i.e.,
subdivisors of fibers); we may thus choose the ample divisor in such a way
that for any such divisor,  either $D\in\Z K$ or
\[
\frac{D\cdot \sO_X(1)}{K\cdot \sO_X(1)}
\notin
\Z.
\]
This ensures that any semistable sheaf not supported on the special fiber
will be stable; the various inequalities are forced to be strict by
integrality.  This will also force any semistable sheaf supported on the
special fiber to be $S$-equivalent to a sum of stable sheaves supported on
$C_\alpha$, see below. (We should also note that for any fixed $r\ge 1$,
that there are only finitely many divisor classes with both $D$ and $-rK-D$
effective on {\em some} triple $(X,\Gamma,C_\alpha)$, and could thus in
principle choose the ample divisor in a uniform way over any family of
elliptic surfaces.)

Since stable sheaves are simple (and stability is an open condition), we
find that the moduli space has an open subset which agrees with an open
subset of the moduli space of simple sheaves.  Any sheaf in the open subset
has support disjoint from $C_\alpha$, and thus that open subset has a natural
symplectic structure.  In particular, we find that this open subset is
smooth.  It is also minimal, in that it cannot contain any $-1$-curve of a
smooth compactification (e.g., the desired minimal proper regular model).
(More precisely, one can define intersections of line bundles with
projective curves in quasiprojective surfaces; that any quasiprojective
symplectic surface is minimal follows by noting that $0=K\cdot e=-1$ for
any curve $e$ which can be blown down.)  And, of course, it has a natural
fibration over an affine line (induced by that of the complement of
$C_\alpha$ in $X$) such that the smooth locus of the generic fiber is
$\Pic^r$ of the corresponding fiber of $X$.  In other words, this open
subset of the moduli space is precisely the minimal proper regular model of
the relative $\Pic^r$ of $X\setminus C_\alpha$.  It is thus natural to
conjecture that the Zariski closure of this open subset is the full minimal
proper regular model.  We call this Zariski closure the {\em main
  component} of the semistable moduli space (in fact, it is typically the
smallest, but most interesting, component).

It will thus be necessary to understand the remainder of the moduli space.
A key observation is that if $M$ is semistable and supported on $rC_\alpha$,
then $M$ is $S$-equivalent to a sheaf scheme-theoretically supported on
$C_\alpha$.  Indeed, $\omega_X|_{C_\alpha}$ is in the identity component of
$\Pic(C_\alpha)$, and thus twisting by $\omega_X$ preserves semistability.
If for some $l>0$, $M$ is supported on $(l+1)C_\alpha$ but not on $lC_\alpha$,
then we have a nonzero morphism $M\otimes \omega_X\to M$ between semistable
sheaves of the same slope, and thus $M$ is $S$-equivalent to the sum of the
image and the cokernel of this morphism.  This makes $M$ $S$-equivalent to
a sheaf supported on $lC_\alpha$, and we may proceed by induction in $l$.

\begin{lem}\label{lem:antican_stable}
Let $C_\alpha$ be an anticanonical curve on a rational surface $X$.  If $M$
is a semistable sheaf supported on $C_\alpha$ with $c_1(M)=rC_\alpha$,
$\chi(M)=r$, then $M$ is $S$-equivalent to a sum of stable sheaves with
$c_1(M)=C_\alpha$, $\chi(M)=1$.
\end{lem}

\begin{proof}
  We first claim that the slope 0 sheaf ${\cal O}_{C_\alpha}$ is stable.
  We need to show that any quotient sheaf has positive Euler
  characteristic, and can easily reduce to the case of a torsion-free
  quotient, i.e., ${\cal O}_D$ for some curve $D\subset C_\alpha$.  But by
  Lemma \ref{lem:Ca_is_num_conn}, we have $\chi({\cal O}_D)=h^0({\cal
    O}_D)>0$ as required.  It follows immediately that the ideal sheaf of a
  point is stable, since all of its subsheaves are subsheaves of ${\cal
    O}_{C_\alpha}$, so have negative Euler characteristic.  Then, by
  duality on $X$, we obtain stability of the sheaf
\[
{\cal O}_{C_\alpha}(p)
:=
\sExt^1_X({\cal I}_p,\omega_X),
\]
the unique nontrivial extension of ${\cal O}_p$ by ${\cal O}_{C_\alpha}$.

It will thus suffice to show that $M$ is $S$-equivalent to a sum of sheaves
of the form ${\cal O}_{C_\alpha}(p)$.  In fact, it will suffice to construct
a nonzero homomorphism from $M$ to some ${\cal O}_{C_\alpha}(p)$: since
${\cal O}_{C_\alpha}(p)$ is stable of the same slope as $M$ (true regardless
of the choice of ample divisor!), such a morphism is necessarily
surjective.  The kernel of the surjection will then remain semistable of
the same slope, and we may proceed by induction.  The same argument shows
that $\Hom(M,{\cal O}_{C_\alpha})=0$ (since ${\cal O}_{C_\alpha}$ is stable of
smaller slope than $M$), and thus by duality (note that $C_\alpha$ is
Gorenstein with trivial dualizing sheaf), $H^1(M)=0$.

For any point $p\in C_\alpha$, consider the short exact sequence
\[
0\to M'_p\to M\to M\otimes {\cal O}_p\to 0.
\]
If the map $H^0(M)\to H^0(M\otimes {\cal O}_p)$ fails to be surjective,
or, equivalently,
\[
\Ext^1(M\otimes {\cal O}_p,{\cal O}_{C_\alpha})
\to
\Ext^1(M,{\cal O}_{C_\alpha})
\]
fails to be injective, then $\Hom(M'_p,{\cal O}_{C_\alpha})\ne 0$, and any
such morphism induces a nontrivial extension of $M\otimes {\cal O}_p$ by
${\cal O}_{C_\alpha}$ together with a nonzero morphism from $M$ to this
extension.  (Indeed, considerations of Hilbert polynomials show that the
image of this morphism has first Chern class $C_\alpha$.)  Since
$\dim\Ext^1({\cal O}_p,{\cal O}_{C_\alpha})=1$, this extension has the form
${\cal O}_p^n\oplus {\cal O}_{C_\alpha}(p)$, and thus $M$ has a nonzero
morphism to ${\cal O}_{C_\alpha}(p)$.

If the map $H^0(M)\to H^0(M\otimes {\cal O}_p)$ is always surjective, then
the map $H^0(M)\otimes {\cal O}_{C_\alpha}\to M$ is surjective on fibers, and
thus surjective.  But since both sheaves have the same first Chern class and
$M$ has the larger Euler characteristic, this is impossible!
\end{proof}

\begin{rems}
  Note that one can reconstruct $p$ from the sheaf ${\cal O}_{C_\alpha}(p)$,
  since the latter has a unique global section.  It follows that there are
  at most $r$ distinct points $p_i$ admitting morphisms $M\to {\cal
    O}_{C_\alpha}(p_i)$.  Since the cokernel of the natural morphism
  $H^0(M)\otimes {\cal O}_{C_\alpha}\to M$ is supported (set-theoretically)
  on those points, we conclude that the natural morphism is injective, and
  the cokernel is a $0$-dimensional sheaf of degree $r$, from which we can
  read off the $S$-equivalence class of $M$.
\end{rems}

\begin{rems}
  If we replace semistability by the weaker condition that any nonzero
  quotient of $M$ has positive Euler characteristic, we may still conclude
  by the same argument that $M$ has a surjective morphism to some sheaf of
  the form $\sO_{C_\alpha}(p)$.  If we further replace $\chi(M)=r$ by
  $\chi(M)<r$, then $H^0(M)\to H^0(M\otimes \sO_p)$ can never be
  surjective, so that $M$ has a surjective morphism to {\em every} sheaf of
  the form $\sO_{C_\alpha}(p)$.
\end{rems}

\smallskip

We thus conclude that the portion of the moduli space classifying sheaves
supported on the special fiber consists (up to $S$-equivalence) of sums of
sheaves ${\cal O}_{C_\alpha}(p)$, and need to know which of these sheaves
lie on the main component.  The key additional constraint comes from the
observation that if $M$ is supported on $X\setminus C_\alpha$, then
$M\otimes \omega_X\cong M$.  Twisting by $\omega_X$ induces an automorphism
of the full semistable moduli space, and it follows that this automorphism
must act trivially on the main component.  In other words, if $M$ is
$S$-equivalent to
\[
\bigoplus_{1\le i\le r} {\cal O}_{C_\alpha}(p_i),
\]
the multiset of points $p_i$ must be permuted by the action of $\omega_X$.
Since this action is free of order $r$ on the smooth locus, we find that
the $S$-equivalence classes fixed by the automorphism consist of sheaves
\[
\bigoplus_{1\le k\le r} {\cal O}_{C_\alpha}(p)\otimes \omega_X^k
\]
with $p$ in the smooth locus, together with sums
\[
\bigoplus_{1\le i\le r} {\cal O}_{C_\alpha}(p_i)
\]
in which each $p_i$ is a singular point of $C_\alpha$.  In our case, since
$C_\alpha$ is reduced, the latter gives only finitely many points.  The first
family of sheaves is manifestly classified by the smooth locus $C'_\alpha$,
with the closure of $C'_\alpha$ containing in addition only the sheaves
${\cal O}_{C_\alpha}(p)^r$ with $p$ singular.  (There could in principle be
isolated additional points in the main component, but we will see below
that this cannot happen.)

The main difficulty at this point is that it is very difficult to determine
tangent spaces to GIT quotients at semistable points (especially so in our
case, since we only want the tangent vectors coming from a particular
component).  To get around this, we will consider one more moduli space.

The condition that a simple sheaf is invertible on its support is open (we
can express it as the condition that the first Fitting scheme is empty), as
is the condition that it be semistable.  If $C$ is any fiber of the genus 1
fibration on $X$, then $\sO_C$ has a unique global section, and thus any
invertible sheaf on $C$ is simple.  We thus obtain an algebraic space
parametrizing semistable invertible sheaves on fibers of $X$.  As before,
any semistable invertible sheaf not supported on the special fiber
$rC_\alpha$ is stable, and thus away from the special fiber, we recover the
N\'eron model of the relative $\Pic^r$.  This fails on the special fiber
for the simple reason that a given $S$-equivalence class can occur more
than once.  By the above classification of $S$-equivalence classes, we find
that the $S$-equivalence class of $M$ is determined by the orbit under
twisting by $\omega_X$ of the invertible sheaf $M|_{C_\alpha}$; in
particular, $M$ is stable iff $M|_{C_\alpha}\cong {\cal O}_{C_\alpha}(p)$ for
some point $p$ of the smooth locus.  Thus each $S$-equivalence class is
represented by $r$ distinct points of the algebraic space parameterizing
semistable invertible sheaves.

The point is that we can compute tangent spaces in this algebraic space:
\begin{align}
\dim\Ext^1_X(M,M)
&=
\dim\Hom_X(M,M)
+
\dim\Hom_X(M,M\otimes \omega_X)\notag\\
&=
\dim\Hom_{rC_\alpha}(M,M)
+
\dim\Hom_{rC_\alpha}(M,M\otimes \omega_X)\notag\\
&=
\dim\Gamma(\sO_{rC_\alpha})
+
\dim\Gamma(\omega_X|_{rC_\alpha})\notag\\
&=
2,
\end{align}
and thus the ($2$-dimensional) algebraic space is smooth at these points.
Since twisting by $\omega_X$ acts without fixed points on the special
fiber, it preserves tangent spaces, and thus the corresponding subset of
the semistable moduli space is smooth.  In particular, we conclude that the
main component of the semistable moduli space is smooth on the locus
represented by invertible sheaves, i.e., on the smooth locus of $C'_\alpha$.

The minimal desingularization of the main component is thus a proper
regular model of the relative $\Pic^r$, so blows down to the minimal proper
regular model.  It follows that if we simply remove the singular points
from the main component, the result maps to the minimal proper regular
model.  Now, the special fiber of the main component has the same number of
components as the special fiber of the minimal proper regular model (which
must have the same Kodaira type as $C_\alpha$,
\cite[Thm.~5.3.1]{CossecFR/DolgachevIV:1989}).  Thus the only way the surfaces
can fail to be isomorphic is if the map from the minimal desingularization
of the main component blows down a component of the original special fiber.
(Indeed, we must blow down as many components as the minimal
desingularization introduces.)  Now, any section of $\Pic^r(X\setminus
C_\alpha)$ extends to a $-1$-curve on the minimal proper regular model, which
must in particular meet the special fiber in a point of the smooth locus.
It follows that if the corresponding curve in the main component meets the
special fiber in a point of the smooth locus, the corresponding component
cannot be contracted.  We will see that (under the additional saturation
hypothesis) any component is met by some $-1$-curve, giving the desired
isomorphism.

Now, let $e$ be any $-1$-class on $X$ which is transverse to $C_\alpha$, and
consider the corresponding $\tau$-divisor $\tau(e)$.  This certainly
determines a well-defined curve in the complement of the special fiber (and
any non-integral fibers), and we claim that its closure in the main
component meets the special fiber in a single point, which lies in the
smooth locus.  Indeed, if $M$ is supported on the special fiber and
$\Gamma(M(-e))\ne 0$, then we find that $M$ is $S$-equivalent to
$\sO_{rC_\alpha}(e)$.  Indeed, we may state the condition as
$\Hom(\sO_{rC_\alpha}(e),M) = \Hom(\sO_X(e),M) \ne 0$.  Since the image is
both a quotient of the semistable sheaf $\sO_{rC_\alpha}(e)$ and a subsheaf
of the semistable sheaf $M$, the image is also semistable, and can be
extended to Jordan-H\"older filtrations of both $M$ and $\sO_{rC_\alpha}$.
Since $M$ and $M\otimes \omega_X$ are $S$-equivalent, this is enough to
completely determine the $S$-equivalence class of $M$ as required.

In particular, we find that $\tau(e)$ meets the special fiber in the image
of $e\cap C_\alpha$ in $C'_\alpha$. (In particular, we may choose $e$ so that
this point lies in any desired component of $C'_\alpha$.)  It remains only to
show that $\tau(e)$, or rather the corresponding invertible sheaf $\det
R\Gamma(M(-e))$, is a $-1$-class on the minimal proper regular model, and
that this correspondence between $-1$-classes extends to a homomorphism
preserving the intersection pairing.

If $\sO_C(e)$ is stable for every fiber $C$, then $\tau(e)$ consists
precisely of sheaves of that form, and is thus a rational curve as
required.  Since it meets the generic fiber (and thus the anticanonical
curve) in a single point, we conclude that it is a $-1$-curve.  More
generally, adding a component of a nonspecial fiber to $e$ does not change
how $\tau(e)$ meets the special fiber or any integral fiber, and in this
way we can arrange for $\sO_C(e)$ to be stable for all $C$.  In particular,
we find that any class $\det R\Gamma(M(-e))$ obtained in this way is the
sum of the class of a $-1$-curve and a linear combination of components of
nonspecial fibers.

It remains to see that this correspondence extends to a homomorphism and
preserves the intersection pairing.  Both of these are closed conditions on
the (irreducible) moduli stack, so we may impose any dense conditions we
desire.  In particular, we may assume that $C_\alpha$ is smooth and every
nonspecial fiber of $X$ is integral, so that $\tau(e)$ is a $-1$-curve for
every $-1$-class $e$.  Let $e'$ be another $-1$-curve on $X$.  Then
$\tau(e')\cdot\tau(e)$ may be computed as the degree of $\det
R\Gamma(M(-e'))|_{\tau(e)}$, or equivalently as the degree of $\det
R\Gamma(\sO_C(e-e'))$ as $C$ varies over fibers of the genus 1 fibration on
$X$.  Now, consider the natural presentation
\[
0\to \sO_X(e-e'+rK)\to \sO_X(e-e')\to \sO_C(e-e')\to 0
\]
Since $(e-e')\cdot K=0$ and $X$ has no $-2$-curves, $e-e'$ is
ineffective, and similarly for $e-e'+rK$, $e'-e+K$, $e'-e+(1-r)K$.
Thus $R\Gamma(\sO_C(e-e'))$ is represented by the complex
\[
H^1(\sO_X(e-e'+rK))\to H^1(\sO_X(e-e')),
\]
since the other cohomology groups vanish.  This map depends linearly on the
original map $\sO_X(rK)\to \sO_X$, and thus the desired degree may be
computed as the common dimension of the two cohomology groups, which by
Hirzebruch-Riemann-Roch is equal to $-1-(e-e')^2/2=e\cdot e'$ as required.

Finally, to see that this extends to a homomorphism, we note that the
intersection form on the ten $-1$-curves $s-e_1$, $f-e_1$,
$e_1$,\dots,$e_8$ has determinant $-1$, so the corresponding $\tau$
divisors span $\Pic(X')$, just as the original divisors span $\Pic(X)$.
Since we may use the intersection form to expand any element of $\Pic(X)$
in that basis, we conclude that the $\tau$-divisor map is linear.

We excluded two cases above, in which $C_\alpha$ has Kodaira symbol $I_8$ or
$I_9$.  In the latter case, we can easily see that any Jacobian fibration
with an $I_9$ fiber has Weierstrass form $y^2+txy+a_3y=x^3$ over the
algebraic closure, with $a_3\ne 0$.  Any two such surfaces are
geometrically isomorphic (in a nonunique way), and thus the claim follows
immediately.  Similarly, in the bad $I_8$ case, the corresponding Jacobian
fibration must be the desingularization of the blow-up in the identity of a
surface
\[
y^2+txy = x^3+a_2 x^2+a_4 x
\]
with $a_4\ne 0$; two such surfaces are geometrically isomorphic iff they
have the same value of $a_2^2/a_4$.  There are two $-1$-curves on this
surface that do not meet the singular point, which meet the corresponding
$G_m$ in the points $\lambda$, $1/\lambda$ where
\[
\frac{(\lambda+1)^2}{\lambda} = \frac{a_2^2}{a_4}.
\]
In particular, the surface is determined up to (again nonunique)
isomorphism by the two points of intersection, and again the claim follows.
\end{proof}

\begin{rem}
Note that in the good cases, we prove a slightly stronger fact: not only is
the special fiber of $\Pic^r(X)$ isomorphic to $C'_\alpha$, but the the
isomorphism we construct is compatible with the isomorphism $X'\cong
\Pic^r(X)$.  This presumably still holds in the $I_8$ and $I_9$ cases, but
the above calculation does not suffice.
\end{rem}

\begin{rem}
  We should note that by the remark following Proposition
  \ref{prop:ell_surf_rat}, there is no universal family over the stable
  locus of this moduli space.
\end{rem}

The case that $C_\alpha$ is nonreduced appears to be more subtle.  Since
$\Pic^0(C_\alpha)\cong G_a$, this case (for $r>1$) only arises in
characteristic $p$, with $r=p$.  We still have an action of
$\Pic^0(C_\alpha)$ on $C_\alpha$, which restricts to an action of the
\'etale subscheme generated by $K|_{C_\alpha}$.  Since this action fixes
the singular points of $C_\alpha$, it must act trivially on any component
appearing with multiplicity.  In the simplest case, $I_0^*$, this would
indicate that $C'_\alpha\cong C_\alpha$.  Indeed, a curve of type $I_0^*$
is determined by the cross-ratio $\lambda$ of the four points where the
double component meets the reduced components, or, more precisely, by the
invariant function
\[
\frac{j}{256} = \frac{(\lambda^2-\lambda+1)^3}{\lambda^2(\lambda-1)^2}.
\]
Since the group acts trivially on the special fiber, the quotient should
preserve the cross-ratio, but experiments (for $p=2$, $p=3$) instead
suggest that the true special fiber of the moduli space satisfies
$\lambda(C'_\alpha)=\lambda(C_\alpha)^p$.

\medskip

The one disadvantage of considering $\Pic^r$ is that the corresponding
morphisms $B$ are never maps of trivial bundles; as a result, we cannot
directly apply statements about $\Pic^r$ in the noncommutative setting.
With this in mind, we consider the corresponding moduli space of semistable
sheaves of Euler characteristic 0.  The individual fibers are not too
difficult to understand, but to understand the full moduli space, we will
need to work \'etale locally, and thus must to some extent consider more
general genus 1 fibrations (e.g., with total space which is not rational).
For this, we note first that the case $r=1$ of the above result is easy to
extend to arbitrary elliptic surfaces.

\begin{thm}
  Let $\psi:X\to C$ be a smooth, relatively minimal, genus 1 fibration with
  no multiple fibers, and let ${\cal M}_X(1)$ be the moduli space
  classifying stable sheaves $M$ of Euler characteristic 1 and with
  $c_1(M)$ a fiber of $\psi$. Then there is a natural isomorphism ${\cal
    M}_X(1)\cong X$.
\end{thm}

\begin{proof}
  We first note that Lemma \ref{lem:antican_stable}, though stated for
  anticanonical curves on rational surfaces, applies equally well (with
  essentially the same argument) to curves of canonical type (i.e., curves
  with the same intersection matrix as a fiber of a minimal proper regular
  model of a genus 1 fibration) on surfaces, so in particular to fibers of
  $\psi$.  In particular, the construction there of sheaves $\sO_C(p)$, and
  the proof that those sheaves are stable, carries over directly.  In
  particular, any stable sheaf in ${\cal M}_X(1)$ has this form, and we can
  easily construct a corresponding universal sheaf on $X\times X$.
\end{proof}

For Jacobian fibrations, the Euler characteristic 0 case is also relatively
straightforward to deal with.

\begin{cor}
  Let $\psi:X\to C$ be a smooth, relatively minimal, genus 1 fibration with
  a section $s:C\to X$, and let ${\cal M}_X(0)$ be the moduli space
  classifying stable sheaves $M$ of Euler characteristic 0 and with
  $c_1(M)$ a fiber of $\psi$.  Then there is a natural birational
  morphism $X\to {\cal M}_X(0)$ for which the exceptional locus is the
  union of all vertical curves not meeting the section.
\end{cor}

\begin{proof}
  To construct a morphism from $X\cong {\cal M}_X(1)$ to ${\cal M}_X(0)$,
  it will suffice to give a construction taking stable sheaves of Euler
  characteristic 1 to semistable sheaves of Euler characteristic 0.  In
  particular, let $M$ be the sheaf corresponding to a point of ${\cal
    M}_X(1)$, and consider the twist $M(-s)$ of Euler characteristic 0.
  Any subsheaf of this twist has the form $M'(-s)$ for some $M'\subset M$;
  since $s$ meets fibers transversely, we have
\[
\chi(M'(-s))\le \chi(M')\le 0,
\]
where the second inequality follows from stability of $M$.  This, of
course, is precisely the inequality we needed to show to demonstrate
semistability of $M$.

Now, if $M(-s)$ is stable, then it is determined by its
corresponding point in ${\cal M}_X(0)$, so that we can recover $M$ from its
image; thus if $M(-s)$ is stable, then it falls outside the
exceptional locus.

Write $M=\sO_f(p)$ with $f$ a fiber of $\psi$ and $p\in F$ a closed point.
If $f$ is integral, then $M(-s)$ is automatically stable (as a
torsion-free sheaf with integral support).  In addition, if $p$ is a smooth
point of the identity component (the component meeting $s$) of $f$, then
$M(-s)$ is an invertible sheaf of degree $0$ on every component.
It follows that twisting by $M(-s)$ preserves stability; since
$\sO_f$ is stable, so is $M(-s)$ in this case.

If $p$ is not a point of the identity component, then $M(-s)$ is
invertible near the identity component, and has degree $-1$ on that
component.  We thus have a morphism
\[
M(-s)\to \sO_{f_0}(-1),
\]
where $f_0$ is the identity component.  This makes $M$ strictly semistable.
By semicontinuity, such a morphism continues to exist on the closure of the
complement of the identity component.

In other words, $M(-s)$ is strictly semistable iff $p$ lies
on some nonidentity component of $f$.  It remains to show that these
components are contracted, and that the morphism is dominant.  By
dimensionality, it will be enough to show that any invertible sheaf in
${\cal M}_X(0)$ has degree $0$ on every component (so remains stable upon
twisting by the stable sheaf $\sO_f(s)$), and that each
fiber of ${\cal M}_X(0)$ has at most one strictly semistable sheaf.  This
follows from the next Lemma.
\end{proof}

\begin{rem}
In particular, ${\cal M}_X(0)$ is the Weierstrass model of $X$.
\end{rem}

\begin{lem}\label{lem:canon_stable}
  Let $C$ be a curve of canonical type on a surface $X$, and let $M$ be a
  stable sheaf of Euler characteristic $0$ set-theoretically supported on
  $C$.  Then either $M$ is an invertible sheaf on $C$, of degree $0$ on
  every component of $C$, or $M$ is the direct image of $\sO_{\P^1}(-1)$
  under some morphism $\P^1\to C$.
\end{lem}

\begin{proof}
Note that since $C$ is orthogonal to every component of $C$, $M(-C)$ is
also stable, and thus the natural map $M(-C)\to M$ must be 0; it follows
that $M$ is scheme-theoretically supported on $C$.

If $c_1(M)$ is not a multiple of $C$, then signature considerations show
$c_1(M)^2<0$ and thus, since $c_1(M)$ is a sum of components of $C$, that
there exists a component $C_1$ of $C$ such that $C_1\cdot c_1(M)<0$.  In
that case, we find
\[
\dim\Hom(\sO_{C_1}(-1),M)
-
\dim\Ext^1(\sO_{C_1}(-1),M)
+
\dim\Ext^2(\sO_{C_1}(-1),M)
=
-C_1\cdot c_1(M)
>
0
\]
Since $M$ has canonical type, $C_1$ is a $-2$-curve, so orthogonal to the
canonical class of $X$, and thus
$\dim\Ext^2(\sO_{C_1}(-1),M)=\dim\Hom(M,\sO_{C_1}(-1))$.  We thus conclude
that there is a morphism between $M$ and $\sO_{C_1}(-1)$; since $M$ is
stable, this can only occur if $M\cong \sO_{C_1}(-1)$.

Thus, suppose $c_1(M)=rC$, and consider the sheaf $M\otimes {\cal O}_C(p)$
for some smooth point $p\in C$, of Euler characteristic $r$.  If this has a
nonzero quotient of nonpositive Euler characteristic, then twisting by the
ideal sheaf of $p$ gives a nonzero proper quotient of $M$ of nonpositive Euler
characteristic, contradicting stability.  But then by the proof of Lemma
\ref{lem:antican_stable}, there is a surjective morphism
\[
M\otimes {\cal O}_C(p)\to {\cal O}_C(p')
\]
for some point $p'\in C$, giving a surjection
\[
M\to {\cal O}_C(p')\otimes {\cal I}_p.
\]
Since $M$ is stable and both sheaves have Euler characteristic 0, this must
in fact be an isomorphism, and we have already seen that such sheaves are
strictly semistable unless $p'$ does not lie on any nonidentity component
of $C$.  Thus either $M$ is invertible on $C$, of degree $0$ on every
component, or $C$ is integral and $p'$ is singular on $C$; in the latter
case, the normalization of $C$ is $\P^1$, and the sheaf must be the direct
image of $\sO_{\P^1}(-1)$ on the normalization.
\end{proof}

\begin{cor}\label{cor:weier1}
  Let $\psi:X\to C$ be a smooth, relatively minimal, genus 1 fibration with
  no multiple fibers, and let ${\cal M}_X(0)$ be the moduli space
  classifying stable sheaves $M$ of Euler characteristic 0 and with
  $c_1(M)$ a fiber of $\psi$.  Then ${\cal M}_X(0)$ is naturally isomorphic
  to the Weierstrass model of the relative Jacobian of $X$.
\end{cor}

\begin{proof}
  If $X$ has a section, we are done; in general, such a section exists
  \'etale locally, allowing us to identify ${\cal M}_X(0)$ with the
  Weierstrass model \'etale locally on $C$.  Since these local
  identifications are natural, they descend to give a natural isomorphism
  as required.
\end{proof}

\begin{thm}
Under the hypotheses of Theorem \ref{thm:Picr}, let ${\cal M}_X(0)$ denote
the moduli space of semistable 1-dimensional sheaves with $c_1(M)=-rK_X$,
$\chi(M)=0$.  Then ${\cal M}_X(0)$ is isomorphic to the
Weierstrass model of the relative $\Pic^r$ as computed above.
\end{thm}

\begin{proof}
  Choose a $-1$-curve $e$ on $X$, and consider the map $M\to M(-e)$ from
  the moduli space of semistable sheaves with $\chi(M)=r$ to ${\cal
    M}_X(0)$.  Note that this is in general only a rational map, since for
  sheaves on reducible fibers, the image could easily be unstable.
  However, this map is certainly well-defined for sheaves on smooth fibers,
  and more importantly, is well-defined on the special fiber.  Since
  extensions of semistable sheaves of Euler characteristic 0 are
  semistable, it is sufficient to prove that $M'(-e)$ is semistable for
  $M'$ stable of the same slope as $M$.  But this implies $\chi(M')=1$,
  which we have already seen suffices for stability.

  We thus see that ${\cal M}_X(0)$ is naturally isomorphic to the
  Weierstrass model in a neighborhood of the special fiber.  Since we also
  have such an identification on the complement of the special fiber (as
  the complement of the unique multiple fiber has no multiple fibers), the
  claim follows.
\end{proof}

\begin{rem} In particular, given a choice of $-1$-curve $e$ on $X$ such that
  the corresponding $\tau$ divisor is irreducible, ${\cal M}_X(0)$ is
  obtained from $X'$ by contracting all $-2$-curves disjoint from $e$.
\end{rem}

\section{Elliptic difference equations}\label{sec:elldiff}

We now wish to translate the above theory back to the realm of difference
equations.  From a geometric perspective, the simplest case is that of
symmetric elliptic difference equations, since then not only is the surface
smooth, but so is the anticanonical curve.  If $C$ is a smooth genus 1
curve, then the above considerations tell us that symmetric difference
equations on $C$ twisted by a line bundle are in natural correspondence
with triples $(X_0,\phi,M_0)$ where $X_0$ is a Hirzebruch surface (in
particular with specified map to $\P^1$), $M_0$ is a sheaf on $X_0$ with
$H^0(M_0)=H^1(M_0)=0$, and $\phi:C\to X_0$ embeds $C$ as an anticanonical
curve.  (As we mentioned above, this is not quite correct, as these sheaves
also include degenerate cases where some of the singularities cancel each
other.)  Moreover, the pairs $(X_0,\phi)$ are classified by elements of
$\Pic^2(C)\times \Pic^2(C)$ or $\Pic^1(C)\times \Pic^2(C)$, depending on
the parity of the Hirzebruch surface, or equivalently depending on the
parity of the degree of the twisting line bundle.

Since $C\cong C_\alpha$ is smooth, $M_0|_{C_\alpha}$ is a direct sum of
structure sheaves of jets.  If $X$ is the minimal desingularization of the
blowup of $X_0$ in those jets, then there is a natural way to lift $M_0$ to
a sheaf $M$ on $X$ which is disjoint from the anticanonical curve.  (This
is the minimal lift of \cite{poisson}, see in particular Proposition 6.12
there.)  In this way, we encode the singularities of the equation in the
surface $X$ and the Chern class of $M$.

As above, the surface is determined by the classes $\phi^*(s)$, $\phi^*(f)$
and $\phi^*(e_i)$ for $1\le i\le m$; since $C$ is smooth, the bundles
$\phi^*(e_i)$ can be identified with line bundles $\sO_C(p_i)$ for
points $p_i$.  Now, consider the extension
\[
B:\rho^*V\to \rho^*W(s)
\]
of our original $B$ to $X_0$.  The Chern class of $M_0$ is given by that of
$\det(B)$, so has the form $ns+df$ where $n=\rank(W)$ is the order of the
difference equation.  At any point $p\in C_\alpha$, we can view $B$ as a
matrix over the local ring $\sO_{C_\alpha,p}$.  Up to left- and
right-multiplication by invertible matrices, we can diagonalize $B$, and
then define a partition $\lambda(B;p)$ by letting $\lambda_j(B;p)$ for
$j\ge 1$ be the number of diagonal elements contained in $\mathfrak{m}^j$.
Local computations then tell us that $\lambda_1(B;p)$ is the rank of $M_0$
at $p$, and $\lambda_j(B;p)$ is the rank after blowing up $p$ $j-1$ times.
If $e_{p,j}$ denotes the $j$-th class in the sequence $e_1$,\dots,$e_m$
such that $\phi^*(e_i)\cong \sO_C(p)$, then we find
\[
c_1(M) = c_1(M_0) - \sum_{p,i} \lambda_i(B;p) e_{p,i}.
\]
(Note that we must blow up $p$ at least as many times as there are parts of
$\lambda$ in order to make the resulting sheaf disjoint from the
anticanonical curve.)

We can also describe the invariants $\lambda(B;p)$ in terms of the original
shift matrix $A$.  If $p$ is not fixed by $\eta$, then we can again
diagonalize $A$ over the local ring at $p$ (by left- and right-
multiplication).  The resulting equivalence classes are given by weights of
$\GL_n$, i.e., nonincreasing sequences of integers.  We then find that
$\lambda(B;p)$ is determined by the positive coefficients of this weight;
the negative coefficients of the weight appear in $\lambda(B;\eta(p))$.
When $p$ is fixed by $\eta$, the situation is more complicated; up to the
relevant equivalence relation (left multiplication by invertible matrices
over the local ring, right multiplication by symmetric invertible matrices
over the local field), $B$ is a direct sum of matrices
\[
\begin{pmatrix}1\end{pmatrix},\qquad
\begin{pmatrix}u\end{pmatrix},\qquad
\begin{pmatrix}1 & u\\0 & u^e\end{pmatrix}, e>1.
\]
(This needs to be adjusted slightly when the equation is twisted by a line
bundle.)  The second case is a singularity of order 1, but corresponds to
an eigenvalue $-1$ of $A$ (assuming the characteristic is not 2); the third
cases have order $e$, but appear to have order $e-1$.  (In characteristic
2, this phenomenon is worse: singularities of order $e$ appear to have
order $\max(e-2,0)$ or $\max(e-4,0)$ depending on whether $C$ is ordinary
or supersingular.)

Of course, a sheaf $M$ on $X$ disjoint from $C_\alpha$ need not come from a
maximal morphism $B$.  The condition
$\Hom(M,\sO_g(-1))=\Hom(\sO_g(-1),M)=0$ for every component $g$ of a fiber
is simplified by disjointness, since we need only consider those $g$ which
are disjoint from $C_\alpha$.  In particular, $g$ must be a $-2$-curve, and
since it is contained in a fiber, must be a root of the $D_m$ subsystem.
The subquotient corresponding to the standard representation of a
difference equation will differ from $M$ by a number of copies of sheaves
$\sO_g(-1)$, and thus in particular correspond to strictly semistable
points of the moduli space.  As a result, the specific extension classes
will be irrelevant, and we thus obtain precisely one point of the moduli
space for each difference equation that arises.

The constraint on the difference equations underlying such semistable
points is that their Chern class must differ from the specified Chern class
by a nonnegative sum of $-2$-curves disjoint from $C_\alpha$.  This can be
translated in terms of the local data $\lambda(B;p)$ as follows.
Subtracting roots of the form $e_i-e_j$ simply replaces the partitions
$\lambda(B;p)$ by partitions of the same size which cover it in the
dominance ordering.  Roots of the form $f-e_i-e_j$ either subtract 1 from
the first parts of both $\lambda(B;p)$ and $\lambda(B;\eta(p))$ or (when
$p=\eta(p)$) subtract $1$ from the first two parts of $\lambda(B;p)$.  For
$p\ne \eta(p)$, the conditions combine to say that the relevant weight of
$\GL_n$ (related to the conjugate partitions) becomes smaller in dominance
order, with something similar in the case of ramification points.

One special case we should note is that when $D=s-f$ (assuming this is
effective), then as usual, the moduli space is a point (the sheaf
$\sO_{s-f}(-1)$), and the corresponding difference equation is just the
trivial equation $v(z+q)=v(z)$.  More generally, if $s-f$ is effective and
$D\cdot (s-f)<0$, then the corresponding difference equation will have a
block-triangular structure such that the first or last block is trivial.

\medskip

Of course, a sheaf on an anticanonical surface $X$ with $C_\alpha\cong C$
does not determine a difference equation unless we also choose a blowdown
structure (more precisely, a choice of blowdown structure modulo the action
of ineffective roots of the $A_m$ subsystem).  In other words, a given
moduli space of sheaves corresponds to many different moduli spaces of
difference equations, one for each blowdown structure on the surface.  In
addition, we can also twist by line bundles and apply the duality
$\sExt^1(-,\omega_X)$.  The latter is canonical, but to make sense of the
former requires a choice of blowdown structure.  Thus in the generic
situation, we have an action of $\Aut(C)\times (W(E_{m+1})\times
Z_2)\ltimes \Z^{m+2}$, where the cyclic group $Z_2$ acts by duality; in the
nongeneric situation, there is a partial action taking into account the
usual issues with effective reflections.  Since the group $W(E_{m+1})$
simply acts on the set of ways of interpreting sheaves, it certainly
respects the Poisson structure on the moduli space, and the construction of
the Poisson structure implies that it is preserved by twisting.  Duality is
anti-Poisson (\cite[Prop.~7.12]{poisson}), and $\Aut(C)$ acts on the
Poisson structure in the same way it acts on holomorphic differentials.  In
particular, $\Pic^0(C)\subset \Aut(C)$ preserves the Poisson structure, and
hyperelliptic involutions are anti-Poisson; in the $j=0$ and $j=1728$
cases, we also have automorphisms multiplying the Poisson structure by
other roots of unity.  In terms of moduli spaces of difference equations,
each of these operations will change the parameters (the twisting line
bundle, the points with allowed singularities, $q$), but should give
birational maps between the corresponding moduli spaces.  The action on $q$
is essentially forced: Poisson maps should preserve $q$, while anti-Poisson
maps should negate $q$ (and for $j=0$, $j=1728$, $\Aut(C)$ acts as one
would expect on $q\in \Pic^0(C)$).

Note that the subgroup $D^\perp\subset \Z^{m+2}$ acts as a (large) abelian
group of rational Poisson automorphisms of the moduli spaces $\Irr(D,x)$.
(The element $K_X$ acts trivially, of course, as does any effective class in
$D^\perp$.) In particular, they give rise to a discrete integrable system
acting on a rational variety, which relative to the fibration by $\supp(M)$
acts by translation within each fiber (a torsor over the Jacobian of the
support).  Similarly, we will describe a $q$-twisted version of this
action, which appears to be an analogue of (higher-order) discrete
Painlev\'e equations, a non-autonomous version of translation on an abelian
variety.

Of course, in our setting, we only have a relaxation of the true moduli
spaces of difference equations, but will still gain insight by looking at
how simple reflections, twists by basis elements, and duality act; in each
case, there will be an obvious way to take into account the shifting by
$q$.  Since the simplest operations do not preserve the triviality
condition $H^0(M)=H^1(M)=0$ even generically (since they do not preserve
the condition $\chi(M)=0$), we need to allow $W$ to be nontrivial.  We thus
note (per \cite{ABR} and \cite{noncomm1}) that in the (analytic)
difference equation case, $B$ corresponds to a morphism
\[
B:\pi_{\eta'}^*V\to \pi_{\eta}^*W\otimes {\cal L}
\]
of bundles on $\C/\Lambda$, where $\eta$ is the involution $z\mapsto -q-z$,
$\eta'$ is the involution $z\mapsto -z$, and ${\cal L}$ is the twisting
line bundle.  We do, however, assume $V$ maximal where convenient, since
this is in any event the main case of interest (and it is easy enough to
figure out what goes wrong when maximality fails).

The simplest operation is twisting by $\sO_X(f)$, which simply twists
the bundles $V$ and $W$ by $\sO_{\P^1}(1)$.  This is also easy to extend to
an action on difference equations: the only change is that since $V$ and
$W$ are pulled back through different degree $2$ maps, we must absorb the
difference into ${\cal L}$, thus changing the twisting bundle by the
element of $\Pic^0(C)$ corresponding to $q$.  (In other words, twisting by
$f$ changes $\phi^*(s)$ by $q$.)

\medskip

Although the operation $M\mapsto M(s)$, or equivalently
$M_0\mapsto M_0(s)$, is just as natural in terms of
sheaves, the translation to morphisms of vector bundles on $C$ is quite a
bit more subtle, for the simple reason that twisting does not respect the
resolution we are using.  Now, we can write the original morphism
$B:\pi_\eta^*\to \pi_\eta^*W\otimes \phi^*\sO_X(s)$ in the form
\[
B = B_0(x,w) y_0 + B_1(x,w) y_1
\]
where $x,w$ are homogeneous coordinates on
$\P^1\cong\pi_\eta(C)\cong\rho(X)$, and $y_0$, $y_1$ generate
$\pi_{\eta*}\phi^*\sO_X(s)\cong \rho_*\sO_X(s)$, viewed as a graded module over
$k[x,w]$.  (In the untwisted case, we used $\phi^*\sO_X(s_{\min})$, but this
just differs by twisting by $f$; i.e., we assume that the twisting bundle
has degree 2 when its degree is even.) Now, since $B$ comes from the
standard resolution of $M_0$, we can twist by $\sO_X(s)$ and take direct
images to obtain a short exact sequence
\[
0\to V\xrightarrow{(B_0,B_1)} W\otimes
\rho_*\sO_X(s)\xrightarrow{(B'_1,-B'_0)} \rho_*(M_0(s))\to 0
\]
where $B'_1$, $B'_0$ are suitable morphisms of sheaves on $\P^1$.
The sheaf $\rho_*(M_0(s))$ is torsion-free, since 
otherwise $\Hom(\sO_f(-1),M_0)$ would be nonzero for some fiber $f$,
contradicting maximality of $V$.  In particular, we have
\[
W' \cong \rho_*(M_0(s)).
\]
We also have a commutative diagram
\[
\begin{CD}
0@>>> \rho^*V @>>> \rho^*W\otimes \rho^*\rho_*\sO_X(s)@>>>
\rho^*\rho_*(M_0(s)) @>>>0\\
@. @| @VVV @VVV @.\\
0@>>> \rho^*V @>>> \rho^*W(s)@>>>
M_0(s) @>>>0
\end{CD}
\]
with exact rows; since each sheaf in the bottom sequence is $\rho$-globally
generated, the vertical morphisms are surjective, and have isomorphic
kernels.  We thus find that $V'$ fits into an exact sequence
\[
0\to \rho^*(V')(-s)\to \rho^*W\otimes \rho^*\rho_*\sO_X(s)\to
\rho^* W(s)\to 0.
\]
It follows that
\[
V'\cong \begin{cases} W& s^2=0\\
                      W(-f) & s^2=-1
\end{cases}
\]
We moreover find that (apart from this twisting), the map from $V'$ to $W'$
is simply given by $B'_1y_1+B'_0y_0$.  To avoid the issue with twisting, we
will compute $M(s+f)$ in the odd case.

We now observe that, viewing this as a morphism $B'$ on $C$, we have
\[
B' \eta^*B - \eta^*(B'\eta^*B)
= 
(B'_0B_1-B'_1B_0)(y_0\eta^*y_1-y_1\eta^*y_0)
=
0,
\]
since $B'_0B_1=B'_1B_0$ by construction.  Since $B'\eta^*B$ is
$\eta^*$-invariant, we can use it to modify the factorization of $A$ to
obtain
\[
A = (\eta^*B)^{-t}B^t = (B')^t (\eta^*B')^{-t},
\]
and find that the new $A$ has the form
\[
A' = (\eta^*B')^{-t} (B')^t = (B')^{-t} A (B')^t.
\]
It is somewhat more natural to express the inverse of this operation.

\begin{prop}
  Let $A(z)$ be a twisted elliptic matrix with $\eta^*A = A^{-1}$, twisted
  by a line bundle of degree $\delta$, and let $M$ be the corresponding
  sheaf.  Then the twisted sheaf $M(-s-(2-\delta) f)$
  corresponds to the matrix $B^t A B^{-t}$, where $B$ comes from the
  minimal factorization of $A$.
\end{prop}

This operation need only be modified very slightly (since conjugation
should become a gauge transformation) to make sense for difference
equations: if we start with the system
\[
v(z+q) = B(-q-z)^{-t} B(z)^t v(z)\qquad v(-z)=v(z),
\]
we simply want the equations satisfied by $w(z)=B(z)^t v(z)$, namely
\[
w(-q-z)=w(z)\qquad B(-z)^{-t}w(-z)=B(z)^{-t} w(z).
\]
The only nonobvious point is that this new equation is symmetric with
respect to a slightly different involution; we will see this phenomenon
naturally arising in the noncommutative setting.  (Here twisting by $s$
changes $\phi^*(f)$ by $q$; it also changes $\phi^*(s)$ when $s^2=-1$.  In
general, the rule is that twisting by $D$ changes $\phi^*$ by $(D\cdot
{-})q$; this is the only $W(E_{m+1})$-invariant rule compatible with what
we have so far seen.)  This can also be sidestepped by choosing a point
``$q/2$'' such that $2(q/2)=q$, and replacing the above $w$ by $w(z) =
B(z-q/2)^t v(z-q/2)$.

\medskip

The next operation we consider is duality, as this can also be computed on
the Hirzebruch surface.  Indeed, by \cite[Prop.~7.11]{poisson}, the minimal
lift operation commutes with the canonical duality, so we just need to
understand the dual of $M_0$.  Applying $\dR\sHom(-,\omega_X)$ to the
standard presentation
\[
0\to \rho^*V(-s)\xrightarrow{B} \rho^*W\to M_0\to 0
\]
gives
\[
0\to \sHom(\rho^*W,\omega_X)\xrightarrow{B^t}
\sHom(\rho^*V,\omega_X)(s)\to \sExt^1(M_0,\omega_X)\to 0.
\]
Now, we have
\[
\sHom(\rho^*W,\omega_X) \cong \rho^*\sHom(W,\sO_{\P^1})\otimes \omega_X
\]
and $\omega_X\cong \sO_X(-2s-(4-\delta)f)$, where $\delta\in \{1,2\}$ is
the degree of the twisting bundle.  We thus see that this is a presentation
of the alternate kind we just considered, and can thus determine the
corresponding relaxed difference equation.  We thus find that
$\sExt^1(M,\omega_X)(2f)$ corresponds to the matrix
\[
A' = B (\eta^*B)^{-1} = A^t = \eta^* A^{-t}.
\]
In other words, dual sheaves correspond (up to the involution) to adjoint
difference equations.  More precisely, if we both dualize and act by
$\eta$, we obtain the adjoint equation
\[
w(z+q) = A(z)^{-t} w(z);
\]
the dual sheaf itself corresponds to
\[
w(z-q) = A(z-q)^t w(z),
\]
which is of course precisely the same equation viewed as a $-q$-difference
equation.  Note that the additional $2f$ twist is precisely what we need
in order for the $\chi(M(-f))=0$ condition to be preserved
by duality.

We next turn to twists by $e_i$.  From \cite[Cor.~6.7]{poisson}, we find
that the action of such twists on $M_0$ has the following form.  If
$M'\cong M(-e_i)$, then $M'$ is acyclic for $\pi:X\to X_0$,
and its direct image $M'_0$ fits into a short exact sequence
\[
0\to M'_0\to M_0\to \sO_p^r\to 0
\]
where $p$ is the point of $X_0$ lying under $e_i$, and the morphism $M_0\to
\sO_p^r$ is suitably canonical (with $r = c_1(M)\cdot e_i$).  Local
computations let us describe this morphism in the elliptic case.  First, if
$e_i$ arises from the first time we blow up $p$, then it is just the
canonical morphism
\[
M_0\to \Hom_k(\Hom(M_0,\sO_p),\sO_p).
\]
More generally, the structure of $B$ over the local ring induces a natural
increasing filtration $F_l$ of $\rho^*W\otimes\sO_p^n$, induced by tensor
product from the filtration
\[
F^+_l = \im(u^{1-l}B)\cap \rho^*W
\]
where $u$ is a uniformizer.  If $e_i$ arises from the $l$-th time we blow up
$p$, then the corresponding morphism $M_0\to \sO_p^r$ is induced by the
morphism
\[
\rho^*W\to \rho^*W\otimes\sO_p\cong F_{\infty}\to F_{\infty}/F_l.
\]

Once we have identified the map $M_0\to \sO_p^r$, we then need to
understand how the new $A'$ is related to the original $A$.  We first note
that $M'_0$ is indeed $\rho_*$-acyclic.  Otherwise, there is a nonzero
morphism from $M'_0$ to $\sO_f(-2)$ for some fiber $f$, so that the short
exact sequence defining $M'_0$ pushes forward to an extension
\[
0\to \sO_f(-2)\to F\to \sO_p^r\to 0.
\]
Since $\Hom(M_0,\sO_f(-2))=0$, this is a non-split extension, so $p\in f$;
but then $F\cong \sO_f(-1)\oplus \sO_p^{r-1}$, contradicting the fact that
$\Hom(M_0,\sO_f(-1))=0$ (since $M_0(-s)$ is
$\rho_*$-acyclic).  We also find that $W'=\rho_*M'_0$ is torsion-free, since
$\Hom(\sO_f,M'_0)\subset \Hom(\sO_f,M_0)=0$ for any fiber $f$.

Since $M'_0$, is $\rho_*$-acyclic, we can take the direct image of the
defining extension to obtain a short exact sequence
\[
0\to W'\xrightarrow{D_1} W\to \sO_{\rho(p)}^r\to 0,
\]
for a suitable morphism $D_1$, and since $\rho$ is flat the inverse image is
also exact:
\[
0\to \rho^*W'\to \rho^*W\to \sO_f^r\to 0,
\]
where $f$ is the fiber containing $p$.  We thus obtain a map of short exact
sequences
\[
\begin{CD}
0@>>> \rho^*V(-s) @>B>> \rho^*W @>>> M_0@>>> 0\\
@. @VVV @VVV @VVV @.\\
0@>>> \sO_f(-1)^r @>>> \sO_f^r @>>> \sO_p^r @>>> 0,
\end{CD}
\]
in which the second and third vertical maps are surjective.  The snake
lemma tells us that the natural four-term exact sequence
\[
0\to \rho^*V'(-s)\to \rho^*W'\to M'_0\to \sO_f(-1)^{r'}\to 0
\]
(the cokernel has the form $\sO_f(-1)^{r'}$ since this is the only extent
to which $M'_0$ can fail to be $\rho$-globally generated) is related to a
four-term exact sequence
\[
0\to \rho^*V'(-s)\to \rho^*V(-s)\to \sO_f(-1)^r\to \sO_f(-1)^{r'}\to 0.
\]
This is
the twist of the inverse image of an exact sequence
\[
0\to V'\xrightarrow{D_2} V\to \sO_p^r\to \sO_p^{r'}\to 0.
\]
Of course, if $r'\ne 0$, then $M'_0$ is not globally generated, so we
should really replace $M'_0$ by the generated subsheaf; in this case, the
twisting operation will not be invertible, but a finite amount of such
twisting will suffice to remove any components of $M'_0$ supported on
fibers.  In any event, the new $B$ can be written as
\[
B' = D_1^{-1} B D_2,
\]
and thus, since $D_1$ and $D_2$ are $\eta$-invariant,
\[
A' = D_1^t A D_1^{-t}.
\]
Again, this conjugation should become a gauge transformation: if
$w(z) = D_1(z)^t v(z)$, then $w$ satisfies the equation
\[
w(z+q) = D_1(z+q)^t A(z) D_1(z)^{-t} w(z).
\]
This gauge transformation has the effect of shifting the singularity at $p$
by $q$ (as we expect from how twisting should affect $\phi^*$); in terms of
the invariants $\lambda(B;p)$, it moves the appropriate part to the
partition corresponding to $p-q$.  The case $r'\ne 0$ corresponds to a
situation in which the shifted singularity ends up cancelling an existing
singularity at $p-q$.  (In particular, we see that the sheaves with
components supported on fibers correspond to equations with apparent
singularities, that is to say singularities which can be removed by a
suitable gauge transformation.)

\begin{rem}
  A similar construction (unfortunately called ``elementary
  transformations'') for sheaves on $\P^2$ was given in
  \cite{VinnikovV:1990}.
\end{rem}

We note that since $M\otimes \sO_{C_\alpha}=0$, twisting by
\[
\sO_X(C_\alpha)\cong \sO_X(2s+2f-\sum_i e_i)
\]
has no effect on the sheaf.  Since for difference equations, the various
twisting operations all change various parameters by multiples of $q$, this
cannot quite be true for difference equations; instead, twisting by the
canonical class simply shifts $z$ by $q$.

\begin{rem}
  With the above constructions in mind, we can also easily identify the
  various gauge transformations of \cite{isomonodromy} (called isomonodromy
  transformations there) with twists; in particular, the gauge
  transformations with matrices described by \cite[Thm.~4.6]{isomonodromy}
  are twists by $s+f-\sum_{1\le i\le m+3} e_i$.
\end{rem}

\medskip

It remains to understand how the group $W(E_{m+1})$ acts.  The $S_m$
subgroup is of course easy to understand, as it simply changes the order in
which we blow up the distinct singular points.  Indeed, since it does not
change the final Hirzebruch surface, we should not expect it to have any
effect on the interpretation of the sheaf.

To understand the subgroup $W(D_m)$, it will be enough to understand how
the elementary transformation acts, as it conjugates the two $S_m$
subgroups in $W(D_m)$.  We suppose now that $X_0$ is even, since of course
the odd to even elementary transformation is just the inverse.  And of
course, since the elementary transformation does not change the blowdown
structure past $X_1$, it suffices to consider the direct image $M_1$ of $M$ on
$X_1$.  Let $e$ be the exceptional curve on $X_1$ over $X_0$, and let
$X'_0$ be the transformed Hirzebruch surface.  The nature of the
minimal lift operation implies that we have a short exact sequence of the form
\[
0\to \sO_e(-1)^r\to \pi^*M_0\to M_1\to 0.
\]
If we take the direct image under $\pi'$ (the map that blows down the
complement $e'$ of $e$ in its fiber), then we see that $\sO_e(-1)$ is
acyclic with direct image of the form $\sO_f(-1)$, and thus
\[
0\to \sO_f(-1)^r\to \pi'_*\pi^*M_0\to \pi'_*M_1\to 0
\]
is exact.  This, of course, is precisely the situation we encounter with
non-maximal splittings; in particular, we can compute the new matrix $A'$
equally well from either $\pi'_*M_1$ or $\pi'_*\pi^*M_0$.  Now,
\[
\rho'_*\pi'_*\pi^*M_0 \cong \rho_*\pi_*\pi^*M_0\cong \rho_*M_0,
\]
and thus $W'\cong W$.  Similarly, the nonmaximal bundle $V''$ can be
computed (up to a scalar) by
\begin{align}
\rho'_*(\pi'_*\pi^*M_0(-s'-f))
&=
\rho'_*\pi'_*(\pi^*M_0(e_1-s-f))\notag\\
&=
\rho_*\pi_*(\pi^*M_0(e_1-s-f))
\cong
\rho_*(M_0(-s-f)),\notag
\end{align}
and thus $V''\cong V\otimes \sO_{\P^1}(-1)$.  We furthermore find that
the corresponding map $\rho'_*(B''(s'+f))$ on $\P^1$
factors as
\[
V\otimes \sO_{\P^1}(-1)\xrightarrow{\rho_*(B(s))}
W\otimes \rho_*\sO_{X_0}(s-f)
\xrightarrow{1\otimes \psi}
W\otimes \rho'_*\sO_{X_0}(s')
\]
where $\psi:\rho_*\sO_{X_0}(s-f)\to \rho'_*\sO_{X_0}(s')$ is the image of the
natural map $\sO_{X_1}(s-f)\to \sO_{X_1}(s-e_1)$ on $X_1$.  But then, as a
morphism on $C$, $B''$ is just the composition with the corresponding map
of vector bundles on $C$.  (If we want, we can then compute the true $B'$
by restoring maximality.)

That is, if ${\cal L}_0$ is the original (degree 2) twisting bundle, then
the new twisting bundle is ${\cal L}_1:={\cal L}_0(-p)$
where $p\in C$ is the point corresponding to $e_1$; and $B''=\psi B$ where
$\psi$ is the unique (up to scalars) global section of
\[
\sO_C(-p)\otimes\pi_\eta^*\sO_{\P^1}(1),
\]
essentially a degree 1 theta function vanishing at $\eta(p)$.  We then find
that $A' = \eta^*\psi^{-1}\psi A$.  In other words, the elementary
transformation simply multiplies the shift matrix by a ratio of two degree
1 theta functions, preserving the symmetry.

\medskip

The remaining simple reflection is much more subtle, as can be seen in
particular by the fact that it does not preserve the rank of the equation.
Indeed, if $D=ns+df-\sum_i r_i e_i$ is the original class, then after
reflecting in $s-f$, we obtain an equation of class $D=ds+nf-\sum_i r_i
e_i$.  Since this swaps the order of the equation and a measure of its
degree (relative to ordinary multiplication), this suggests that this
operation should correspond to some sort of generalized Fourier
transformation.  This is in fact the case, and the transform is essentially
that of Spiridonov and Warnaar \cite{SpiridonovVP/WarnaarSO:2006}, but we
postpone the discussion to \cite{noncomm1}, where the transform will
play a crucial role.

\bigskip

As we mentioned above, we were led to consider symmetric elliptic
difference equations by their appearance in two contexts: as the equations
satisfied by elliptic hypergeometric integrals, and as equations related to
elliptic biorthogonal functions (and Painlev\'e theory).  We should
therefore explain how these equations fit into the current framework.

In \cite{dets}, Spiridonov and the author computed the explicit matrix $A$
for the difference equation satisfied by the ``order $m$ elliptic beta
integral''.  For generic parameters, these equations are nonsingular at the
ramification points, and it is thus straightforward to compute their
singularity structure.  The order $m$ elliptic beta integral satisfies an
elliptic equation of order $m+1$ with $2m+4$ ``simple zeros'', i.e., points
where $A$ is holomorphic and $\det(A)$ vanishes once, as well as
two points where $A$ vanishes identically.  This gives a sheaf of Chern class
\[
(m+1)s+(m+2)f-(m+1)e_1-(m+1)e_2-\sum_{3\le i\le 2m+6}e_i
\]
on a blowup of $F_2$ (we can recover the coefficient of $f$ by degree
considerations once we have found all the singularities).  If we perform an
elementary transformation, swap $e_1$ and $e_2$, then again perform an
elementary transformation (i.e., if we reflect in $f-e_1-e_2$), this gives
us a sheaf on a blowup of $F_0$ or $F_2$ with Chern class
\[
(m+1)s+f-\sum_{3\le i\le 2m+6}e_i,
\]
reflecting the fact that the equation given in \cite{dets} had two
singularities introduced precisely in order to make it elliptic rather than
twisted.  Since the first two blowups are independent of the remaining
blowups, we can move those to the end, then ignore them.  Thus the most
natural sheaf-relaxation of this equation has Chern class
\[
(m+1)s+f-\sum_{1\le i\le 2m+4}e_i
\]
on a surface with $K_X^2=4-2m$, relative to an even blowdown structure.
Now, this vector is actually a positive (real) root for $E_{2m+5}$, and
thus (since by construction the sheaf is supported on the complement of the
anticanonical curve) is generically the class of a $-2$-curve.  In
particular, the sheaf, and thus the difference equation, is rigid; this of
course explains why it was even possible to write down the equation
explicitly.

To verify that the vector is a positive root, we can of course apply the
usual algorithm.  The only simple root that has negative intersection with
the class is $s-f$, and thus the first step is the generalized Fourier
transformation mentioned above.  This gives a first-order equation of class
\[
s+(m+1)f-\sum_{1\le i\le 2m+4}e_i,
\]
at which point the action of $D_{2m+4}$ suffices to transform it to the
trivial equation $s-f$ (assuming sufficiently general parameters).  Since
first-order equations have explicit meromorphic solutions given by elliptic
Gamma functions \cite{RuijsenaarsSNM:1999}, we see that the above $m+1$-st
order equation should have a solution expressed as an integral involving
elliptic Gamma functions.  This is, of course, hardly surprising
considering that the equation arose as the equation of an integral, and
indeed, we recover the elliptic beta integral in this way.

Now, suppose one starts with the trivial equation and performs some
sequence of elementary transformations in various points and
Spiridonov-Warnaar transformations.  This will have the effect of replacing
the original $-2$-curve $s-f$ by some image under $W(E_{m+1})$, which will
still be the class of a $-2$-curve, and thus corresponds to a rigid
equation.  (Indeed, for the relaxation, rigidity is a property of the sheaf
on $X$, so is independent of the blowdown structure.)  In terms of
solutions, this starts with $1$, and performs some sequence of the
operations ``multiply by a symmetric product of elliptic Gamma functions''
and ``apply the Spiridonov-Warnaar transform''.  One thus expects that the
result of such a sequence of operations will always satisfy a {\em rigid}
difference equation.  (Despite our identification of the reflection in
$s-f$ with the Spiridonov-Warnaar transform, this is not quite a theorem;
the action on difference equations is purely formal, and relies on an
assumption that there are no extra residue terms coming from certain
required contour shifts.)  Conversely, since every rigid equation comes
from a $-2$-curve, the standard algorithm suggests a way of building up an
integral representation for the solution to any rigid equation.  (This
appears related to the notion of Bailey chains/trees, see
\cite{SpiridonovVP:2002,SpiridonovVP:2004} for the elliptic case.)

\medskip

The other main motivating family of equations are those of
\cite{isomonodromy}, the equations satisfied by certain functions which are
biorthogonal with respect to the order $m$ elliptic beta integral.  These
equations are no longer explicit (though the residues can be expressed as
multivariate integrals of products of elliptic Gamma functions), but it is
still quite feasible to determine their singularities.  These start out
twisted, so we need to make one of the 16 compatible choices of twisting
data, but there is a natural choice making the equation nonsingular at the
ramification points, at least for generic parameters.  (Alternately, as in
the elliptic beta integral case, \cite{isomonodromy} gives a
well-controlled elliptic version of the difference equation, corresponding
to the non-elliptic version by a pair of elementary transformations.)  We
find that we are in the even case, and have a second-order equation with
$2m+6$ simple singularities.  The corresponding Chern class is thus
\[
2s+(m+1)f-\sum_{1\le i\le 2m+6} e_i.
\]
When $m=0$, this is rigid (and indeed all of the multivariate integrals
arising as coefficients can be explicitly evaluated); this is of course one
of the rigid cases we just saw, corresponding to the fact that the order 0
elliptic beta integral admits hypergeometric biorthogonal functions.  (This
generalizes the fact that the Jacobi polynomials, orthogonal with respect
to the usual beta integral, are hypergeometric.)  Otherwise, the class is
in the fundamental chamber, so nef, and is easily checked to be generically
integral.  Since the coefficients of the Chern class are relatively prime,
we find that the corresponding moduli space of sheaves is rational, and
thus the same is true for the moduli space of difference equations (at
least for the components where the moduli spaces are birational; of course,
then $q$-deformed twisting should make this true for the difference
equations in any component).

The case $m=1$ is of particular interest, as in that case the
$2$-dimensional moduli space is an open subvariety of an elliptic surface,
the relative Jacobian of the original $X_8$ (which is isomorphic to $X_8$
given the choice of a section).  More precisely, it is obtained from the
relative Jacobian by removing both the Jacobian of $C_\alpha$ and the
divisor where the bundle fails to be trivial.  As we mentioned, this is
just the theta divisor, so corresponds to the identity section of the
relative Jacobian.  Since a section of a rational elliptic surface is a
$-1$-curve (it is smooth, rational, and meets $C_\alpha$ in a single
point), we could just as well blow down that section before removing it.
This gives us an alternative interpretation of the moduli space as the
complement of $C_\alpha$ in a del Pezzo surface of degree 1: the section
blows down to a point of $C_\alpha$, so there is nothing else to remove.
The $q$-deformed twisting operations discussed above must still act as
rational maps on this del Pezzo surface, and one can check directly that
those actions all factor through blowing up the point of $C_\alpha$
corresponding to $q$ \cite{ABR}.  In other words, the moduli space of
difference equations of this type is precisely the sort of rational surface
studied by \cite{SakaiH:2001}, and the twist operations correspond directly
to the elliptic Painlev\'e equation.  Since the divisor class is
anticanonical, its expansion in the standard basis is invariant under the
action of $W(E_9)$ on blowdown structures, and thus we obtain an action of
$W(E_9)$ on the family of moduli spaces of difference equations.  The
translations (which is how Sakai defined the elliptic Painlev\'e equation
in \cite{SakaiH:2001}) act the same way on the parameters as the twist
operations, and thus must actually act in the same way on the moduli
spaces.  In other words, there is a large abelian subgroup of
$\Z^{10}\rtimes W(E_9)$ that acts trivially on the moduli space; this
almost certainly is special to the $m=1$ case.  In particular, from a
difference equation perspective, the correct way to generalize the elliptic
Painlev\'e equation is clearly the twisting action rather than the Coxeter
group action, as the former is what acts by gauge (i.e., isomonodromy)
transformations in general.  (In fact, the relation between these
interpretations is mediated by a derived equivalence, see
\cite{noncomm1}.)

For general $m$, the fact that $B$ is a morphism between rank 2 bundles and
becomes rank 1 at the singular points (at least generically) allows us to
define a number of rational functions on the moduli space.  A typical
example is the following: when $W$ is trivial, the $1$-dimensional
subspaces $\ell_i:=\im(B(p_i))$ determine eight points in a common
projective line, and thus we can take the cross-ratio of any four of them
to obtain a rational function on the moduli space.  This function has a
particularly nice divisor, and can formally be written as a ratio of tau
functions.  More precisely, if for any $v\in \Z^{2m+8}$, we define
$\tau(v)$ to be the tau function which vanishes when $H^0(M(v-f))\ne 0$, then
\[
\chi(\ell_1,\ell_2,\ell_3,\ell_4)
\propto
\frac{\tau(f-e_1-e_3)\tau(f-e_2-e_4)}
     {\tau(f-e_1-e_4)\tau(f-e_2-e_3)}.
\]
(This is really just a statement about the divisor of the left-hand side.)
This follows from the observation that (for generic parameters)
$\ell_i=\ell_j$ precisely when twisting by $-e_i-e_j$ changes $W$ from
$\sO_{\P^1}^2$ to $\sO_{\P^1}(-2)\oplus \sO_{\P^1}$, an easy consequence of
our description of how twisting affects $W$.  (It is not quite clear how
this should extend to the locus $\tau(0)$ where $W$ itself is not trivial,
but it is clear that whatever order it has along $\tau(0)$ will not depend
on $i$ or $j$.)  If we take into account the behavior when
$\im(B(p_i))=\im(B(p_j))$, this suggests a more precise statement
\[
\chi(\ell_1,\ell_2,\ell_3,\ell_4)
\propto^?
\frac{\theta(e_1-e_3)\tau(f-e_1-e_3)\theta(e_2-e_4)\tau(f-e_2-e_4)}
     {\theta(e_1-e_4)\tau(f-e_1-e_4)\theta(e_2-e_3)\tau(f-e_2-e_3)},
\]
where for any positive root $r$, $\theta(r)$ is the divisor on the moduli
stack of surfaces that vanishes where the root is a $-2$-curve.  This is
very suggestive of the formula for the cross-ratio given in
\cite{isomonodromy}, in which there are multivariate integrals taking the
places of factors
\[
\frac{\tau(f-e_i-e_j)}{\theta(f-e_i-e_j)};
\]
and further suggests that we should have an equation of the form
\begin{align}
&\theta(e_1-e_2)\tau(f-e_1-e_2)\theta(e_3-e_4)\tau(f-e_3-e_4)\notag\\
{}-{}&
\theta(e_1-e_3)\tau(f-e_1-e_3)\theta(e_2-e_4)\tau(f-e_2-e_4)\notag\\
{}+{}&
\theta(e_1-e_4)\tau(f-e_1-e_4)\theta(e_2-e_3)\tau(f-e_2-e_3)=^?0
,\notag
\end{align}
though it is as yet unclear how to make precise sense of such a statement.
(The main issue is producing suitable canonical isomorphisms between tensor
products of bundles of the form $\det \dR\Gamma$, since $\theta$ and $\tau$
are canonical global sections of such bundles.)  Note in particular that
when $m=1$, if we could make sense of the above relation, then it would
induce an entire $W(E_9)$-orbit of relations, which are precisely the
relations satisfied by a tau function for the elliptic Painlev\'e equation,
as given in \cite[Thm.~5.2]{KajiwaraK/MasudaT/NoumiM/OhtaY/YamadaY:2006}.

\begin{rem}
  In \cite{recur}, a four-term variant (related to Pl\"ucker relations
  for pfaffians) of this $W(E_9)$-invariant system of recurrences
  was introduced, corresponding to a slightly different family of
  multivariate hypergeometric integrals.  Is there a corresponding
  model (presumably involving relative Pryms rather than relative
  Jacobians) from a geometric perspective?  This would presumably
  involve a component of the fixed locus of an involution of the
  form $M\mapsto \iota^*\sExt^1(M,\omega_X)$, where $\iota$ is an
  involution on an anticanonical surface of the form $X=X_8$ that
  acts as a hyperelliptic involution on the relevant anticanonical
  curve.  (The latter condition ensures that the combined involution is
  Poisson.)
\end{rem}

\medskip

The case $D=2rs+2rf-\sum_{1\le i\le 8} re_i$ is also likely to have
interesting behavior (assuming it is generically integral, i.e., that
$\sO_{C_\alpha}(C_\alpha)$ has exact order $r$ on $X_8$).  In this
case, $X_8$ is an elliptic surface on which $C_\alpha$ appears as the
underlying curve of an $r$-fold section, but the moduli space is still an
open subset of the relative Jacobian.  As we saw, the moduli space of
matrices $B$ is always rational in the $2$-dimensional case, and the same
reasoning as before tells us that it is an affine del Pezzo surface of
degree 1.  Again, for any $q$, we have an induced action of $\Z^{10}\rtimes
W(E_9)$ as birational maps on this family of del Pezzo surfaces.  Since
Theorem \ref{thm:Picr} tells us the parameters of these del Pezzo surfaces,
we can control the action of the birational maps, and find that the action
factors through a suitable one-point blowup, just as in the case $r=1$.
Note, however, that the lack of a universal family on this moduli space
means that the most obvious way of producing a corresponding Lax pair will
not work.

\section{Degenerations}
\label{sec:degen}

As one might expect, the story becomes more complicated once the
anticanonical curve becomes singular.  The simplest case is when the
anticanonical curve on $X$ is still integral; in that case, the
considerations of the previous section carry over with little change.  The
main constraint is that the symmetric (ordinary and $q$-) difference
equations must have only finite singularities, since blowing up the node or
cusp will introduce a new component to the anticanonical curve.  This can
actually be violated in a mild way: if the difference equation is twisted
by a line bundle, we can use the singularity to single out a global section
of the bundle (modulo scalars), and in this way obtain an untwisted
equation with only a mild singularity at $\infty$.  (E.g., in the
$q$-difference case, the matrix $A$ will no longer be $1$ at $\infty$, but
will still be a multiple of the identity.)  (Equations with more
complicated singularities at $\infty$ might correspond to sheaves that
cannot be separated from the anticanonical curve by a suitable blowup,
though this can always be fixed by a finite number of ``pseudo-twists''
\cite[Lem.~6.8]{poisson}, and we have seen that these correspond to gauge
transformations.)

For more degenerate cases, we can still be guided by what happens in the
elliptic case.  For the $W(E_{m+1})$ action, the $A_{m-1}$ subsystem merely
permutes the singularities, while elementary transformations change the
twisting bundle and multiply the shift matrix by a corresponding
meromorphic section.  The reflection in $s-f$, in contrast, can have a more
significant effect on the nature of the equation.  

We have already considered how the different equations look on $F_2$, and
something similar applies to $F_0$ or $F_1$.  Indeed, since elementary
transformations should not affect the type of equation, but can introduce
or contract fibers which are components of $C_\alpha$, the rule is quite
simple: contract any components of class $f$, and then recognize the curve from
the same list of possibilities as for $F_2$.  (More precisely, choose any
section of the ruling which is transverse to $C_\alpha$, and perform a
sequence of elementary transformations moving that section to the
$-2$-curve of an $F_2$; the resulting anticanonical curve will be disjoint
from $C_\alpha$, and differs from the original only in self-intersections of
components and the contraction of fibers.)  Since this rule treats sections
and fibers differently, it is clear that the result can depend on the
choice of ruling on $F_0$.

Moreover, it is in general not possible to avoid this issue.  One might be
tempted to adapt the algorithms for checking integrality to use only
elements of the group that stabilizes the decomposition of $C_\alpha$, but
this encounters two significant problems: the group need not be a
reflection group, and the reflection subgroup need not have finite rank.
Either possibility denies us any kind of ``fundamental chamber''; there
need not be any computable fundamental domain for the action.  One must
thus use the full algorithm, and this can most certainly change the kind of
equation.

\medskip

As an example of how birational maps can change the type of an equation,
consider the case of a nonsymmetric $q$-difference equation with three
polar singularities.  We recall that such equations (with even twisting; as
mentioned, this allows $A(\infty)$ to be a general multiple of the
identity) correspond to sheaves on $\P^1\times \P^1$ with specified
intersection with a union of two bilinear curves meeting in two distinct
points.  The constraint on singularities means that the sheaf meets the
component of $C_\alpha$ corresponding to poles in the specified three
points (and the restriction is a direct sum of structure sheaves of
subschemes of the corresponding degree $3$ scheme).  Now, a bilinear curve
in $\P^1\times \P^1$ has self-intersection $2$, so after blowing up three
points, the strict transform has self-intersection $-1$.  The assumption on
singularities means that this $-1$-curve $e$ is disjoint from the lift of
$M$, and thus there are blowdown structures for which $c_1(M)$ is in the
fundamental chamber and such that $e_m=e$.  Such a blowdown structure blows
down one of the two components of the anticanonical curve, so produces an
integral anticanonical curve on the eventual Hirzebruch surface.

In other words, there is a birational map taking nonsymmetric
$q$-difference equations with three polar singularities (and regular at $0$
and $\infty$, modulo twisting) to symmetric $q$-difference equations (again
regular at $0$ and $\infty$).  By looking at what the algorithm does to
move $e$ to $e_m$, we see that this involves reflecting in $s-f$ precisely
once.  Before reflecting, the anticanonical curve on $\P^1\times \P^1$ has
two components with classes $s+2f$ and $s$, while after reflecting it has
components of classes $2s+f$ and $f$.

\begin{rem}
Of course, something similar applies if we have more than three polar
singularities or the nonsymmetric equation is singular at $0$ or $\infty$;
the only difference is that the resulting symmetric equation will be
singular at $0$ and $\infty$, possibly in a complicated way.
\end{rem}

On $F_0$, we have a total of 16 possible ways the anticanonical curve can
decompose, each of which corresponds to a different kind of generalized
Fourier transform.  There are 10 such transforms that preserve the type of
equation, and three pairs that change the type.  Of those, one changes
between symmetric and nonsymmetric $q$-difference equations, one is the
ordinary difference analogue, and a final one changes between differential
and ordinary difference equations (a Mellin/z transform).  For the
transforms that preserve type, there are one each for the three integral
types, as well as three transforms on nonsymmetric $q$-difference
equations, two for ordinary difference equations, and two for differential
equations.  Of the latter, the most degenerate is just the Fourier/Laplace
transform, while the other is essentially the transform used in
\cite{KatzNM:1996} (usually called ``middle convolution'' in the later
literature, though this is something of a misnomer).

\medskip

As we mentioned above, we can also model certain Deligne-Simpson problems
via moduli spaces of sheaves on rational surfaces, and this gives rise to
additional interesting maps of moduli spaces.  One, of course, is the
(essentially trivial) observation that moduli spaces of Fuchsian
differential equations correspond to moduli spaces of solutions to additive
Deligne-Simpson problems; in our terms, we can see this by noting that the
anticanonical curves in the latter case become a double $\P^1$ once we
contract all fiber components.  There is another relation, though, which we
consider in the multiplicative case.  Recall that we modeled four-matrix
Deligne-Simpson problems via sheaves with a quadrangular anticanonical
curve in $\P^1\times \P^1$ (with components of class $f$, $f$, $s$, and
$s$).  This is a somewhat cumbersome model for three-matrix problems, as we
need to take one of the four matrices to be the identity.  The
corresponding component of the anticanonical curve contains a single
singular point; if we blow up this point and blow down both the fiber and
the section containing it, we obtain a sheaf on $\P^2$.  If $g_1$, $g_2$,
$g_3$ are the three matrices with product $1$, the sheaf on $\P^2$ is
modeled by the cokernel of the matrix $x+g_1y+g_1g_2z$, and the conjugacy
classes are determined by the restriction to the anticanonical curve
$xyz=0$.  (The additive variant involves sheaves with specified restriction
to $xy(x+y)=0$.)

Much as in the case of a nonsymmetric difference equation with three poles,
a three-matrix Deligne-Simpson problem in which one matrix has a quadratic
minimal polynomial gives rise to a surface in which the anticanonical curve
contains a $-1$-curve disjoint from the relevant sheaf.  In particular, we
obtain a sheaf on an even Hirzebruch surface by blowing up the two roots of
the minimal polynomial, then blowing down the anticanonical component.  On
that Hirzebruch surface, the anticanonical curve has two components, both
of class $s+f$, and we thus obtain a nonsymmetric $q$-difference equation.
(More precisely, we obtain such an equation {\em after} choosing one of the
two rulings; this is tantamount to choosing one of the two roots of the
minimal polynomial.)

This map can be made precise as follows.  Let $g_0$, $g_1$, $g_\infty$ be
a solution to the Deligne-Simpson problem, with
$(g_1-1)(g_1-\beta^{-1})=0$.  (We have rescaled the chosen root of the
minimal polynomial to $1$.)  Define a matrix $A^+(z)$ with rational
coefficients by
\[
A^+(z) = (1-g_\infty^{-1}z)^{-1}(1-g_0z)
       = (g_\infty-z)^{-1}(g_\infty-g_1^{-1}z).
\]
The matrix $A^+(z)$ fixes any vector fixed by $g_1$, and thus has a
well-defined action on the quotient $\im(g_1-1)$.  If $A(z)$ is the matrix
of this action in some basis, then we find $A(0)=1$, $A(\infty)=\beta$, so
that $A$ represents a twisted $q$-difference equation which is regular at
$0$ and $\infty$.  Moreover, we see that the zeros of $A$ occur at the
eigenvalues of $g_0^{-1}$, and the poles of $A$ occur at the eigenvalues of
$g_\infty$, and thus we obtain a rational map between the two moduli
spaces.  (It is unclear how to make the inverse map explicit, though it
certainly exists, due to the description in terms of sheaves.)  Note that
if $g_\infty$ has a cubic minimal polynomial, then we can proceed further,
turning this $q$-difference equation with three polar singularities into a
symmetric $q$-difference equation.  Similarly, a solution to a three-matrix
additive Deligne-Simpson problem with a quadratic minimal polynomial
produces an ordinary difference equation, which can be further transformed
to a symmetric equation if another minimal polynomial is cubic.

\medskip

In \cite{EOR}, several natural multiplicative Deligne-Simpson problems were
considered in which the moduli spaces were shown to be complements of
anticanonical curves in del Pezzo surfaces.  The present approach not only
recovers these results, but gives an alternate intrinsic description of the
del Pezzo surface, making it possible to identify the result explicitly.
There were four problems considered there, one for each of the root systems
of type $D_4$, $E_6$, $E_7$, $E_8$; we consider only the $E_8$ case in
detail.  In that case, the Deligne-Simpson problem is to classify $6l\times
6l$ matrices $g_1$, $g_2$, $g_3$ with $g_1g_2g_3=1$, such that $g_1$,
$g_2$, and $g_3$ have (specified) minimal polynomials of degrees $2$, $3$,
and $6$ respectively.  The corresponding surface blows up $\P^2$ in the
$2+3+6=11$ points corresponding to the roots of the minimal polynomial, and
its anticanonical curve has decomposition
\[
(h-e_1-e_2)+(h-e_3-e_4-e_5)+(h-e_6-e_7-e_8-e_9-e_{10}-e_{11}).
\]
The Chern class of the sheaf has the form
\[
6lh - \sum_i r_i e_i,
\]
where
\[
r_1+r_2=r_3+r_4+r_5=r_6+r_7+r_8+r_9+r_{10}+r_{11}=6l.
\]
The dimension of the corresponding moduli space is determined by the
self-intersection of the divisor, which is maximized when $r_1=r_2$,
$r_3=r_4=r_5$, etc.  This has self-intersection $0$, so apart from some
isolated $-2$-curve cases, is the only interesting case.  Now, $h-e_1-e_2$
is a $-1$-curve, so can be blown down, after which $h-e_3-e_4-e_5$ becomes
a $-1$-curve; after that, we end up on a $9$-point blow up of $\P^2$ with
an integral anticanonical curve, and our divisor class becomes a multiple
of the anticanonical curve.  Thus the solution will be generically
irreducible, i.e., the divisor will be generically integral, precisely when
\[
6h-3e_1-3e_2-2e_3-2e_4-2e_5-e_6-e_7-e_8-e_9-e_{10}-e_{11}
\]
determines a line bundle on $xyz=0$ of exact order $l$.  (In other words,
the corresponding product of zeros of the minimal polynomials must be an
$l$-th root of unity.)  In that case, we find that the relevant surface is
elliptic, with an $l$-tuple fiber (of type $I_1$ in this case), and the
moduli space is an open subset of the relative Jacobian.

Now, just as in the difference equation case, we have a simple numerical
criterion for the sheaves in this open subset to have presentations
involving trivial bundles: again, we want $H^0(M)=H^1(M)=0$, and thus the
moduli space is the complement of a fiber and section on the relative
Jacobian.  The fiber corresponds to the original $l$-tuple fiber, and has
the same Kodaira type (since relative Jacobians preserve Kodaira types of
tame multiple fibers, \cite[Thm.~5.3.1]{CossecFR/DolgachevIV:1989}); since
the anticanonical divisor on that surface was integral (nodal), we see that
the fiber being removed is an integral nodal curve.  Moreover, as in the
elliptic case, we could blow down the section before removing it, and in
this way obtain a del Pezzo surface of degree 1 with a nodal integral fiber
removed.  The problem of identifying this del Pezzo surface reduces to the
problem of identifying the corresponding elliptic surface, and thus to a
special case of Theorem \ref{thm:Picr}.  For $l=1$, this agrees with the
del Pezzo surface for which explicit equations were given in \cite{EOR};
for $l>1$, the conclusion of  Theorem \ref{thm:Picr} settles the conjecture
made there (to wit that the formula for the equation of the moduli space
need only be modified by taking $l$-th powers of the input).

The $E_7$ and $E_6$ cases are similar: $E_7$ has $4l\times 4l$ matrices
with minimal polynomials of degrees $2$, $4$, and $4$, while $E_6$ has
$3l\times 3l$ matrices with cubic minimal polynomials.  In the $E_6$ case,
the surface is already a relatively minimal elliptic surface (with an
$l$-tuple fiber of type $I_3$), while in the $E_7$ case, we must blow down
a component of the anticanonical curve, so end up with an $l$-tuple fiber
of type $I_2$.  In the $E_7$ case, blowing down the tau divisor produces a
$-1$-curve in the anticanonical curve of the del Pezzo surface, so we can
continue by blowing it down, and obtain a degree 2 del Pezzo surface with a
nodal integral anticanonical curve removed.  Similarly, in the $E_6$ case,
we do this twice, and end up with an affine cubic surface with nodal curve
at infinity.  The $D_4$ case is somewhat different, in that it is a
four-matrix Deligne-Simpson problem with quadratic minimal polynomials.  We
end up on a relatively minimal elliptic surface with a multiple $I_4$
fiber, but now have {\em two} tau divisors that need to be removed.  After
blowing down those $-1$-curves, we have a degree 2 del Pezzo surface with a
quadrangle at infinity, with two components of self-intersection $-1$ and
two of self-intersection $-2$.  If we blow down one of the $-1$-curve
components, the result is a triangle of lines on a cubic surface; we could
stop there (the description given for $l=1$ in \cite{OblomkovA:2004}), or
continue to a degree 4 del Pezzo surface with a curve of type $I_2$
removed.  (In this context, we also note that \cite{ABR} used explicit
invariant theory to compute the moduli space of $2\times 2$ matrices with
specified determinant, where the entries are global sections of a degree 4
line bundle on a genus 1 curve; that the result is a del Pezzo surface of
degree 2 follows in the same way, as again one must remove two
tau divisors.)

\medskip

In \cite{Crawley-BoeveyW/ShawP:2006}, multiplicative Deligne-Simpson
problems were related to Coxeter groups, in this case to groups with
arbitrary star-shaped Dynkin diagrams.  (In particular, $E_{m+1}$ has a
star-shaped diagram, and the corresponding Deligne-Simpson problems have a
quadratic and a cubic minimal polynomial.)  At least in the three- and
four-matrix cases, we can see these Coxeter groups from the rational
surface perspective.  In the three-matrix case, these are precisely the
reflection subgroups of the stabilizer of the decomposition of $C_\alpha$,
see \cite{LooijengaE:1981}; the simple roots of the subsystem are
\[
h-e_{11}-e_{21}-e_{31},\qquad\text{and}\qquad e_{ij}-e_{i(j+1)},
\]
where the first subscript on the $e$'s indicates which component of
$C_\alpha$ they intersect.  In the four-matrix case, the description is
slightly more subtle: a four-matrix Deligne-Simpson problem corresponds to
a sheaf which on the Hirzebruch surface has class $n(s+f)$; thus in
addition to stabilizing the decomposition of $C_\alpha$, we also want to
stabilize the root $s-f$.  In addition to the obvious $A_l$-subsystems, we
have a simple root $s+f-e_{11}-e_{21}-e_{31}-e_{41}$.  In any event, even
if we assume that these reflections generate the full stabilizer (rather
than just the subgroup generated by reflections), the algorithms still
become more complicated.  Indeed, we have already seen that there are
several kinds of elliptic pencil on such a surface (even if we exclude
multiple fibers), distinguished by the components that get blown down on
the relatively minimal model; as a result, it is more difficult to identify
whether a class is integral based merely on its image in the fundamental
chamber of this smaller Coxeter group.

\medskip

Of course, as we remarked above, the stabilizer of the anticanonical
decomposition can fail to be a reflection group (and that reflection group
can apparently fail to be of finite rank).  There are a few cases
\cite{LooijengaE:1981} in which the stabilizer is a reflection group and
has been explicitly identified, specifically those with nodal (and thus
polygonal) anticanonical curve, having at most $5$ components.  (Looijenga
also remarks that the $6$ component case is probably feasible; and we have
already seen that the stabilizer can fail to be a reflection group when
there are $7$ or more components.) The $2$ component case, unsurprisingly, has a
star-shaped diagram with one very short leg, corresponding to the relation
between nonsymmetric $q$-difference equations and three-matrix
multiplicative Deligne-Simpson problems in which one of the minimal
polynomials is quadratic.

\medskip

Just as the elliptic hypergeometric equation corresponds to a $-2$-curve,
so is rigid, most other hypergeometric difference/differential equations
can be seen to be rigid in the same way.  As an example, we consider a
maximally degenerate case: the Airy function, which satisfies the
non-Fuchsian differential equation
\[
\Ai''(z) = z\Ai(z),
\]
or in matrix form
\[
v'(z) = \begin{pmatrix} 0 & z \\ 1 & 0\end{pmatrix} v(z).
\]
(Note that there is an ambiguity when passing between straight-line and
matrix forms of an equation: the matrix form is only determined up to a
gauge transformation, so (as long as we can avoid apparent singularities)
the sheaf will only be determined up to ``pseudo-twist''.)  As we have
seen, differential equations correspond to sheaves on $F_2$ with
anticanonical curve of the form $y^2=0$; we find that the above matrix
translates to
\[
\begin{pmatrix}
w^2 & yx\\
y & w^3
\end{pmatrix}
:
\sO_{F_2}(-s_{\min}-2f)\oplus \sO_{F_2}(-s_{\min}-3f) \to \sO_{F_2}^2.
\]
In the $q$-difference case, we noted that we can often absorb particularly
well-behaved singularities into a twist; something similar applies here,
and we should perform an elementary transformation centered at the
subscheme with ideal $(y,w^2)$.  That is, blow up this subscheme, minimally
desingularize, then blow down the original fiber and the $-2$-curve.  The
resulting morphism on $\P^1\times \P^1$ is
\[
\begin{pmatrix}
y_0 & y_1 x_1\\
y_1 & y_0 x_0
\end{pmatrix},
\]
where the new coordinates relate to the original coordinates by
\[
x_1/x_0 = x/w,\qquad y_1/y_0 = y/w^2,
\]
and the new anticanonical curve has equation $y_1^2x_0^2$.  The cokernel
has smooth support, so there is no difficulty in resolving its intersection
with the anticanonical curve.  The support $y_0^2x_0=y_1^2x_1$ meets the
anticanonical curve in a $6$-jet, and each of the first five blowups
introduces a new component to the anticanonical curve, with multiplicities
$3$, $4$, $3$, $2$, and $1$ respectively.  The result is a curve of Kodaira
type $\text{III}^*$, except with one reduced fiber removed.  As for the sheaf
itself, we started with support of class $2s+f$ and blew up 6 points, so
the result is a $-2$-curve as expected.  (Had we started from the sheaf on
$F_2$, we would have obtained a sheaf with first Chern class class
$2s+3f-2e_1-2e_2-e_3-e_4-e_5-e_6-e_7-e_8$, which is naturally also a
positive (real) root.)

In general, a rigid second-order equation (of whatever kind) can always be
transformed by a sequence of elementary transformations into one of class
$2s+f-e_1-e_2-e_3-e_4-e_5-e_6$.  (That is, this is the unique class in the
fundamental chamber with respect to $D_m$ and satisfying $D\cdot f=2$.)  We
can thus describe a moduli space of rigid equations, namely the locally
closed substack of ${\cal X}^{\alpha}_6$ on which this class represents a
$-2$-curve.  Since we can readily rule out the existence of $-d$-curves for
$d>2$, it is a (reasonably small) finite problem to determine the different
types of anticanonical surfaces that arise.  There are a total of $3182$
such types, but the action of $W(D_6)$ reduces this to only $41$
equivalence classes, each of which describes a different kind of
hypergeometric equation.  These range from the elliptic hypergeometric
equation (satisfied by the order 1 elliptic beta integral) down to the two
maximally degenerate cases, the Airy equation and the $q$-difference
equation $v(q^2z)=\beta z v(z)$, passing through examples such as the
differential and difference equations satisfied by the hypergeometric
function of type ${}_2F_1$.

Similarly, in the case $2s+2f-e_1-e_2-e_3-e_4-e_5-e_6-e_7-e_8$,
corresponding in the generic case to the elliptic Painlev\'e equation,
which we saw above corresponded to $22$ orbits of $W(E_8)$ (as in Figure 1
above), there are a total of $139981$ types (we omit the figure!).  The
$W(E_8)$ orbits correspond to different types of Painlev\'e equation, but
the Lax pairs themselves are essentially classified by the $W(D_8)$ orbits
(since as we have seen, the map between equations and sheaves is only
defined up to scalar gauge equivalence in any event).  There are $61$ such
orbits, implying that a given (discrete or continuous) Painlev\'e equation
can have qualitatively different second-order Lax pairs.

For instance, the classification indicates that the Painlev\'e VI equation
should have a Lax pair in the form of a second order symmetric difference
equation.  Indeed, there is a $W(D_8)$-orbit of types of surfaces in which
the anticanonical curve decomposes as
\[
(2s+f-e_1-e_2-e_3-e_4-e_5-e_6)
+(f-e_5-e_6)
+(e_5-e_6)
+2(e_6-e_7)
+(e_7-e_8),
\]
which by our discussion of singularities above corresponds to a symmetric
difference equation with four finite singularities and an indecomposable
singularity at infinity.  The corresponding Lax pair was constructed in
\cite{symmPVI}, along with degenerations to Painlev\'e V and III.

The case $r(2s+2f-e_1-\cdots-e_8)+e_8-e_9$, which corresponds to an
elliptic version of the ``matrix Painlev\'e equation''
\cite{KawakamiH:2015}, has precisely the same classification (apart from
blowing up a point of an anticanonical component of self-intersection
$-1$); there are two apparent additional types for $r=2$, but they have a
component $e_8-e_9$, preventing the linear system from being generically
integral.

The case $2s+3f-e_1-\cdots-e_{10}$, which generically corresponds to the
simplest non-Painlev\'e case of the elliptic Garnier equation
\cite{ellGarnier}, gives rise to a total of $6374578$ types in $84$
$W(D_{10})$-orbits.  Note that since the divisor class is $-K+f$, it is not
preserved by reflection in $s-f$; as a result, each degeneration of the
simplest elliptic Garnier equation corresponds to a unique type of
second-order Lax pair up to scalar gauge equivalence.  (There are, of
course, higher order Lax pairs, e.g., the third-order Lax pair obtained by
reflecting in $s-f$.)

\section{Calogero-Moser spaces}
\label{sec:CMspace}

In \cite[Defn.~5.2.3]{EtingofP/GanWL/OblomkovA:2006}, Etingof, Gan, and
Oblomkov considered a family of four multiplicative Deligne-Simpson
problems associated to representations of a certain family of algebras,
such that the associated moduli space of solutions could be interpreted as
generalized Calogero-Moser spaces.  (To be precise, the moduli space of
solutions of the analogous additive Deligne-Simpson problems can be
identified with the spherical subalgebra of a certain generalized double
affine Hecke algebra, per
\cite{EtingofP/GinzburgV:2002,EtingofP/GanWL/OblomkovA:2006}.)  As above,
these moduli spaces can be identified with moduli spaces of sheaves on
rational surfaces of the form considered in Section \ref{sec:degen}.  We
find that with a suitable choice of blowdown structure, the corresponding
divisor class is $D=-nK_8+e_8-e_9$.  Thus more generally, we should expect
suitable open subsets of moduli spaces $\Irr_X(D,x)$ to play a similar
role.

Now, this same divisor class appeared in Theorem \ref{thm:mod_irr} as one
of the two cases in which a nef divisor disjoint from $C_\alpha$ could fail
to have a generically integral linear system.  In particular, in the
codimension 1 substack of ${\cal X}_9^{\alpha,\le 2}$ such that
$\sO_X(-nK_8+e_8-e_9)_{C_\alpha}\cong \sO_{C_\alpha}$, there is a smaller
substack (of codimension $2$ in ${\cal X}_9^{\alpha,\le 2}$) on which the
general section of the linear system $|{-}nK_8+e_8-e_9|$ is reducible.  We
will see that, at least generically, the corresponding symplectic moduli
space of sheaves can be identified with the Hilbert scheme of $n$ points on
a suitable quasiprojective surface, and thus the full family of moduli
spaces provides flat deformations of such moduli spaces.

Since $e_9\cdot D=1$, we may feel free to consider instead the divisor
class $-nK_8+e_8$ on the surface $X=X_8$, as the remaining point of
intersection is then uniquely determined.  The generically reducible case
has $|-K_8|$ a pencil, and thus $X$ is in this case an elliptic surface
with a section.

It will be convenient to first consider the other generically reducible
case corresponding to elliptic surfaces.

\begin{lem}\label{lem:weiern}
  Let $\psi:X\to C$ be a smooth, relatively minimal, genus 1 fibration with
  no multiple fibers, and let ${\cal M}^{(n)}_X(0)$ be the moduli space
  classifying stable sheaves $M$ of Euler characteristic 0 and with
  $c_1(M)$ a sum of $n$ fibers of $\psi$.  Then ${\cal M}^{(n)}_X(0)$ is naturally
  isomorphic to $\Sym^n(W)$, where $W$ is the Weierstrass model of the
  relative Jacobian of $X$.
\end{lem}

\begin{proof}
  Any sheaf $M$ classified by ${\cal M}^{(n)}_X(0)$ is S-equivalent to a
  sum of stable sheaves.  Since a stable sheaf has connected support, we
  find that each summand is set-theoretically supported on a single fiber,
  allowing us to apply Lemma \ref{lem:canon_stable}.  We thus find that $M$
  is S-equivalent to a sum of (a) invertible sheaves on fibers, (b)
  torsion-free but not integral sheaves on integral fibers, and (c) images
  of $\sO_{\P^1}(-1)$ under maps to reducible fibers.  The first two cases
  have the same first Chern class as a fiber, while the multiplicities of
  the summands of the third type are uniquely determined by their
  contribution to the first Chern class.  In particular, we find that $M$
  is S-equivalent to a sum of $n$ semistable sheaves, each of which has
  $c_1$ the class of a single fiber; the summands are further uniquely
  determined up to S-equivalence.  That ${\cal M}^{(n)}_X(0)$ is the
  Weierstrass model of the relative Jacobian then follows from Corollary
  \ref{cor:weier1}.
\end{proof}

\begin{thm}\label{thm:CMspace}
  Let $\psi:X\to \P^1$ be a smooth, relatively minimal, genus 1
  fibration with no multiple fibers and a section $s$, and let $f$ denote
  the class of a fiber of $\psi$.  Then the moduli space ${\cal
    M}^{(n)}_X(1)$ classifying stable sheaves $M$ of Euler characteristic 1
  with $c_1(M)=nf+s$ is a locally symplectic resolution of $\Sym^n(W)$,
  where $W$ is the Weierstrass model of $X$.
\end{thm}

\begin{proof}
  A divisor $D$ linearly equivalent to $nf+s$ meets the generic fiber
  transversely, and in the same point as $s$ (since $\sO_X(D)$ meets the
  generic fiber in the same line bundle as $\sO_X(s)$).  It in
  particular follows that $s$ is a component of $D$ (and the unique
  horizontal component).  In particular, if $M$ is any sheaf corresponding
  to a point of ${\cal M}^{(n)}_X(1)$, the the support of $M$ contains $s$.
  
  More generally, let $M$ be a stable sheaf of Euler characteristic 1 such
  that the 0-th Fitting scheme is a sum of $s$ and a nonnegative linear
  combination of components of fibers.  We claim that $M$ admits a
  surjective morphism to $\sO_s$.  If $M$ is supported on $s$, this is
  immediate.  Otherwise, let $C$ be a fiber containing some vertical
  component of the support of $M$, and consider the image $M'$ of the
  natural morphism $M\to M\otimes \sO_X(C)$.  Now, $M'$ is a (proper)
  nonzero quotient of $M$, so $\chi(M')>0$, while $M'\otimes \sO_X(-C)$
  is a proper subsheaf of $M$, so $\chi(M'\otimes \sO_X(-C))\le 0$.
  Since $c_1(M)-c_1(M')$ is vertical, we find that $c_1(M')\cdot
  C=c_1(M)\cdot C=1$, and thus $\chi(M')=\chi(M'\otimes \sO_X(-C))+1$,
  implying $\chi(M')=1$.  Since any quotient of $M'$ is a quotient of $M$,
  we see that $M'$ is again stable, so by induction admits a surjection to
  $\sO_s$, giving the required surjection $M\to \sO_s$.

  Now, let $M$ be the sheaf corresponding to a point of ${\cal
    M}^{(n)}_X(1)$.  Then $M$ certainly has a surjection to $\sO_s$, and
  since the kernel has no map to $\sO_s$, this surjection is unique.  Thus
  $M$ uniquely determines the kernel $M_v$ of this surjection, a sheaf of
  Euler characteristic 0 with $c_1(M_v)=nf$ such that any subsheaf (being a
  proper subsheaf of $M$) has nonpositive Euler characteristic.  In
  particular, $M_v$ is semistable, so that we obtain a morphism from ${\cal
    M}^{(n)}_X(1)$ to the corresponding semistable moduli space, which by
  Lemma \ref{lem:weiern} is isomorphic to $\Sym^n(W)$.

  It remains only to show that this map is locally (on the base) a symplectic
  resolution.  If $X$ is symplectic, then ${\cal M}^{(n)}_X(1)$ is
  symplectic, and this is immediate.  More generally, if $X$ is Poisson,
  then so is ${\cal M}^{(n)}_X(1)$, and moreover ${\cal M}^{(n)}_X(1)$ is
  generically symplectic; thus the symplectic locus is a symplectic
  resolution of the symmetric power of the Weierstrass model of the
  symplectic locus of $X$.  By varying the Poisson structure, we can cover
  $\Sym^n(W)$ by open subvarieties over which ${\cal M}^{(n)}_X(1)$ is a
  symplectic resolution.

  The remaining case is that $-K_X$ is ineffective, but then $K_X$ is a
  positive sum of fibers, making $K_X$ effective.  Any nonzero section
  $\beta$ of $\omega_X$ induces a (closed) 2-form on ${\cal M}^{(n)}_X(1)$
  (see, e.g., \cite[\S 10]{HuybrechtsD/LehnM:1997}).  Moreover, we claim
  that this 2-form is nondegenerate whenever the corresponding sheaf is
  transverse to the zero locus of $\beta$; again, by varying $\beta$, this
  will suffice to make ${\cal M}^{(n)}_X(1)\to \Sym^n(W)$ a symplectic
  resolution locally on the base.

  We need to show that the induced (and self-dual) morphism $\Ext^1(M,M)\to
  \Ext^1(M,M\otimes \omega_X)$ is an isomorphism, or equivalently that the
  connecting map $\Hom(M,M\otimes \coker(\beta))\to \Ext^1(M,M)$ is 0.
  Now, $M\otimes \coker(\beta)$ is a $0$-dimensional quotient of $\sO_s$ of
  degree $d-2$, so that
\[
\dim\Hom(M,M\otimes\omega_X)\le d-1,
\]
with equality iff the connecting map vanishes.  Since $\omega_X$ has this
many global sections, none of which vanish on the support of $M$, the claim
follows.
\end{proof}

Over a field of characteristic 0, the main result of
\cite{FuB/NamikawaY:2004} shows that for any smooth surface $X$, the
natural morphism $\Hilb^n(X)\to \Sym^n(X)$ from the Hilbert scheme of
points is the unique crepant resolution of $\Sym^n(X)$.  In particular,
when $W$ is smooth (and characteristic 0), then ${\cal M}^{(n)}_X(1)\cong
\Hilb^n(W)$.  When $W$ is not smooth, the most natural choice of crepant
resolution of $\Sym^n(W)$ is the Hilbert scheme of its minimal
desingularization $X$.  However, it turns out
\cite{perverseHilbert} that this is not the crepant resolution given by
${\cal M}^{(n)}_X(1)$, which is in fact given by the ``perverse'' Hilbert
scheme of $W$.

\medskip
In the rational case, we may more generally consider the moduli spaces
${\cal M}_X(-nK_8+e_8,1)$ classifying semistable 1-dimensional sheaves of
Euler characteristic 1 and first Chern class $-nK_8+e_8$, which we can view
as a family of deformations of (perverse) Hilbert schemes of points on
(Jacobian) rational elliptic surfaces.  Replacing $X$ by a noncommutative
deformation should give an additional parameter in the family of moduli
spaces, which should afford a deformation of the (perverse, presumably)
Hilbert scheme of points on a general rational surface with $K_X^2=0$.
There is also a notion of a Hilbert scheme of a noncommutative surface
\cite{NevinsTA/StaffordJT:2007,noncomm1}, presumably related to this case
by a derived equivalence.  Another flat deformation of the symmetric power
of a general rational surface will be constructed in \cite{elldaha} as a
special case of a noncommutative deformation of such symmetric powers.

We can also obtain deformations of (perverse) Hilbert schemes of affine del
Pezzo surfaces (i.e., the complement of an anticanonical curve on such a
surface).  Indeed, for $X$ elliptic such that $X\setminus C_\alpha$
contains no $-2$-curve, we have already seen that the symplectic leaf of
${\cal M}_X(-nK_8+e_8,1)$ is a symplectic resolution of $\Sym^n(X\setminus
C_\alpha)$, so in general we obtain a flat family of symplectic varieties
having an additional parameter.  To obtain a corresponding deformation for
del Pezzo surfaces, we need simply remove a suitable combination of tau
divisors.  (In particular, this allows us to identify the Deligne-Simpson
moduli spaces considered above with symplectic deformations of symmetric
powers of the obvious del Pezzo surfaces.)

Similar considerations apply in the case $X=E\times \P^1$, as again the
relevant divisor class remains effective and generically disjoint from the
anticanonical curve for a larger family of moduli problems.\footnote{There
  is one more case with this property: for each $g\ge 2$, the moduli stack
  of K3 surfaces of genus $g$ contains (as a codimension 1 substack) the
  moduli stack of Jacobian elliptic K3 surfaces polarized by $gs+f$.} These
moduli problems have been considered before, see \cite[\S
  10.1]{HurtubiseJC/MarkmanE:2002b}.  In particular, when $X$ is the
$\P^1$-bundle associated to an indecomposable bundle of rank 2 and trivial
determinant, the corresponding moduli space is the ambient space of the
elliptic Calogero-Moser system of type $\GL_n$.

With this in mind, we refer to the action of the relevant lattice on the
above rational moduli spaces as (degenerations of) the ``symmetric elliptic
difference Calogero-Moser system''.  (Note that the usual action of
$\Lambda_{E_8}$ is extended by an additional twist, since we have an
additional point to blow up.)  It would be particularly interesting to
understand the noncommutative deformation of this action, as it would give
a $1$-parameter deformation of the $n$-th symmetric power of the elliptic
Painlev\'e equation (and of the other equations in Sakai's hierarchy), the
differential level of which has recently been constructed
\cite{KawakamiH:2015}.

\bibliographystyle{plain}

\end{document}